\documentclass[a4paper,11pt,draft]{amsart}
\usepackage[margin=22mm]{geometry}
\usepackage{amssymb, amsmath, amsthm, amscd}

\usepackage[T1]{fontenc}
\usepackage{eucal,mathrsfs,dsfont}
\usepackage{color}

\renewcommand{\le}{\leqslant}
\renewcommand{\ge}{\geqslant}

\definecolor{mno}{rgb}{0.5,0.1,0.5}

\newcommand{\R}{\mathds R}

\newcommand{\w}{\omega}

\newcommand{\Pp}{\mathds P}
\newcommand{\Ee}{\mathds E}

\newcommand{\I}{\mathds 1}
\newcommand{\N}{\mathds{N}}
\newcommand{\Z}{\mathds Z}

\newcommand{\D}{\mathscr{D}}

\newcommand{\F}{\mathscr{F}}

\newtheorem{theorem}{Theorem}[section]
\newtheorem{lemma}[theorem]{Lemma}
\newtheorem{proposition}[theorem]{Proposition}

\theoremstyle{definition}

\newtheorem{remark}[theorem]{Remark}

\numberwithin{equation}{section}

\begin{document}
\allowdisplaybreaks
\title[Random Conductance Models with Stable-like Jumps]
{\bfseries Random Conductance Models with Stable-like Jumps:
Quenched Invariance Principle}
\author{Xin Chen\qquad Takashi Kumagai\qquad Jian Wang}
\thanks{\emph{X.\ Chen:}
   Department of Mathematics, Shanghai Jiao Tong University, 200240 Shanghai, P.R. China. \texttt{chenxin217@sjtu.edu.cn}}
\thanks{\emph{T.\ Kumagai:}
 Research Institute for Mathematical Sciences,
Kyoto University, Kyoto 606-8502, Japan.
\texttt{kumagai@kurims.kyoto-u.ac.jp}}
  \thanks{\emph{J.\ Wang:}
    College of Mathematics and Informatics \& Fujian Key Laboratory of Mathematical Analysis and Applications, Fujian Normal University, 350007 Fuzhou, P.R. China. \texttt{jianwang@fjnu.edu.cn}}

\date{}

\maketitle

\begin{abstract} We study the quenched invariance principle for random conductance models with long range jumps on $\Z^d$, where
the transition probability from $x$ to $y$ is, on average, comparable to
$|x-y|^{-(d+\alpha)}$ with $\alpha\in (0,2)$
but is allowed to be degenerate.
Under some moment conditions
on the conductance, we prove that the scaling limit of
the Markov process is  a symmetric
$\alpha$-stable L\'evy process on $\R^d$.
The well-known corrector method in homogenization theory does not seem to work
in this setting. Instead, we utilize probabilistic potential theory for the corresponding jump processes. Two essential ingredients of our proof are the tightness estimate and the H\"{o}lder regularity of caloric
functions for non-elliptic $\alpha$-stable-like processes on graphs. Our method is robust enough to apply not only for $\Z^d$ but also for more
general graphs whose scaling limits are nice metric measure spaces.

\medskip

\medskip

\end{abstract}
\allowdisplaybreaks

\section{Introduction and Main Results}\label{section1}
Over the last decade, significant progress has been made concerning the quenched invariance principle on random conductance models. A typical and important
example is random walk on the infinite cluster of supercritical
bond percolation on $\mathbb Z^d$. It is shown that the scaling limit of the random walk
is a (constant time change of) Brownian motion on $\mathbb R^d$ in the quenched sense,
namely almost surely with respect to the randomness of the media.
See \cite{ABDH,BD,BeB,BP,CCK,M,MatP,SS} for related progress on this subject and \cite{Bs,Kum} for overall
introduction on this area and related topics.
Besides i.i.d.\
nearest neighbor
random conductance models, recently there have been great developments on the
scaling limit of short range random conductance models on stationary ergodic media
(or the media with suitable correlation conditions),
see  \cite{ADS1,ADS2,ACDS,BsRod,DNS,PRS} for more details. Here, short range means only
finite number of conductances are directly connected to each vertex.

Unlike the short range case, there are only a few results concerning quenched invariance principle for long range random conductance models due to their fundamental technical  difficulties. There is a beautiful paper by Crawford and Sly \cite{CS1} that obtains
the quenched invariance principle for random walk on the long range percolation cluster to
an isotropic $\alpha$-stable L\'evy process in the range $0<\alpha<1$. While \cite{CS1}
proves the invariance principle for a very singular object like the long range percolation, the arguments heavily rely on the special
properties
(see for instance \cite{Be,Bs1,CS0} for related discussions) of the long range percolation and cannot be easily
generalized to the setting of general (long range) random conductance models.

In this paper, we will discuss the quenched invariance principle on long range random conductance models.
In particular, we consider the case where the conductance between $x$ and $y$ is, on average, comparable to $|x-y|^{-(d+\alpha)}$ with $\alpha\in (0,2)$
but is allowed to be degenerate.
In this setting,
there is a significant difficulty in applying classical techniques of homogenization for nearest neighbor  random walk (in
random environment) due to the existence of long range conductances.
 To emphasize the novelty of our paper, we first make some remarks.
 Some more details and technical difficulties of our methods are further discussed at the end of the introduction.
\begin{itemize}\item [(i)] The well known harmonic decomposition method (also called  the corrector method in the literature) has been widely used
for the nearest neighbor random walk in random media,
 see \cite{ABDH,ADS1,ADS2,ACDS,BD,BeB,BsRod,SS}.
Because of the lack of
 $L^2$ integrability, such method does not work (at least in a straightforward way) for our long range model here.
\item [(ii)] Due to singularity in the infinite cluster of long range percolation,
 \cite{CS1} established the quenched invariance principle of the associated random walk in the sense of
 weak convergence on $L^q$
 (not the Skorohod topology)  and only for the case $0<\alpha<1$. In the present paper, we can justify
 quenched invariance principle of our model under the Skorohod topology for all $\alpha\in (0,2)$.
(To be fair, the long range percolation is \lq\lq more singular\rq\rq, and it is not
included in our conductance model.)
Moreover, compared with \cite{CKK}, we can prove the quenched invariance principle for the process with fixed
initial point, see e.g. Remark \ref{r4-6} below.
 \item[(iii)] Our approach is to utilize recently developed
de Giorgi-Nash-Moser theory for jump processes (see for instance \cite{BBK,CK,CK08,CKW1}).
While detailed heat kernel estimates and Harnack inequalities are
established for uniformly elliptic $\alpha$-stable-like processes,
the arguments rely on pointwise estimates of
the jumping density (conductance in this setting), which cannot hold in our setting unless we assume uniform ellipticity of
conductance. Furthermore, as will be shown in the accompanied paper \cite{CKWan}, Harnack inequalities do not hold (even for large enough balls) in general on long range random conductance models. For these reasons, highly non-trivial modifications are required to work on the present random conductance setting.
 Roughly speaking, in this paper we are concerned
 with the long range conductance model
 with large scale summability conditions on the conductance, which can be viewed as a counterpart of the so-called \lq\lq good ball
 condition\rq\rq\,\,in \cite{B,BC} to the non-local setting.
 We believe that our methods are rather robust and
could be fundamental tools in exploring scaling limits of random walks on long range random media.
\item[(iv)] The advantage of our methods is that they do not use translation invariance of the original graph
(we do not use the idea of \lq\lq the environment viewed from the particle\rq\rq); hence
they are applicable not only for $\Z^d$ but also for more
general graphs whose scaling limits are nice metric measure spaces.
Even in the setting of $\Z^d$,
our results can be applied to
the case that the conductance is independent but
possibly degenerate and
not necessarily identically distributed; that is, our results are efficient for some long range random walks on
degenerate non-ergodic media.
The disadvantage
is, since we use the Borel-Cantelli lemma to deduce quenched estimates,
the arguments require \lq\lq strong mixing properties\rq\rq\, of the random conductance
(see \eqref{p3-2-1}--\eqref{l3-1-1-1} below). Hence our method cannot be generalized to
general stationary ergodic case on $\Z^d$.
\end{itemize}

To illustrate our contribution,
we present the statement
about the quenched invariance principle on a half/quarter space
$F:= \mathbb{R}^{d_1}_+\times\mathbb{R}^{d_2}$ where
$d_1,d_2\in \mathbb{N}\cup\{0\}$.
(This is the simplest example of state spaces that is not translation invariant when $d_1\ne 0$.)
The readers may refer to Sections \ref{section3} and \ref{section5} for general results.
Let $\mathbb{L}:=\mathbb{Z}^{d_1}_+\times\mathbb{Z}^{d_2}$.
Consider a Markov generator
\begin{equation}\label{eq:geneoe}
L^\omega_{\mathbb{L}} f(x)=\sum_{y\in \mathbb{L}}(f(y)-f(x))
\frac{w_{x,y}(\w)}{|x-y|^{d+\alpha}},
\quad x\in \mathbb{L},
\end{equation}
where $d=d_1+d_2$, $\alpha\in (0,2)$ and $\{w_{x,y}(\w):x,y\in \mathbb{L}\}$ is a sequence of random variables
such that $w_{x,y}(\w)=w_{y,x}(\w)\ge0$
for all $x\neq y$. We use the convention that
$w_{x,x}(\w)=w_{x,x}^{-1}(\w)=0$ for all $x \in \mathbb{L}$.
Let $(X^{\w}_t)_{t\ge 0}$ be the corresponding Markov process.
For every $n\ge1$ and $\w\in \Omega$, we define a process
$X_{\cdot}^{(n),\w}$ on $V_n=n^{-1}\mathbb{L}$ by
$X_{t}^{(n),\w}:={n}^{-1}X_{n^{\alpha}t}^{\w}$ for any $t\ge0$. Let
$\Pp_{x}^{(n),\w}$ be the law of $X_{\cdot}^{(n),\w}$ with initial
point $x\in V_n$. Let $Y:=((Y_t)_{t\ge0},
(\Pp_{x}^Y)_{x\in F})$ be a $F$-valued strong Markov process.
 We say that the quenched invariance principle holds for
$X_{\cdot}^{\w}$ with limit process being $Y$, if for any $\{x_n \in
V_n:n\ge1\}$ such that $\lim_{n \rightarrow \infty}x_n=x$ for
some $x \in F$, it holds that for $\Pp$-a.s.\ $\w\in \Omega$ and
every $T>0$, $\Pp_{x_n}^{(n),\w}$ converges weakly to $\Pp_{x}^Y$ on
the space of all probability measures on $\D([0,T];F)$,  the collection
of c\`{a}dl\`{a}g $F$-valued functions on $[0,T]$ equipped with
the Skorohod topology.

\begin{theorem}\label{th1} Let $d>4-2\alpha$, and $E=\{(x,y): \ x,y\in \mathbb{L}\}$ be the collection of all
unordered pairs
on $\mathbb{L}$.
Suppose that $\{w_{x,y}: (x,y)\in E\}$ is a sequence of non-negative
independent random variables such that $\Ee w_{x,y}=1$ for all
$x,y\in \mathbb{L}$,
\begin{equation}\label{eq: prob}
\sup_{x,y \in \mathbb{L},x\neq y}\Pp\big(w_{x,y}=0\big)<1/2
\end{equation} and
\begin{equation}\label{eq:fhibw}
\sup_{x,y\in \mathbb{L}}\Ee[w_{x,y}^{2p}]<\infty,\quad \sup_{x,y\in \mathbb{L}}\Ee[w_{x,y}^{-2q}\I_{\{w_{x,y}>0\}}]<\infty
\end{equation}
with
\begin{equation}\label{eq:fhibw22}
p>\max\big\{{(d+2)}/{d}, {(d+1)}/(2(2-\alpha))\big\},\quad q>{(d+2)}/{d}.
\end{equation}
Then the quenched invariance principle holds for
 $X^{\w}_{\cdot}$ with the limit process being a reflected symmetric $\alpha$-stable process
 $Y$ on $F$ with jumping measure $|z|^{-d-\alpha}\,dz$. \end{theorem}

\begin{remark}
\begin{itemize}
\item[(i)]
When $\alpha\in (0,1)$, the conclusion of Theorem \ref{th1} still holds true for $d>2-2\alpha$, if $p>\max\big\{{(d+2)}/{d},{(d+1)}/(2(1-\alpha))\big\}$ and $q>{(d+2)}/{d}.$
See Proposition \ref{ex3-1} for details.

\item[(ii)]
The probability
$1/2$ in \eqref{eq: prob} is far from optimal.
In fact, the critical probability $p_c$ allowing the conductances to be degenerate heavily depends on large scale properties
of the long-range percolation cluster associated with the random conductance model in our paper, which are
different from these of the nearest neighbor percolation cluster (see e.g. \cite{ABDH,B,BeB,CCK,M,SS}) or these of the long-range percolation cluster
investigated in \cite{Be,Bs1,CS0,CS1}. We do not know the exact value of $p_c$, and even whether $p_c=1$ or $p_c<1$.

\item[(iii)] We note that the integrability condition \eqref{eq:fhibw22} is far from
optimal too,  and we also do not even know what could be the optimal integrability condition.
Furthermore, by tracking the proofs of Lemma \ref{nl2-1} and Proposition \ref{ex3-1} below, the negative moment condition $\sup_{x,y\in \mathbb{L}}\Ee[w_{x,y}^{-2q}\I_{\{w_{x,y}>0\}}]<\infty$
in \eqref{eq:fhibw} is only required
to guarantee the local Poincar\'e inequality \eqref{nl2-1-2a}. For the i.i.d nearest neighbor percolation cluster, such kind of negative
moment condition can be removed by the domination behavior of the percolation cluster and time change arguments, see e.g. \cite{ABDH,M}.
However, as we mentioned above, since properties of the nearest neighbor percolation cluster are quite different from those of long-range percolation cluster
associated with the random conductance model in our paper, the arguments in \cite{ABDH,M} do not work for the present setting. By now we do not know whether this negative moment condition is essentially necessary.
\end{itemize}
\end{remark}
Here is one simple example that satisfies
\eqref{eq: prob} and \eqref{eq:fhibw}:
for each distinct $x,y\in \mathbb{L}$,
\begin{eqnarray*}
&\Pp(w_{x,y}=|x-y|^\varepsilon)=(3|x-y|^{2p\varepsilon})^{-1},\quad
\Pp(w_{x,y}=|x-y|^{-\delta})=(3|x-y|^{2q\delta})^{-1},\\
&\Pp\big(w_{x,y}=0\big)=
3^{-1},\quad \Pp(w_{x,y}=g(x,y))=1-(3|x-y|^{2p\varepsilon})^{-1}-(3|x-y|^{2q\delta})^{-1}-3^{-1},
\end{eqnarray*}
where $\varepsilon, \delta, p,q>0$ so that \eqref{eq:fhibw22} is satisfied, and
$g(x,y)$ are chosen so that $\Ee w_{x,y}=1$. (It is easy to see that
$c^{-1}\le g(x,y)\le c$ for some constant $c\ge1$.)

\medskip

At
the end of the introduction, let us briefly discuss technical difficulties and the ideas of the proof.  There
 are two essential ingredients in our proof; namely the tightness estimate and the H\"{o}lder regularity of caloric
 functions for non-elliptic $\alpha$-stable-like processes on graphs. In order to obtain the former estimate, we first split small jumps and big jumps, which is a standard approach for jump processes, and then change the conductance
 to the averaged one outside a ball (we call it
 the localization method). By this technique and the on-diagonal heat kernel upper bound (Proposition \ref{np2-1}),
 we can apply the so-called Bass-Nash method to control the mean
 displacement of the process (Proposition \ref{np-1}). The tightness estimate
 (Theorem \ref{exit}) is established by comparing
 the original process, truncated process and the localized process.
We note that when $0<\alpha<1$, tightness can be proved in a much simpler way
using martingale arguments (Proposition \ref{L:tight}).
The key ingredient for the H\"{o}lder regularity of caloric
functions (Theorem \ref{T:holder}) is to deduce the Krylov-type estimate
(Proposition \ref{Kr}) that controls the hitting probability to a large set
before exiting some parabolic cylinder.
Once these estimates are established, we use the arguments in \cite{CKK}
to deduce generalized Mosco convergence, and then obtain the
weak convergence (Theorem \ref{t3-1}).

\section{Truncated $\alpha$-stable-like processes on graphs}\label{S:tr}
In this and the next sections, we fix graphs and discuss
$\alpha$-stable-like processes on them. Hence we do not consider
randomness of the environment. With a slight abuse of notation, we
still use $w_{x,y}$ as the deterministic version. Let $G=(V,E_V)$ be
a locally finite and connected graph, where $V$ is the set of
vertices, and $E_V$ the set of edges.
 For any $x\neq y \in V$, we write $\rho(x,y)$ for the graph distance, i.e., $\rho(x,y)$ is the length of the shortest path (that is, a sequence $x_0=x, x_1,\cdots,x_l=y$ such that $(x_i,x_{i+1})\in E_V$ for all $0\le i\le l-1$) joining $x$ and $y$.
  Set $\rho(x,x)=0$ for all $x\in V$.
  We let $B(x,r)=\{y\in V:\rho(x,y)\le  r\}$
denote the ball in the graph metric with center $x\in V$ and radius $r>0$. Let $\mu$ be a measure on $V$ such that $\mu_x:=\mu(\{x\})$ satisfies for some constant $c_M\ge1$ that
\begin{equation}\label{al2-0}
c_M^{-1}\le \mu_x\le c_M,\quad x\in V.
\end{equation}
For each $p\in[1,\infty)$, let $L^p(V;\mu)=\{f\in \R^V:\sum_{x\in V}|f(x)|^p\mu_x<\infty\}$, and
 denote by $\|f\|_p$ the $L^p$ norm of $f$ with respect to $\mu$. Let $L^\infty(V;\mu)$ be the space of bounded measurable functions on $V$, and let $\|f\|_\infty$ be the $L^\infty$ norm of $f$.
We assume that $(G,\mu)$ satisfies the $d$-set condition with $d>0$, i.e., there exist $r_{G}\in [1,\infty]$ and $c_G\ge1$ such that
\begin{equation}\label{al2-1}
c_G^{-1}r^d\le \mu(B(x,r))\le c_Gr^d,\quad x\in V,1\le r<r_G.
\end{equation}
For example, if $V$ is a subset of $\mathbb{Z}^d$, then we can take $r_G=1+{\rm diam}{V}$,
where ${\rm diam}V$ denotes the
diameter of $V$. In particular, when $V$ is bounded, $r_G<+\infty$
-- see Section \ref{bddLipd}.
Throughout the paper, all the constants appearing
in the statements of lemmas, propositions and theorems are independent of $r_G$.
We will also not stress the dependence on $c_G$ and $c_M$ above for these constants.

We consider the operator
$Lf(x)=\sum_{z\in V}(f(z)-f(x))\frac{w_{x,z}}{\rho(x,z)^{d+\alpha}}\mu_z$ and the quadratic form
\begin{align*}
D(f,f)&=\frac{1}{2}\sum_{x,y\in V}(f(x)-f(y))^2\frac{w_{x,y}}{\rho(x,y)^{d+\alpha}}\mu_x\mu_y,\quad f\in \F=\{f\in L^2(V;\mu): D(f,f)<\infty\},
\end{align*}
where $\alpha\in (0,2)$ and $\{w_{x,y}:x,y\in V\}$ is a sequence
such that $w_{x,x}=0$ for all $x\in V$,
$w_{x,y}\ge0$ and $w_{x,y}=w_{y,x}$ for all $x\neq y$, and
\begin{equation}\label{eq:condwxy}
\sum_{y \in V}\frac{w_{x,y}}
{\rho(x,y)^{d+\alpha}}
\mu_y<\infty,\quad x\in V.
\end{equation}
 Here by convention we set $0/0=0$.
According to (the first statement in) \cite[Theorem 3.2]{CKK},
$(D,\F)$ is a regular symmetric Dirichlet form on $L^2(V;\mu)$. Let
$X:=(X_t)_{t\ge0}$ be the symmetric Hunt process associated with
$(D,\F)$. Set $C_{x,y}:=w_{x,y}/\rho(x,y)^{d+\alpha}$. Under
$\Pp^x$, $X_0=x$, and the process $X$ waits for an exponentially distributed random
time of parameter $C_{x}:=\sum_{y\in V}C_{x,y}\mu_y$ and jumps to
point $y\in V$ with probability $C_{x,y}\mu_y/C_x$; this procedure
is then iterated choosing independent hopping times. Such a Markov
process $X$ is called a variable speed $\alpha$-stable-like random walk on $V$
-- see Section \ref{constSRW} concerning a constant speed $\alpha$-stable-like random walk.
We write $p(t,x,y)$ for the heat kernel of $X$
on $V$; that is, the transition density of the process $X$ with
respect to $\mu$ which is defined by
$p(t,x,y)=\mu_y^{-1}\Pp^x(X_t=y).$

The goal of this section is to study moment estimates for truncated
$\alpha$-stable-like processes on graphs, which are crucial to
obtain probability estimates for the exit time of the (original)
process. For this, we first consider on-diagonal upper
bounds for heat kernels of
the truncated process in Subsection
\ref{Subsection2.1}. Then, by adopting the localization method and
using the idea of Bass-Nash,
in Subsection \ref{Subsection2.2} we
establish  moment estimates for the truncation of the localized
$\alpha$-stable-like process on graphs.

\subsection{On-diagonal upper
bounds
for heat kernel}\label{Subsection2.1}
 In this subsection, we are concerned
with the truncated Dirichlet form corresponding to $(D,\F)$. For
fixed $1\le \delta<r_G$, define the operator $L^{\delta}f(x)=\sum_{z
\in V:\rho(z,x)\le
\delta}\big(f(z)-f(x)\big)\frac{w_{z,x}}{\rho(z,x)^{d+\alpha}}\mu_z.$
Then, the associated bilinear form is given by
$$
D^{\delta}(f,f)=\frac{1}{2}\sum_{x,y\in V: \rho(x,y)\le
\delta}\big(f(x)-f(y)\big)^2\frac{w_{x,y}}{\rho(x,y)^{d+\alpha}}\mu_x\mu_y.$$ Throughout this part, we always assume that
\begin{equation}\label{nl2-1-0} C_{V,\delta}:=\sup_{x\in V}\sum_{y\in V: \rho(x,y)>\delta}\frac{w_{x,y}}{\rho(x,y)^{d+\alpha}}\mu_y<\infty.\end{equation}
By \eqref{nl2-1-0} and the symmetry of $w_{x,y}$, we can easily see that for all $f\in \F$,
\begin{align*}D^{\delta}(f,f)\!\le  D(f,f)\!\le  D^\delta (f,f)\!+2\sum_{x\in V} f(x)^2\mu_x\sum_{y\in V:\rho(y,x)>\delta} \frac{w_{x,y}}{\rho(x,y)^{d+\alpha}}\mu_y
\!\le D^{\delta}(f,f)\!+2C_{V,\delta}\|f\|_2^2. \end{align*}
Consequently,
$(D^{\delta},\F)$ is also
a regular and symmetric Dirichlet form
on $L^2(V;\mu)$. Denote by $X^{\delta}:=\big((X_t^{\delta})_{t\ge 0}, (\Pp_x)_{x\in V}\big)$  the associated Hunt process, which is called the truncated process associated with $X$ in the literature.

In the following, we denote by $p^{\delta}(t,x,y)$ the heat kernel
of $X^{\delta}$. Given a sequence of
$w:=\{w_{x,y}:x,y\in V\}$, we set
$B^w(x,r):=\{z\in B(x,r): w_{x,z}>0\}$ for every $x\in V$ and $r\ge 1$.
The main statement in this part is as follows.
\begin{proposition}\label{np2-1}
Suppose  that \eqref{nl2-1-0} holds for some $\delta>0$,  and that
there exist  constants $\theta \in (0,1)$ and $C_1, C_2>0$
such that for every
$\delta^{\theta}\le r \le \delta$,
\begin{equation}\label{nl2-1-1}
\sup_{x\in V}\sum_{y \in B^w(x,r)}w_{x,y}^{-1}\le C_1r^{d},
\end{equation}
\begin{equation}\label{np2-1-1a}
\inf_{x\in V}\mu\big(B^w(x,r)\big)\ge C_2r^d
\end{equation}
and
\begin{equation}\label{np2-1-1}
\sup_{x \in V}\sum_{y \in V: \rho(y,x)\le
r}\frac{w_{x,y}}{\rho(x,y)^{d+\alpha-2}}\le C_1r^{2-\alpha}.
\end{equation} Then, for each $\theta'\in (\theta,1)$,
there are constants $\delta_0>0$
$($depending only
on $\theta'$, $\theta$$)$ and $C_3>0$
such that when the constant $\delta>0$ above satisfies $\delta_0\le \delta<r_G$,
the following estimate holds
\begin{equation}\label{np2-1-2}
 p^{\delta}(t,x,y)\le C_3t^{-d/\alpha},\quad \forall\, 2\delta^{\theta' \alpha}\le t \le \delta^{\alpha} \textrm{ and }x,y\in V.
 \end{equation}
\end{proposition}

In order to get on-diagonal upper
bounds
for the heat kernel of the truncated process
$X^\delta$, we need the following scaled Poincar\'{e}-type
inequality.
We note that the inequality \eqref{nl2-1-2a} below is different from the standard
Poincar\'e inequality since here the term $(f)_{B(x,r_0)}$ (for the standard Poincar\'e inequality) is replaced by $(f)_{B^w(z,r_0)}$, which will be
adopted to deal with the case that $w_{x,y}$ is degenerate.

\begin{lemma}\label{nl2-1}
Suppose that there exist
constants $C_1,C_2>0$ and $1\le r_0<r_G$
such that
\begin{equation}\label{nl2-1-1a}
\sup_{x \in V}\sum_{y\in B^w(x,r_0)}w_{x,y}^{-1}\le C_1
r_0^{d}
\end{equation} and
\begin{equation}\label{nl2-1-1b}
\inf_{x\in V}\mu(B^w(x,r_0))\ge C_2r_0^d.
\end{equation}
Then there is a constant $C_3>0$,
independent of $r_0$,
such that for all $x\in
V$ and measurable function $f$ on $V$,
\begin{equation}\label{nl2-1-2a}
\begin{split}
\sum_{z\in B(x,r_0)}\!\!\!(f(z)-&(f)_{B^w(z,r_0)})^2\mu_z
\le C_3 r_0^{\alpha}\!\!\!
\sum_{y,z\in V: z\in B(x,r_0), y\in B(z,r_0)}
\!\!\!\!(f(z)-f(y))^2 \frac{w_{z,y}}{\rho(z,y)^{d+\alpha}}\mu_z\mu_y,\end{split}
\end{equation}
where for $A\subset V$,
$(f)_A:={\mu(A)}^{-1}\sum_{z\in A}f(x)\mu_z.$
\end{lemma}
\begin{proof}
For every $x\in V$ and measurable
function $f$ on $V$, we have
\begin{align*}
&\sum_{z\in B(x,r_0)}(f(z)-(f)_{B^w(z,r_0)})^2\mu_z=\sum_{z\in B(x,r_0)}\Big(\frac{1}{\mu(B^w(z,r_0))}
\sum_{y\in B^w(z,r_0)}(f(z)-f(y))\mu_y\Big)^2\mu_z\\
&\le \frac{c_1}{r_0^{2d}}\sum_{z\in B(x,r_0)}
\bigg[\Big(\sum_{y\in B^w(z,r_0)}(f(z)-f(y))^2 \frac{w_{z,y}}{\rho(z,y)^{d+\alpha}}\Big)
\Big(\sum_{y\in B^w(z,r_0)}w_{z,y}^{-1}\rho(z,y)^{d+\alpha}\Big)\bigg] \\
&\le c_2r_0^{-d+\alpha}\Big(\sup_{z \in V} \sum_{y\in B^w(z,r_0)}w_{z,y}^{-1}\Big)\Big(
\sum_{y,z\in V: z\in B(x,r_0), y\in B(z,r_0)}
\big(f(z)-f(y)\big)^2\frac{w_{z,y}}{\rho(z,y)^{d+\alpha}}\Big)\\
&\le c_3r_0^{\alpha}
\sum_{y,z\in V: z\in B(x,r_0), y\in B(z,r_0)}
\big(f(z)-f(y)\big)^2\frac{w_{z,y}}{
\rho(z,y)^{d+\alpha}}\mu_z\mu_y,
 \end{align*}
 where the first inequality follows from \eqref{al2-0}, \eqref{nl2-1-1b} and the Cauchy-Schwarz inequality, in the second inequality we have
 used the fact that
 $\rho(z,y)\le r_0$ for every $y\in B^w(z,r_0)$, and the third
 inequality is due to \eqref{al2-0} and \eqref{nl2-1-1a}.
 This proves \eqref{nl2-1-2a}.
\end{proof}

We are now in a position to present the

 \begin{proof}[Proof of Proposition $\ref{np2-1}$]
Without mention, throughout the proof constant $c_i$ will be
independent of $\delta$, $t$, $x$, $y$ and $r_G$. Since, by the
Cauchy-Schwarz inequality, $$p^{\delta}(t,x,y)\le
p^{\delta}(t,x,x)^{1/2}p^{\delta}(t,y,y)^{1/2}$$ for any $t>0$ and
$x,y\in V,$ it suffices to verify \eqref{np2-1-2} for the case that
$x=y$. The proof is split into three steps.

{\bf Step (1):} We first note that under \eqref{nl2-1-0} and \eqref{np2-1-1},
$\sup_{x\in V}\sum_{y\in V}\frac{w_{x,y}}{\rho(x,y)^{d+\alpha}}\mu_y<\infty.$ This along with (the second statement in) \cite[Theorem 3.2]{CKK} yields that the process $X^\delta$ is conservative. By \cite[Proposition 5 and Theorem 8]{D}, we have the following upper bound
for $p^{\delta}(t,x,y)$:
\begin{equation}\label{np2-1-3}
\begin{split}
p^{\delta}(t,x_1,x_2)&\le \mu_{x_1}^{-1/2}\mu_{x_2}^{-1/2}\inf_{\phi\in L^{\infty}(V;\mu)}
\exp\big(\phi(x_1)-\phi(x_2)
+b(\phi)t\big)
\end{split}
\end{equation} for all $t>0$ and $x_1,x_2\in V,$
where
$$
b(\phi):=\frac{1}{2}\sup_{x \in V}\sum_{y \in V: \rho(y,x)\le \delta}
\frac{w_{x,y}}{\rho(x,y)^{d+\alpha}}\Big(e^{\phi(y)-\phi(x)}+
e^{\phi(x)-\phi(y)}-2\Big)\mu_y.
$$

For fixed $x_1,x_2\in V$, taking $\phi(x)=\rho(x,x_1)\wedge
\rho(x_1,x_2)$ for any $x\in V$, we get that
\begin{align*}
b(\phi)&\le \frac{1}{2}\sup_{x \in V}\sum_{y \in V:\rho(y,x)\le
\delta}
\frac{w_{x,y}}{\rho(x,y)^{d+\alpha}}\Big(e^{\rho(x,y)}+e^{-\rho(x,y)}-2\Big)\mu_y\\
&\le \frac{1}{2}\sup_{x \in V}\sum_{y \in V:\rho(y,x)\le \delta}
\frac{w_{x,y}}{\rho(y,x)^{d+\alpha}}\rho(x,y)^2e^{\rho(x,y)}\mu_y\\
&\le \frac{1}{2}e^{\delta}\sup_{x \in V}\sum_{y\in V:\rho(y,x)\le
\delta}
\frac{w_{x,y}}{\rho(x,y)^{d+\alpha-2}}\mu_y\le c_1e^{\delta}\delta^{2-\alpha}\le 2c_1e^{2\delta},
\end{align*}
where in the first inequality above we have used the facts that
$s\mapsto e^s+e^{-s}$ is increasing on $[0,\infty)$ and
$|\phi(x)-\phi(y)|\le \rho(x,y)$ for all $x,y\in V$, the second inequality is due to the fact that
$e^{s}+e^{-s}-2\le s^2 e^s$ for all $s\ge 0$, and the fourth
inequality follows from \eqref{np2-1-1}. Combining this with
\eqref{np2-1-3}, we arrive at the statement that for all $t>0$ and $x_1,x_2\in
V$,
\begin{equation}\label{np2-1-3a}
p^{\delta}(t,x_1,x_2)\le c_M\exp\big(-\rho(x_1,x_2)+2c_1e^{2\delta }t\big).
\end{equation}

Furthermore, it follows from the symmetry of $w_{x,y}$, the fact that
$p^\delta(t,x,y)\mu_y\le 1$ for all $t>0$ and $x,y\in V$, \eqref{np2-1-1} and \eqref{np2-1-3a} that for
every $x\in V$,
\begin{align*}
&\sum_{z,v\in V:\rho(z,v)\le \delta}
\big(p^{\delta}(t,x,z)-p^{\delta}(t,x,v)\big)^2
\frac{w_{z,v}}{\rho(z,v)^{d+\alpha}}\mu_z\mu_v\\
&\le \sum_{z,v\in V:\rho(z,v)\le \delta}
\big(p^{\delta}(t,x,z)+p^{\delta}(t,x,v)\big)^2
\frac{w_{z,v}}{\rho(z,v)^{d+\alpha}}\mu_z\mu_v\\
&\le 4c_M\sum_{z \in V}p^{\delta}(t,x,z)\Big( \sup_{z\in
V}\sum_{v\in V:\rho(v,z)\le \delta}
\frac{w_{z,v}}{\rho(z,v)^{d+\alpha}}\Big)\\
&\le 4c_M\sum_{z \in V}p^{\delta}(t,x,z)\Big( \sup_{z\in
V}\sum_{v\in V:\rho(z,v)\le \delta}
\frac{w_{z,v}}{\rho(z,v)^{d+\alpha-2}}\Big)\le c_2(\delta,t)\sum_{z \in V}\exp(-\rho(z,x))<\infty,
\end{align*}
where in the last inequality we used the fact that
\begin{align*}
\sum_{z \in V}\exp(-\rho(z,x))&\le c_M\sum_{r=0}^{\infty}\sum_{z \in V: \rho(x,z)=r}
e^{-r}\mu_z\le c_M\sum_{r=0}^{\infty}\mu(B(x,r))e^{-r}\le c_M c_G\sum_{r=1}^{\infty}r^d e^{-r}<\infty.
\end{align*}
 Therefore, according to the
Fubini theorem and \eqref{np2-1-3a},
for every $x \in V$,
\begin{equation}\label{np2-1-4}
\begin{split}
&\sum_{z\in V}L^{\delta}p^{\delta}(t,x,\cdot)(z)p^{\delta}(t,x,z)\mu_z
=-\frac{1}{2}\sum_{z,v\in V:\rho(z,v)\le \delta}
\big(p^{\delta}(t,x,z)-p^{\delta}(t,x,v)\big)^2\frac{w_{z,v}}{\rho(z,v)^{d+\alpha}}\mu_z\mu_v.
\end{split}
\end{equation}

{\bf Step (2):} Below we fix $x\in V$. Let $f_t(z)=p^{\delta}(t,x,z)$ and
$\psi(t)=p^{\delta}(2t,x,x)$ for all $z\in V$ and $t\ge 0$. Then,
 $\psi(t)=\sum_{z\in V} f_t(z)^2\mu_z$, and, by \eqref{np2-1-4},
$$
 \psi'(t)= \!2\sum_{z\in V} \!\frac{d f_t(z)}{dt} f_t(z)\mu_z=\!2\sum_{z \in V}\!
 L^{\delta}f_t(z) f_t(z)\mu_z=\!-\sum_{z,y\in V:\rho(z,y)\le \delta}
 \!(f_t(z)-f_t(y))^2
 \frac{w_{z,y}}{\rho(z,y)^{d+\alpha}}\mu_z\mu_y.
$$

Let $\delta^{\theta}\le r(t)\le \delta$ and $R:=R(\delta)\ge 1$ be some
constants to be determined later. Suppose that $B(x_i, r(t)/2)$
($i=1,\cdots, m$) is a maximal collection of disjoint balls with
centers in $B(x,R)$. Set $B_i=B(x_i,r(t))$ and $B_i^*=B(x_i,
2r(t))$. Then, $B(x,R)\subset \cup_{i=1}^mB_i\subset
B(x,R+r(t))\subset \cup_{i=1}^mB_i^*.$ If
$z\in
B(x,R+r(t))\cap B_i^*$ for some $1\le i\le m$, then $B(x_i,r(t)/2)\subset B(z,3r(t))$, and
so
$$c_3r(t)^d\ge \mu(B(z,3r(t)))\ge \sum_{i=1}^m\I_{\{z\in B_i^*\}} \mu(B(x_i,r(t)/2))\ge c_4r(t)^d|\{i:z\in B_i^*\}|.$$In the second inequality we used the fact that $B(x_i, r(t)/2)$,
$i=1,\cdots, m$, are disjoint, and in the first and the last inequality we have used
\eqref{al2-1}.  Thus, every $z\in B(x,R+r(t))$ is in at most
$c_5:=c_3/c_4$ of the balls
$B_i^*$ (hence at most $c_5$ of the balls
$B_i$). In particular,
\begin{equation}\label{np2-1-4a}
\sum_{i=1}^m\sum_{z\in B_i}=\sum_{i=1}^m\sum_{z\in
B(x,R+r(t))}\I_{B_i}(z) =\sum_{z\in
B(x,R+r(t))}\sum_{i=1}^m\I_{B_i}(z) \le  c_5\sum_{z \in
B(x,R+r(t))}.
\end{equation}
According to (the proof of) Lemma \ref{nl2-1}, \eqref{nl2-1-1} and \eqref{np2-1-1a} imply that for every $\delta^{\theta}\le r \le \delta$,
$x\in V$ and measurable function $f$ on $V$,
\begin{equation}\label{nl2-1-2}
\begin{split}
\sum_{z\in B(x,r)}\!\!(f(z)-(f)_{B^w(z,r)})^2\mu_z\le c_6 r^{\alpha}\!\!
\sum_{z,y\in V:z\in B(x,r),y\in B(z,r)}
\!\!(f(z)-f(y))^2 \frac{w_{z,y}}{\rho(z,y)^{d+\alpha}}\mu_z\mu_y.\end{split}
\end{equation}
Hence, noticing that $\delta^\theta\le r(t)\le \delta$,
\begin{align*}
&\sum_{z,y\in V:\rho(z,y)\le \delta} (f_t(z)-f_t(y))^2\frac{w_{z,y}}{\rho(z,y)^{d+\alpha}}\mu_z\mu_y\ge \frac{1}{c_5}\sum_{i=1}^m\sum_{z\in B_i}
\sum_{y\in B(z,r(t))}
(f_t(z)-f_t(y))^2\frac{w_{z,y}}{\rho(z,y)^{d+\alpha}}
\mu_z\mu_y\\
&\ge \frac{c_7}{r(t)^{\alpha}}\Big[\sum_{i=1}^m \sum_{z\in B_i}
f_t^2(z)\mu_z-
2\sum_{i=1}^m \sum_{z\in B_i}f_t(z)(f_t)_{B^w(z,r(t))}\mu_z\Big]=: \frac{c_7}{r(t)^{\alpha}}(I_1-I_2), \end{align*}
where in the second inequality we have used \eqref{nl2-1-2}.

Furthermore, since $f_t(z)\mu_z\le 1$ for all $z\in V$ and $t>0$, we have
\begin{align*}
I_1\ge &\!\sum_{z\in \cup_{i=1}^m B_i} f_t^2(z)\mu_z\ge\!\sum_{z\in B(x,R)}
f_t^2(z)\mu_z
=\!\psi(t)\!-\!\sum_{z\in V: \rho(z,x)> R}f_t^2(z)\mu_z\ge\!  \psi(t)\!-\!\sum_{z\in V:\rho(z,x)> R}f_t(z).
\end{align*}
So, by \eqref{np2-1-3a}, we can choose
$R:=R(\delta)=2c_1e^{4\delta}$ such that for all
$\delta^{\theta\alpha}\le t\le \delta^{\alpha}$,
\begin{align*}
\sum_{z\in V: \rho(z,x)> R}f_t(z)&\le
\sum_{z\in V:\rho(z,x)> 2c_1e^{4\delta}}\exp\big(-\rho(z,x)+2c_1e^{2\delta}\delta^\alpha\big)\\
&\le c_M\sum_{z\in V:\rho(z,x)>2c_1
e^{4\delta}}\exp\big(-\rho(z,x)/2\big)\mu_z\\
&\le c_M\sum_{r=2c_1e^{4\delta}}^\infty\mu(B(x,r))e^{-r/2}\le c_8\delta^{-d}\le
c_8r(t)^{-d},
\end{align*}
where the last inequality follows from the fact that $r(t)\le \delta$.
 On the other hand, due to \eqref{np2-1-1a} and the fact that $\sum_{z \in V}f_t(z)\mu_z\le 1$ for all $t>0$,
\begin{equation*}
\sup_{z \in V}(f_t)_{B^w(z,r(t))}\le
\sup_{z\in V}\mu\big(B^w(z,r(t))\big)^{-1}\cdot \sum_{z\in V}f_t(z)\mu_z\le C_2^{-1}r(t)^{-d}.
\end{equation*}
This along with \eqref{np2-1-4a} yields that
 \begin{align*}
 I_2&\le C_2^{-1}r(t)^{-d}\sum_{i=1}^m\sum_{z \in B_i} f_t(z)\mu_z
 \le C_2^{-1}c_5r(t)^{-d}\sum_{z \in B(x,R+r(t))}f_t(z)\mu_z\le C_2^{-1} c_5r(t)^{-d}.
 \end{align*}
 Therefore, combining all estimates above,
 we arrive at the statement
 that for every $\delta^{\theta}\le r(t)\le \delta$,
\begin{equation}\label{np2-1-5}
\psi'(t)\le -c_{9}r(t)^{-\alpha}\left(\psi(t)-c_{10}r(t)^{-d}\right).
\end{equation}

{\bf Step (3):} For any $\theta'\in (\theta,1)$ and any $1\le \delta<r_G$ large enough, we claim that there exists
$t_0\in [\delta^{\theta\alpha},\delta^{\theta'\alpha}]$ such that
\begin{equation}\label{np2-1-6}
\left( \frac{1}{2c_{10}}\psi(t_0)\right)^{-1/d}\ge \delta^{\theta}.
\end{equation}
Indeed, suppose that \eqref{np2-1-6} does not hold. Then,
\begin{equation}\label{np2-1-7}
\left( \frac{1}{2c_{10}}\psi(t)\right)^{-1/d}< \delta^{\theta},\quad \forall\ \delta^{\theta\alpha}\le t\le \delta^{\theta'\alpha},
\end{equation}
which means that $\psi(t)\ge 2c_{10}\delta^{-d\theta}$ for all $\delta^{\theta\alpha}\le t\le \delta^{\theta'\alpha}$.
Hence, taking $r(t)=\delta^{\theta}$ in \eqref{np2-1-5}, we find that
$\psi'(t)\le -{2}^{-1}c_{9}\delta^{-\theta\alpha} \psi(t)$ for any $\delta^{\theta\alpha}\le t\le \delta^{\theta'\alpha},$
which along with the fact $\psi(t)\le \mu_x^{-1}\le c_M$ for all $t>0$ yields
that
$\psi(t)\le  c_M e^{-2^{-1}{c_{9}}\delta^{-\theta\alpha}(t-\delta^{\theta\alpha})}$ for any $\delta^{\theta\alpha}\le t\le \delta^{\theta'\alpha}.$
In particular,
$\psi(\delta^{\theta'\alpha})\le c_Me^{-2^{-1}{c_{9}}\delta^{-\theta\alpha}(\delta^{\theta'\alpha}-\delta^{\theta\alpha})}.$
On the other hand, according to \eqref{np2-1-7}, we have
$\psi(\delta^{\theta'\alpha})\ge
2c_{10} \delta^{-d\theta}.$
Thus, there is a contradiction between these two inequalities
for $\delta$ large enough, and so \eqref{np2-1-6} is true.

Next, assume we take $1\le \delta<r_G$ large enough such that \eqref{np2-1-6} holds. Since $t\mapsto\psi(t)$ is non-increasing on $(0,\infty)$
and $t_0\le \delta^{\theta'\alpha}$,
$$
\left( \frac{1}{2c_{10}}\psi(t)\right)^{-1/d}\ge \delta^{\theta},\quad \forall\ \delta^{\theta'\alpha}\le t\le \delta^{\alpha}.
$$
Let $$\tilde{t}_0:=\sup\bigg\{t>0: \left(
\frac{1}{2c_{10}}\psi(t)\right)^{-1/d}<\delta/2\bigg\}.$$ By the
non-increasing property of $\psi$ on $(0,\infty)$ again, if $\tilde t_0\le
\delta^{\theta'\alpha}$, then
$
\psi(t)\le \psi(\tilde t_0)=2c_{10}(\delta/2)^{-d}\le
c_{11}t^{-d/\alpha}$ for any $\delta^{\theta'\alpha}\le t\le \delta^{\alpha}.$
 This proves \eqref{np2-1-2} under the assumption $t_0\le \delta^{\theta'\alpha}$.

When $\tilde t_0>\delta^{\theta'\alpha}$,
$$
\delta^{\theta}\le \left( \frac{1}{2c_{10}}\psi(t)\right)^{-1/d}\le  \delta/2
,\quad \forall\ \delta^{\theta'\alpha}\le t\le \tilde t_0.
$$
Then, taking $r(t)=\big( \frac{1}{2c_{10}}\psi(t)\big)^{-1/d}$ in
\eqref{np2-1-5}, we have
$\psi'(t)\le-c_{12}\psi(t)^{1+\alpha/d}$
for any $\delta^{\theta'\alpha}\le t\le \tilde t_0.
$
Hence,
$
\psi(s) \le
c_{13}\big(s-\delta^{\theta'\alpha}+\psi(\delta^{\theta'\alpha})^{-\alpha/d}\big)^{-d/\alpha}\le
c_{14}s^{-d/\alpha}$ for any $2\delta^{\theta'\alpha}\le s\le
\tilde t_0.$
If $\tilde t_0>\delta^{\alpha}$, then \eqref{np2-1-2} holds.
If $\delta^{\theta'\alpha}<\tilde t_0\le \delta^{\alpha}$, then, for all $\tilde t_0 \le s\le \delta^{\alpha}$,
$
\psi(s)\le \psi(\tilde t_0)=2c_{10}(\delta/2)^{-d}\le c_{15}s^{-d/\alpha},
$
so \eqref{np2-1-2} also holds. The proof is complete.
\end{proof}
\begin{remark}
By carefully tracking the proof, we can see that the constant $C_3>0$ in the statement of Proposition
$\ref{np2-1}$ can be chosen independently of the choice of $\delta_0,\delta>0$.
\end{remark}

\subsection{Localization method and moment estimates of the truncated process}\label{Subsection2.2} In this part, we fix $x_0\in V$ and $R\ge 1$. Define a
symmetric regular Dirichlet form $(\hat D^{x_0, R},\hat \F^{x_0,
R})$ as follows
\begin{align*}
\hat D^{x_0,R}(f,f)=&\sum_{x,y\in
V}\big(f(x)-f(y)\big)^2\frac{\hat w_{x,y}}{\rho(x,y)^{d+\alpha}}\mu_x\mu_y,\quad
f\in
\hat \F^{x_0,R},\\
\hat \F^{x_0,R}=&\{f \in L^2(V;\mu): \hat D^{x_0,
R}(f,f)<\infty\},
\end{align*}
where
\begin{equation*}
\hat w_{x,y}=
\begin{cases}
& w_{x,y},\ \ \ \text{if}\ x\in B(x_0, R)\ \text{or}\ y\in B(x_0, R),\\
& \,\, 1,\ \ \ \ \ \ \text{otherwise}.
\end{cases}
\end{equation*}
In particular, coefficients of the Dirichlet form $(\hat D^{x_0, R},\hat \F^{x_0,
R})$ outside $B(x_0,R)$ are uniformly bounded. This point is quite important
in the following arguments for the exit time estimates from $B(x_0,R)$.

Note that, according to the definition of $\hat w_{x,y}$, for any $x\in V$,
\begin{equation}\label{e:bound}
\begin{split}
&\sum_{y\in V}\frac{\hat w_{x,y}}{\rho(x,y)^{d+\alpha}}=\sum_{y \notin B(x_0, R)}\frac{\hat w_{x,y}}{\rho(x,y)^{d+\alpha}}+\sum_{y \in B(x_0, R)}\frac{w_{x,y}}{\rho(x,y)^{d+\alpha}}\\
&\le \sup_{z\in B(x_0, R)} \sum_{v\in V} \frac{w_{z,v}}{\rho(z,v)^{d+\alpha}}+
\sup_{z\notin B(x_0,R)}\sum_{y\in V:y\neq z}\frac{1}{\rho(z,y)^{d+\alpha}}
+\sum_{y \in B(x_0, R)}\frac{w_{x,y}}{\rho(x,y)^{d+\alpha}}\\
&\le  \sup_{z\in B(x_0, R)} \sum_{v\in V} \frac{w_{z,v}}{\rho(z,v)^{d+\alpha}}
+c_M\sup_{z\notin B(x_0,R)}\sum_{k=1}^{\infty}\sum_{y\in V: 2^{k-1}\le \rho(y,z)< 2^{k}}  \frac{1}{\rho(y,z)^{d+\alpha}}\mu_y\\
&\quad+\sum_{y \in B(x_0, R)}\bigg(\sup_{z \in B(x_0, R)}\sum_{v \in V}
\frac{w_{z,v}}{\rho(z,v)^{d+\alpha}}\bigg)\\
&\le  \sup_{z\in B(x_0, R)} \sum_{v\in V} \frac{w_{z,v}}{\rho(z,v)^{d+\alpha}}
\!+\!c_M c_G\sum_{k=1}^{\infty} \frac{2^{kd}}{2^{(k-1)(d+\alpha)}}
\!+\!\sum_{y \in B(x_0, R)}\sup_{z \in B(x_0, R)}\sum_{v \in V}
\frac{w_{z,v}}{\rho(z,v)^{d+\alpha}}\\
&\le  c_1+c_2(1+R^d)\sup_{z \in B(x_0, R)}\bigg(\sum_{v \in V}
\frac{w_{z,v}}{\rho(z,v)^{d+\alpha}}\bigg)=:C(x_0,R)<\infty,
\end{split}
\end{equation}
where \eqref{eq:condwxy} was used in the fourth inequality. In
particular, by \eqref{e:bound} and (the second statement in)
\cite[Theorem 3.2]{CKK}, the associated Hunt process $\hat X^{
R}:=((\hat X_t^{R})_{t\ge 0}, (\Pp_x)_{x\in V})$ is conservative.
Here and in what follows, we omit the index $x_0$ for simplicity.
The process $\hat X^{ R}$ is called the localized
$\alpha$-stable-like process (with parameters $x_0\in V$ and $R\ge1$).

We also consider the following
truncated Dirichlet form $(\hat D^{x_0,R, R}, \hat\F^{x_0,R})$:
\begin{align*}
&\hat D^{x_0, R, R}(f,f)=\sum_{x,y\in V: \rho(x,y)\le
R}\big(f(x)-f(y)\big)^2\frac{\hat w_{x,y}}{\rho(x,y)^{d+\alpha}}\mu_x\mu_y,\quad
f\in \hat\F^{x_0, R}.
\end{align*}
Let $\hat X^{R,R}:=((\hat X_t^{
R,R})_{t\ge 0},(\Pp_x)_{x\in V})$ be the associated Hunt process.
In particular, due to \eqref{e:bound} again, the process $\hat X^{R,R}$ is also conservative.
Denote by $\hat
p^{R}(t,x,y)$ and $\hat p^{
R,R}(t,x,y)$ heat kernels of the processes $\hat X^R$ and $\hat X^{R,R}$, respectively.

\ \

The following statement concerns moment estimates of $\hat
X^{R,R}$. These estimates are key inputs for exit time estimates for the original process $X$
in the next section.

\begin{proposition}\label{np-1}
Suppose that there exist $1\le R_0<r_G$, $\theta \in (0,1)$ and $C_1,C_2>0$
such that
for every $R_0<R<r_G$ and
$R^{\theta}\le r \le R$,
\begin{equation}\label{np1-1}
\sup_{x\in {B(x_0,3R)}}\sum_{y\in V:\rho(x,y)\le  r}
\frac{w_{x,y}}{\rho(x,y)^{d+\alpha-2}}
\le C_1 r^{2-\alpha},
\end{equation}
\begin{equation}\label{np1-1b}
\inf_{x\in B(x_0,3R)}\mu(B^w(x,r))\ge C_2r^d
\end{equation}
and
\begin{equation}\label{np1-1a}
\sup_{x \in { B(x_0,3R)}}\sum_{y\in B^w(x,r)}w_{x,y}^{-1}\le C_1r^d.
\end{equation}
Then for every $\theta' \in (\theta,1)$,
 there exist constants $R_1>R_0$
$($which depends only on $\theta$, $\theta'$ and $R_0)$
and $C_3>0$ $($which is independent of $x_0$, $R_0$ and $R_1$$)$ such
that for every $R_1<R<r_G$ and $x\in V$,
\begin{equation}\label{np1-3}
\Ee_x\big[\rho\big(\hat X_t^{R, R},x\big)\big]\le C_3R\left(\frac{t}{R^\alpha}\right)^{1/2}
\left[1+\log\left(\frac{R^{\alpha}}{t}\right)
\right],\quad \forall \ R^{\theta' \alpha}\le t \le R^{\alpha},
\end{equation}
\end{proposition}
\begin{proof}
Throughout the proof, we first suppose that
there exist
positive constants $c(x_0, R)$ and $\tilde c(x_0,R)$
such that
\begin{equation}\label{np1-2a}
\tilde c(x_0,R)\le \inf_{x,y\in V}\hat w_{x,y}\le \sup_{x,y \in V}\hat w_{x,y}\le c(x_0,R).
\end{equation}
If \eqref{np1-2a} is not satisfied, then, by taking
$w_{x,y}^{\varepsilon}:=w_{x,y}+\varepsilon$ and then letting
$\varepsilon \downarrow 0$, we can prove that \eqref{np1-3} still
holds true. Moreover, all the constants in the proof below are
independent of $\varepsilon$ unless specifically claimed.
The argument below is partly motivated by the method of Bass \cite{Bass} for diffusions
(see also Barlow \cite{B} and Nash
\cite{Nash}), but some non-trivial modifications are required for jump processes.

{\bf Step (1):}
By \eqref{np1-1}, \eqref{np1-1b}, \eqref{np1-1a} and the definition of $\hat w_{x,y}$, for every $R_0<R<r_G$ and $R^{\theta}\le r\le R$,
\begin{equation}\label{np1-4}
\sup_{x\in V}\sum_{y\in V:\rho(x,y)\le r}
\frac{\hat w_{x,y}}{\rho(x,y)^{d+\alpha-2}}\le c_0r^{2-\alpha},
\end{equation}
$
\inf_{x\in V}\mu(B^{\hat w}(x,r))\ge c_1r^d
$
and
$
\sup_{x \in V}\sum_{y \in B^{\hat w}(x,r)}
\hat w_{x,y}^{-1}\le c_0r^d,
$
where $B^{\hat w}(x,r):=\{z\in V:\rho(z,x)\le r,\ \hat w_{z,x}>0\}$.
Let $\theta' \in (\theta,1)$ and $\theta_0=({\theta+\theta'})/{2}$. Taking $\rho=R$ in Proposition \ref{np2-1}, we find that there exists a constant $\tilde R_0\ge R_0$ (which only depends on $\theta$ and  $\theta'$) such that
whenever $\tilde R_0<R<r_G$,
\begin{equation}\label{np1-2}
\hat p^{R,R}(t,x,y)\le c_2t^{-d/\alpha},\quad \forall\
2R^{\theta_0 \alpha}\le t \le R^{\alpha},\ x,y\in V.
\end{equation}

For every $t>0$, we define
$$
M(t)=\sum_{y\in V}\rho(x,y)\hat p^{R,R}(t,x,y)\mu_y,\quad Q(t)=-\sum_{y \in V}\hat p^{R,R}(t,x,y)\left[\log \hat p^{R,R}(t,x,y)\right]\mu_y.
$$
Below, we fix $x \in V$ and  set $f_t(y)=\hat p^{
R,R}(t,x,y)$ for all $y\in V$ and $t>0$.

By \eqref{np1-2a}, we
can obtain upper and lower bounds for $\hat p^{R,R}(t,x,y)$
(see \cite{D} for upper bounds on graph or \cite{CKKmn} for two-sided estimates in the Euclidean space), which yields that
\begin{align*}
&\sum_{y,z\in V:\rho(y,z)\le R}|f_t(y)-f_t(z)|
|\log f_t(y)-\log f_t(z)|\frac{\hat w_{y,z}}{\rho(y,z)^{d+\alpha}}\mu_y\mu_z\\
&\le \sum_{y,z\in V: \rho(y,z)\le R}\big(f_t(y)+f_t(z)\big)
\big(|\log f_t(y)|+|\log f_t(z)|\big)\frac{\hat w_{y,z}}{\rho(y,z)^{d+\alpha}}\mu_y\mu_z
<\infty.
\end{align*}
Thus,
\begin{align*}
&-\sum_{y\in V}(\log f_t(y)+1)\hat L^{R,R}f_t(y)\mu_y\\
&=\frac{1}{2}\sum_{y,z\in V:\rho(y,z)\le R}\big(f_t(y)-f_t(z)\big)
\big(\log f_t(y)-\log f_t(z)\big)\frac{\hat w_{y,z}}{\rho(y,z)^{d+\alpha}}\mu_y\mu_z,
\end{align*} where $\hat L^{R,R}$ is the generator associated with
$(\hat D^{x_0,R, R},\hat \F^{x_0,R,R})$, i.e.,
$$\hat L^{R,R} f(x)=\sum_{y\in V: \rho(x,y)\le R}(f(y)-f(x))\frac{\hat w_{x,y}}{\rho(x,y)^{d+\alpha}}\mu_y.$$
Therefore,
\begin{align*}
Q'(t)&=-\sum_{y\in V}(\log f_t(y)+1)\hat L^{R,R}f_t(y)\mu_y\\
&=\frac{1}{2}\sum_{y,z\in V:\rho(y,z)\le R}\big(f_t(y)-f_t(z)\big)
\big(\log f_t(y)-\log f_t(z)\big)\frac{\hat w_{y,z}}{\rho(y,z)^{d+\alpha}}\mu_y\mu_z\ge 0.
\end{align*}
In particular, $Q(\cdot)$ is a non-decreasing function on $(0,\infty)$.

On the other hand, for all $\tilde R_0<R<r_G$, by the Cauchy-Schwarz inequality,
\begin{align*}
M'(t)&=\sum_{y\in V}\rho(x,y)\hat L^{R,R}f_t(y)\mu_y\\
&=-\frac{1}{2}\sum_{y,z \in V: \rho(y,z)\le R}\big(\rho(x,y)-\rho(x,z)\big)\big(f_t(y)-f_t(z)\big)\frac{\hat w_{y,z}}{\rho(y,z)^{d+\alpha}}
\mu_y\mu_z\\
&\le \left(\frac{1}{4}\sum_{y,z \in V: \rho(y,z)\le R}
\big(\rho(x,y)-\rho(x,z)\big)^2
\big(f_t(y)+f_t(z)\big)\frac{\hat w_{y,z}}{\rho(y,z)^{d+\alpha}}\mu_y\mu_z\right)^{1/2}\\
&\quad \times\left(\sum_{y,z\in V: \rho(y,z)\le R}
\frac{(f_t(y)-f_t(z))^2}{f_t(y)+f_t(z)}\frac{\hat w_{y,z}}{\rho(y,z)^{d+\alpha}}\mu_y\mu_z\right)^{1/2}\\
&\le \left(\frac{c_M}{2}\sup_{z \in V}\sum_{y\in V: \rho(y,z)\le R}
\frac{\hat w_{y,z}}{\rho(y,z)^{d+\alpha-2}}\right)^{1/2}\\
&\quad \times \left(\sum_{y,z\in V: \rho(y,z)\le R}
\frac{(f_t(y)-f_t(z))^2}{f_t(y)+f_t(z)}\frac{\hat w_{y,z}}{\rho(y,z)^{d+\alpha}}\mu_y\mu_z\right)^{1/2}\\
&\le c_3R^{1-\alpha/2}\left(\sum_{y,z\in V:\rho(y,z)\le R}
\frac{(f_t(y)-f_t(z))^2}{f_t(y)+f_t(z)}\frac{\hat w_{y,z}}{\rho(y,z)^{d+\alpha}}\mu_y\mu_z\right)^{1/2},
\end{align*}
where the equality above follows from the fact
$$
\sum_{y,z\in V:\rho(y,z)\le R}|f_t(y)-f_t(z)|
\frac{\hat w_{y,z}}{\rho(y,z)^{d+\alpha-1}}<\infty,
$$
thank to \eqref{np1-2a} again, in the second inequality we used \eqref{al2-0} and the fact that
${\sum_{z\in V}}f_t(z)\mu_z\le 1$
for all $t>0$, and in the last inequality we have
used \eqref{np1-4}.

Noting that
$$
\frac{(s-t)^2}{s+t}\le \big(s-t\big)\big(\log s-\log t\big), \quad s,t>0,
$$
we have
\begin{align*}
&\sum_{y,z\in V:\rho(y,z)\le R}
\frac{(f_t(y)-f_t(z))^2}{f_t(y)+f_t(z)}\frac{\hat w_{y,z}}{\rho(y,z)^{d+\alpha}}\mu_y\mu_z\\
&\le \sum_{y,z\in V:\rho(y,z)\le R}\big(f_t(y)-f_t(z)\big)
\big(\log f_t(y)-\log f_t(z)\big)\frac{\hat w_{y,z}}{\rho(y,z)^{d+\alpha}}\mu_y\mu_z=2Q'(t).
\end{align*}
Hence, combining all the estimates above, we arrive at the statement that for all $\tilde R_0<R<r_G$,
\begin{equation}\label{np1-4a}
M'(t)\le \sqrt{2}c_3R^{1-\alpha/2}Q'(t)^{1/2},\ \ \forall\ t>0.
\end{equation}

{\bf Step (2):} \eqref{np1-2} yields that for all $\tilde R_0<R<r_G$ and $2R^{\theta_0\alpha}\le t \le R^{\alpha}$,
\begin{align*}
Q(t)\ge -\left(\sum_{y \in
V}f_t(y)\right)\log (c_2t^{-d/\alpha})=\frac{d}{\alpha}\log t-c_4,\end{align*} where $c_4>0$ and the conservativeness of $\hat X^{R,R}$ was used in the right hand equality. Define
$$K(t)=d^{-1}\Big(Q(t)+c_4-\frac{d}{\alpha}\log t\Big),\quad t>0.$$ Obviously,
$K(t)\ge 0$ for all $t\in [2R^{\theta_0\alpha},R^{\alpha}]$, and
\begin{equation}\label{np1-5}
Q'(t)=d K'(t)+\frac{d}{\alpha t},\quad t>0.
\end{equation}
Set $T_0(R):=0\vee \sup\{t<2R^{\theta_0\alpha}: K(t)< 0 \}.$ It is easy to see  that
$K(t)\ge 0$ for all $t\in [T_0(R),R^{\alpha}]$ and $T_0(R)\le 2R^{\theta_0\alpha}$.
By \eqref{np1-4a} and \eqref{np1-5}, we have for all $t \in [T_0(R),R^{\alpha}]$,
\begin{equation}\label{np1-6}
\begin{split}
M(t)&=M(T_0(R))+\int^t_{T_0(R)}M'(s)\,ds\le M(T_0(R))+\sqrt{2}c_3R^{1-\alpha/2}\int^t_{T_0(R)}Q'(s)^{1/2}\,ds\\
&= M(T_0(R))+\sqrt{2}c_3R^{1-\alpha/2}\int^t_{T_0(R)}\Big(d K'(s)+\frac{d}{\alpha s}\Big)^{1/2}\,ds.
\end{split}
\end{equation}

Note that, by the mean-value theorem,  for every $a\in\R$ and $b>0$ with $a+b\ge 0$,
\begin{equation}\label{np1-7}
(a+b)^{1/2}\le b^{1/2}+a/(2b^{1/2}).
\end{equation}
Then, applying \eqref{np1-7} in the second term of the right hand
side of \eqref{np1-6} with $a=K'(s)$ and $b=\frac{1}{\alpha s}$, we
obtain that for all $t \in [T_0(R), R^{\alpha}]$,
\begin{equation}\label{np1-8}
\begin{split}
M(t)&\le M(T_0(R))+c_4R^{1-\alpha/2}
\int_{T_0(R)}^t s^{-1/2}\,ds+c_5R^{1-\alpha/2}\int_{T_0(R)}^t s^{1/2}K'(s)\,ds\\
&\le M(T_0(R))+c_6R^{1-\alpha/2}t^{1/2} +
c_5R^{1-\alpha/2}\int_{T_0(R)}^t\left[\big(s^{1/2}K(s)\big)'-\frac{s^{-1/2}K(s)}{2}\right]\,ds\\
&\le
M(T_0(R))+c_6R^{1-\alpha/2}t^{1/2}+c_5R^{1-\alpha/2}t^{1/2}K(t),
\end{split}
\end{equation}
where, in the last inequality, we used the fact that $K(t)\ge 0$ for all
$t \in [T_0(R),R^{\alpha}]$.

Furthermore, suppose that $T_0(R)>0$. Since $Q'(t)\ge 0$, by \eqref{np1-4a} and the
Cauchy-Schwarz
inequality, we have
\begin{align*}
M(T_0(R))&=\int_0^{T_0(R)}M'(s)\,ds\le \sqrt{2}c_3R^{1-\alpha/2}\int_0^{T_0(R)}Q'(s)^{1/2}\,ds\\
&\le \sqrt{2} c_3R^{1-\alpha/2}T_0(R)^{1/2}\left(\int_0^{T_0(R)}Q'(s)\,ds\right)^{1/2}\\
&\le c_7R^{1-\alpha(1-\theta_0)/2}\big(Q(T_0(R))-(Q(0)\wedge0)\big)^{1/2},
\end{align*}
where in the last inequality we have used
the fact that $T_0(R)\le 2R^{\theta_0
\alpha}$.
By the definition of $T_0(R)$, it holds that
$K(T_0(R))=0$, and so
$Q(T_0(R))=({d}/{\alpha})\log T_0(R)-c_4\le c_8(1+\log R),$
where we have used again $T_0(R)\le 2R^{\theta_0 \alpha}$.
 On the other hand,
$Q(0)=\lim_{t\to0} Q(t)=\log \mu_x\ge -\log c_M.$
Thus, we can find
$R_1\ge 1$ large enough such that for all $R>R_1$ and $t \in
[R^{\theta'\alpha},R^{\alpha}]$,
\begin{align*}
M(T_0(R))&\le c_9R^{1-\alpha(1-\theta_0)/2}
(1+\log R)^{1/2}= c_9R^{1-\alpha/2}R^{\theta_0\alpha/2}
(1+\log R)^{1/2}\\
&\le c_9R^{1-\alpha/2}R^{\theta'\alpha/2}\le c_9R^{1-\alpha/2}t^{1/2},
\end{align*}
where in the second inequality we used the fact that $\theta_0\in (\theta,\theta')$, and the last inequality is due to $t\ge R^{\theta'\alpha}$. Note
that $M(0)=0$, so the above estimate still holds when $T_0(R)=0$.

Therefore, combining this with \eqref{np1-8}, we arrive at the statement that
for all $t \in [R^{\theta'\alpha},R^{\alpha}]$,
\begin{equation}\label{np1-9}
M(t)\le c_{10}R^{1-\alpha/2}t^{1/2}\big(1+K(t)\big).
\end{equation}

{\bf Step (3):} Note that $s(\log s+t)\ge -e^{-1-t}$ for all $s>0$ and $t\in \R$. Then,
for every $0<a\le 2$, $b\in \R$ and $t>0$,
\begin{equation}\label{np1-10a}
\begin{split}
-Q(t)+aM(t)+b&=\sum_{y \in V}f_t(y)\big(\log f_t(y)+a\rho(x,y)+b\big)\mu_y\\
&\ge -\sum_{y \in V}\exp\big(-1-a\rho(x,y)-b\big)\mu_y\ge -c_{11}e^{-b}a^{-d},
\end{split}
\end{equation}
where
the equality above follows from the conservativeness of $X^{R,R}$,
and in the last inequality we used the fact that
\begin{align*}
\sum_{y \in V}e^{-a \rho(x,y)}\mu_y&\le c_M+
\sum_{k=1}^{\infty}\sum_{y\in B(x,2^k)\setminus B(x,2^{k-1})}e^{-a2^{k-1}}\mu_y\le  c_M+c_G\sum_{k=1}^{\infty}2^{d k}e^{-a2^{k-1}}
\le Ca^{-d}
\end{align*}
for all
$0<a\le 2$ (see \cite[line 6--7 in p.\ 3056]{B}).

According to \eqref{np1-2}, we could find $R_1>\tilde R_0$ large enough such that for all
$R_1<R<r_G$ and $t\in [R^{\theta'\alpha},R^{\alpha}]$,
\begin{align*}
M(t)&=\sum_{y\in V}\rho(x,y)f_t(y)\mu_y
\ge \sum_{y \in V: \rho(x,y)>0}f_t(y)\mu_y= 1-\Pp_x\big(\hat X_t^{R,R}=x\big)\\
&\ge 1-c_2t^{-d/\alpha}\ge 1-c_2R^{-\theta'd}>1/2.
\end{align*}
Then, choosing $a=1/M(t)$ and $e^b=M(t)^d=a^{-d}$ in
\eqref{np1-10a}, we have
$
-Q(t)+1+d\log M(t)\ge -c_{11},
$
which implies that for all $R_1<R<r_G$ and $t \in [R^{\theta'\alpha},R^{\alpha}]$,
$
M(t)\ge c_{12}\exp(Q(t)/d).
$
This along with the definition of $K(t)$ yields that
\begin{equation}\label{np1-10}
M(t)\ge c_{12}\exp(Q(t)/d)\ge c_{13}t^{1/\alpha}e^{K(t)}.
\end{equation}
Combining \eqref{np1-9} with \eqref{np1-10}, we obtain that for all
$t \in [R^{\theta'\alpha},R^{\alpha}]$,
$$
e^{K(t)}\le c_{14}R^{1-\alpha/2}\big(1+K(t)\big)t^{1/2-1/\alpha}, $$
which is equivalent to
$$
K(t)\le
c_{15}\left[1+\log\left(\frac{R^{\alpha}}{t}\right)+\log(1+K(t))\right].
$$
This implies that for all $R_1<R<r_G$ and $t \in [R^{\theta'\alpha},R^{\alpha}]$,
$$
K(t)\le c_{16}\left[1+\log\left(\frac{R^{\alpha}}{t}\right)\right].
$$
The inequality above along with \eqref{np1-9} further gives us that for all $R_1<R<r_G$ and $t \in [R^{\theta'\alpha},R^{\alpha}]$,
$$M(t)\le
c_{17}R^{1-\alpha/2}t^{1/2}\left[1+\log\left(\frac{R^{\alpha}}{t}\right)
\right]\le c_{18}R\left(\frac{t}{R^{\alpha}}\right)^{1/2}\left[1+\log\left(\frac{R^{\alpha}}{t}\right)
\right].
$$
The proof is complete.
\end{proof}

\section{
$\alpha$-stable-like processes on graphs} Let $(D,\F)$ be a regular
symmetric Dirichlet form on $L^2(V;\mu)$ given in the beginning of
Section \ref{S:tr}, i.e.,
\begin{align*}
D(f,f)&=\frac{1}{2}\sum_{x,y\in
V}(f(x)-f(y))^2\frac{w_{x,y}}{\rho(x,y)^{d+\alpha}}\mu_x\mu_y,\quad
f\in \F=\{f\in L^2(V;\mu): D(f,f)<\infty\},
\end{align*}
where $\alpha\in (0,2)$ and $\{w_{x,y}:x,y\in V\}$ is a sequence
such that $w_{x,x}=0$ for all $x\in V$, $w_{x,y}\ge0$ and
$w_{x,y}=w_{y,x}$ for all $x\neq y$, and \eqref{eq:condwxy} holds.
Let $X:=((X_t)_{t\ge0}, (\Pp_x)_{x\in V})$ be the associated
symmetric $\alpha$-stable-like process associated with $(D,\F)$.

In this section, we will derive exit time estimates for the process
$X$ and the H\"{o}lder regularity of the associated
caloric
functions. Both statements are crucial to establish the weak
convergence of $\alpha$-stable-like processes in the next section.

\subsection{Estimates of exit time: for any fixed starting point}
 In this part, we are concerned on exit
time estimates of the process $X$ for any fixed starting point.
The main statement is as follows.

\begin{proposition}\label{np-2}
Assume that there exist $R_0\ge 1$, $\theta\in (0,1)$ and $C_1>0$
such that
for every $R_0<R<r_G$ and $R^{\theta}\le r \le R$,
\eqref{np1-1}, \eqref{np1-1b} and \eqref{np1-1a} as well as
\begin{equation}\label{np2-1a}
\sup_{x \in B(x_0, R)} \sum_{y\in
V:\rho(x,y)>R}\frac{w_{x,y}}{\rho(x,y)^{d+\alpha}}\le C_1R^{-\alpha}
\end{equation} hold. Then
\begin{itemize}
\item[(i)] for any $\theta'\in (\theta,1)$, there exist constants $R_1\ge 1$ $($which depends only on $\theta$, $\theta'$ and $R_0$
$)$ and
$C_2>0$ $($which is independent of $x_0$, $R_0$ and $R_1$
$)$ such that for
every $R_1<R<r_G$,
\begin{equation}\label{np2-2a}
\Pp_{x_0}\big(\tau_{B(x_0, R)}\le t\big)\le
C_2\left(\frac{t}{R^{\alpha}}\right)^{1/2}\left[1\vee\log\left(\frac{R^{\alpha}}{t}\right)
\right],\quad \forall\ t\ge R^{\theta'\alpha}.
\end{equation}
\item[(ii)] for any $\varepsilon>0$,
there exist constants $R_2\ge1$ $($which depends only on $\theta$, $R_0$
and $\varepsilon)$ and $C_3(\varepsilon)>0$
$($which is independent of
$x_0$, $R_0$ and $R_2$
$)$
such that for all $R_2<R<r_G$,
\begin{equation}\label{np2-2} \Pp_{x_0}\big(\tau_{B(x_0, R)}\le
t\big)\le \varepsilon+\frac{C_3(\varepsilon)t}{R^{\alpha}},\quad
\forall\ t>0.
\end{equation}
In particular, the
process $X$ is conservative.
\end{itemize}
\end{proposition}

To prove Proposition \ref{np-2}, we will make use of Proposition
\ref{np-1}.
We adopt notations from Subsection \ref{Subsection2.2}.
Fix $x_0\in V$ and $R\ge1$. According to the definition of $(\hat
D^{x_0,R},\hat \F^{x_0, R})$, we have
\begin{equation}\label{ne2-1}
\Pp_{x_0}\big(\tau_{B(x_0,R)}\le
t\big)=\Pp_{x_0}\big(\hat\tau^{R}_{B(x_0,R)}\le
t\big),
\end{equation}
where $\tau_A:=\inf\{t>0: X_t\notin A\}$ and $\hat \tau^{
R}_A:=\inf\{t\ge0: \hat X_t^{R}\notin A\}$ for any subset
$A\subseteq V$.
This is, for fixed $x_0\in V$, the distribution of exit time for the
process $X$ exiting from $B(x_0,R)$ is the same as that for the
corresponding localized process $(X_t^{R})_{t\ge0}$ (with parameters
$x_0$ and $R$).

In order to apply Proposition \ref{np-1}
and obtain
$\Pp_{x_0}\big(\hat\tau^{R}_{B(x_0,R)}\le t\big)$, we
now use the truncation idea.
In the following, we denote by $(\hat P_t^{R,B(x_0, R)})_{t\ge 0}$
and $(\hat P_t^{R,R,B(x_0,R)})_{t \ge 0}$ Dirichlet semigroups of
the processes $\hat X^{R}$ and $\hat X^{R,R}$ exiting $B(x_0, R)$,
respectively. Let $\hat\tau^{R,R}_{A}=\inf\{t\ge0: \hat
X_t^{R,R}\notin A\}$ for any $A\subseteq V$.

\begin{lemma}\label{nl-1}
There exists a constant $C_1>0$
such that for every $f\in
L^2(V;\mu)$, $t>0$ and $x\in B(x_0,R)$,
\begin{equation}\label{nl1-1}\begin{split}
|\hat P_t^{R, R,B(x_0,R)}f(x)-&\hat P_t^{
R,B(x_0,R)}f(x)|\le C_1t \left(\sup_{y\in B(x_0,R)}J(y, R)\right)\left(\sup_{z\in B(x_0,R)}|f(z)|\right),\end{split}
\end{equation}
where
\begin{equation}\label{nl-1-00}
J(y,R)=\sum_{z\in V:
\rho(y,z)>R}\frac{ w_{y,z}}{\rho(y,z)^{d+\alpha}}\mu_z,\quad y\in B(x_0,R).
\end{equation}
In particular, it holds that for any $t>0$ and $x\in B(x_0,R)$,
\begin{equation}\label{nl1-2}
\big|\Pp_x\big(\hat\tau_{B(x_0,R)}^{R, R}\le
t\big)-\Pp_x\big(\hat\tau^{R}_{B(x_0,R)}\le t\big)\big|
\le C_1t\sup_{y\in B(x_0,R)}J(y,R).
\end{equation}
\end{lemma}
\begin{proof}
Let $T_{R}^{R}=\inf\{t>0: \rho(\hat X_{t-}^{R},\hat X_t^{R})>R\}$.
By \eqref{e:bound}, $\sup_{y\in V}\sum_{z \in
V:\rho(z,y)>R}\frac{\hat
w_{z,y}}{\rho(z,y)^{d+\alpha}}\mu_z<\infty.$ Then, by Meyer's
construction of $\hat X^{R}$ (see \cite[Section 3.1]{BGK}),
$\hat X_t^{R}$ and $\hat X_t^{R, R}$ enjoy the same distribution
if $t<T_{R}^{R}$.
Hence, for any $f\in
L^2(V;\mu)$,
\begin{align*}
&\big|\hat P_t^{R, R,B(x_0,R)}f(x)-\hat P_t^{R,B(x_0,R)}f(x)\big|\\
&=\big|\Ee_x(f(\hat X_t^{R}):t\le \hat \tau_{B(x_0,
R)}^{R} )- \Ee_x(f(\hat X_t^{R, R}):t\le \hat
\tau_{B(x_0,R)}^{R,R} )\big|\\
&\le \sup_{z\in B(x_0,R)}|f(z)|\Big[\Pp_x\big(T_{R}^{R}\le t \le
\hat\tau_{B(x_0,R)}^{R}\big)+
\Pp_x\big(T_{R}^{R}\le t \le \hat \tau_{B(x_0,R)}^{R,R}\big)\Big]\\
&\le 2\left(\sup_{z\in B(x_0,R)}|f(z)|\right)\Pp_x\big(T_{R}^{R}\le t, \hat X^{
R, R}_{s}\in B(x_0,R)\ \text{for all } s\in [0,T_{R}^{
R}]\big).
\end{align*}
According to \cite[Lemma 3.1(a)]{BGK},
$$\Pp_x\Big(T_{R}^{R}\in dt\big|\mathscr{F}^{\hat X^{R, R}}\Big)=
\hat J(\hat X_t^{R, R},R)\exp\left(-\int_0^t \hat J(\hat
X_s^{R, R},R)\,ds\right)\,dt,$$ where $\F^{\hat X^{
R,R}}$ denotes the $\sigma$-algebra generated by $\hat X^{
R,R}$, and
$$\hat J(y,R)=\sum_{z\in V: \rho(y,z)>R}\frac{
\hat w_{y,z}}{\rho(y,z)^{d+\alpha}}\mu_z,\quad y\in B(x_0,R).$$ In
particular, by the definition of $\hat w_{x,y}$, $J(y,R)=\hat
J(y,R)$ for all $y\in B(x_0,R).$ Therefore,
\begin{align*}
&\Pp_x\Big(T_{R}^{ R}\le t,  \hat X^{R, R}_{s}\in
B(x_0,  R)\ \text{for all }
s\in [0,T_{R}^{ R}]\Big)\\
&\le \Ee_x\left[\int_0^t J(\hat X_r^{ R,
R},R)\exp\left(-\int_0^r J(\hat X_s^{ R, R},R)\,ds
\right)\I_{\{\hat X_s^{R, R}\in B(x_0, R)\text{ for all } s\in [0,r]\}}\,dr\right]\\
&\le c_1t\sup_{y\in B(x_0, R)}J(y,R).
\end{align*}
Combining all the estimates above, we can obtain \eqref{nl1-1}.
\eqref{nl1-2} is a direct consequence of \eqref{nl1-1} by taking
$f\equiv1$ on $B(x_0,R)$.
\end{proof}

\begin{proof}[Proof of Proposition $\ref{np-2}$]
{\bf Step (1):} It immediately follows from  \eqref{np2-1a} that
\begin{equation}\label{np2-3}
\sup_{y\in B(x_0, R)}J(y,R)\le c_1R^{-\alpha},
\end{equation}
where $J(y,R)$ is defined by \eqref{nl-1-00}.

Since \eqref{np1-1}, \eqref{np1-1b} and \eqref{np1-1a} are true,
by \eqref{np1-3}, for any $\theta' \in (\theta,1)$, there
is a constant $\tilde R_1\ge 1$ such that for all $\tilde R_1<R<r_G$ and $x\in V$,
$$
\Ee_x\big[\rho(\hat X_t^{R,R},x)\big]\le
c_2R\Big(\frac{t}{R^{\alpha}}\Big)^{1/2}\left[1+\log\left(\frac{R^{\alpha}}{t}\right)
\right],\quad \forall \
R^{\theta'\alpha}\le t \le R^{\alpha}.
$$
Hence, by the Markov inequality, for all $x\in V$, $\tilde R_1<R<r_G$ and
$R^{\theta'\alpha}\le t \le R^{\alpha}/2$,
$$
\sup_{s \in [t,2t]}\Pp_x\Big( \rho\big(\hat X_{s}^{
R,R},x\big)>\frac{R}{2} \Big)\le
c_3\Big(\frac{t}{R^{\alpha}}\Big)^{1/2}\left[1+\log\left(\frac{R^{\alpha}}{t}\right)
\right].
$$
Therefore, for all $\tilde R_1<R<r_G$ and $R^{\theta'\alpha}\le t \le R^{\alpha}/2$,
\begin{align*}
\Pp_{x_0}\big(\hat\tau_{B(x_0,R)}^{R, R}\le t\big)
&\le  \Pp_{x_0}\Big(\hat\tau_{B(x_0,R)}^{R,R}\le t;
\rho\big(\hat X_{2t}^{R, R},x_0\big)\le \frac{R}{2} \Big)
+\Pp_{x_0}\Big(\rho\big(\hat X_{2t}^{R, R},x_0\big)> \frac{R}{2} \Big)\\
&\le \Ee_{x_0}\left[\I_{\{\hat \tau^{R, R}_{B(x_0,R)}\le t\}}
\Pp_{\hat X_{\hat\tau_{B(x_0,R)}^{
R,R}}^{R,R}} \Big(\rho\big(\hat X_{2t-\hat \tau^{R,
R}_{B(x_0,R)}}^{ R, R},\hat X_0^{
R,R}\big)>\frac{R}{2}\Big)\right]\\
&\quad +c_3\Big(\frac{t}{R^{\alpha}}\Big)^{1/2}\left[1+\log\left(\frac{R^{\alpha}}{t}\right)
\right]\\
&\le \sup_{y \in V}\sup_{s \in [t,2t]}\Pp_{y}\Big(\rho\big(\hat
X_{s}^{R,R},y\big)
>\frac{R}{2}\Big)+c_3\Big(\frac{t}{R^{\alpha}}\Big)^{1/2}
\left[1+\log\left(\frac{R^{\alpha}}{t}\right)
\right]\\
&\le
2c_3\Big(\frac{t}{R^{\alpha}}\Big)^{1/2}\left[1+\log\left(\frac{R^{\alpha}}{t}\right)
\right].
\end{align*}

Combining this with \eqref{ne2-1}, \eqref{nl1-2} and \eqref{np2-3}
yields that for all $\tilde R_1<R<r_G$ and $R^{\theta'\alpha}\le t \le
R^{\alpha}/2$,
$$
\Pp_{x_0}\big(\tau_{B(x_0,R)}\le t\big)\le
2c_3\Big(\frac{t}{R^{\alpha}}\Big)^{1/2}\left[1+\log\left(\frac{R^{\alpha}}{t}\right)
\right]+\frac{c_4
t}{R^{\alpha}}\le
c_5\Big(\frac{t}{R^{\alpha}}\Big)^{1/2}\left[1\vee\log\left(\frac{R^{\alpha}}{t}\right)
\right].
$$
Thus, \eqref{np2-2a} has been verified for all $R^{\theta'\alpha}\le
t \le R^{\alpha}/2$. When $t>R^{\alpha}/2$, it
holds that
$$
\Pp_{x_0}\big(\tau_{B(x_0,R)}\le t\big)\le 1 \le
\Big(\frac{2t}{R^{\alpha}}\Big)^{1/2}\left[1\vee\log\left(\frac{R^{\alpha}}{t}\right)
\right].
$$
Hence we prove \eqref{np2-2a}.

{\bf Step (2):}
Fix $\theta'\in (\theta,1)$.
By \eqref{np2-2a} and Young's inequality, there is a constant $\tilde R_1\ge 1$ such that for every $\tilde
R_1<R<r_G$, $t\ge R^{\theta'\alpha}$ and $\varepsilon>0$,
$
\Pp_{x_0}\big(\tau_{B(x_0, R)}\le t\big)\le
2^{-1}\varepsilon+{c_6(\varepsilon) t}{R^{-\alpha}}.
$
If $0<t\le R^{\theta'\alpha}$, then, taking
$\tilde R_2(\varepsilon)\ge \tilde R_1$ large enough, we obtain that for all $\tilde R_2(\varepsilon)\le R<r_G$,
$\Pp_{x_0}\big(\tau_{B(x_0, R)}\le t\big)\le
\Pp_{x_0}\big(\tau_{B(x_0, R)}\le R^{\theta'\alpha}\big)\le
2^{-1}\varepsilon+c_6(\varepsilon)R^{-(1-\theta')\alpha}\le \varepsilon.
$
Combining both estimates above together, we know that for all  $\tilde R_2(\varepsilon)<R<r_G$ and $t>0$,
$
\Pp_{x_0}\big(\tau_{B(x_0, R)}\le t\big)\le \varepsilon+{c_7(\varepsilon)t}{R^{-\alpha}},
$
which implies that \eqref{np2-2} holds.
\end{proof}

\subsection{Estimates of exit time:
locally uniform with respect to the starting
point}\label{subsection3.2}
For our later use, we need exit time estimates for the process,
which are locally uniform with respect to the starting point. We
first present the following assumption on $\{w_{x,y}:x,y\in V\}$,
which is regarded as the locally uniform version of assumptions in Proposition
\ref{np-2}.
For any $x,z\in V$ and $r>0$, denote $B_z^w(x,r):=\{u\in B(x,r):
w_{u,z}>0\}$. In particular, $B_x^w(x,r)=B^w(x,r)$.

\medskip

  \paragraph{{\bf Assumption (Exi.($\theta$))}}
{\it Suppose that for some fixed $\theta\in (0,1)$ and $0\in V$, there exist constants $R_0\ge 1$, $c_0\in (1/2,1)$ and $C_1,C_2>0$
such that the following hold.
\begin{itemize}
 \item[(i)] For
every $R_0<R<r_G$ and ${R^{\theta}/2}\le r \le 2R$,
\begin{equation}\label{a2-2-1}
\sup_{x\in B(0,6R)}\sum_{y\in V:\rho(x,y)\le  r}
\frac{w_{x,y}}{\rho(x,y)^{d+\alpha-2}}\le C_1 r^{2-\alpha},
\end{equation}
\begin{equation}\label{a2-2-1a}
\mu(B_z^w(x,r))\ge c_0\mu(B(x,r)),\quad \forall x,z\in B(0,6R)
\end{equation}
and
\begin{equation}\label{a2-2-2}
\sup_{x\in B(0,6R)}\sum_{y\in B^w(x,
c_*r)}w_{x,y}^{-1}\le C_1r^{d},
\end{equation}
where $c_*:=8c_G^{2/d}$.
\item[(ii)] For every $R_0<R<r_G$ and $r\ge {R^{\theta}/2}$,
\begin{equation}\label{a2-2-3}
\begin{split}
\sup_{x \in B(0,6R)}\sum_{y \in V:
\rho(x,y)>r}\frac{w_{x,y}}{\rho(x,y)^{d+\alpha}}\le C_1r^{-\alpha}.
\end{split}
\end{equation}
\end{itemize}
}

Then, we have the following statement.

\begin{theorem} \label{exit}
Suppose
that Assumption {\bf (Exi.($\theta$))} holds {with some constant $\theta \in (0,1)$}. Then, for every $\theta'\in (\theta,1)$, there exist
constants $R_1\ge1$, {$\delta \in (\theta,1)$} and $C_0,C_1,C_2>0$
such that for all
$R_1<R<r_G/(2c_*)$ and $ R^{\delta}\le r \le R$,
\begin{itemize}
\item[(1)]
\begin{equation}\label{l2-2-1a}
\sup_{x \in B(0,2R)}\Pp_x\big(\tau_{B(x,r)}\le C_0r^{\alpha}\big)\le
\frac{1}{4}.
\end{equation}
\item[(2)]
\begin{equation}\label{l2-2-0a}
\begin{split}
\sup_{x \in B(0,2R)}\Pp_x\big(\tau_{B(x,r)}\le t \big)\le
C_1\Big(\frac{ t}{r^{\alpha}}\Big)^{1/2}\Big[1\vee
\log\Big(\frac{r^{\alpha}}{t}\Big)\Big],\quad \forall\ t\ge
r^{\theta'\alpha},
\end{split}
\end{equation}
and
\begin{equation}\label{l2-2-1}
\begin{split}
C_2r^{\alpha}\le \inf_{x \in B(0,2R)}\Ee_x\big[\tau_{B(x,r)}\big]
\le \sup_{x \in B(0,2R)}\Ee_x\big[\tau_{B(x,r)}\big]\le
C_1r^{\alpha}.
\end{split}
\end{equation}
\end{itemize}
\end{theorem}

To prove Theorem \ref{exit}, we begin with the following simple
lemma.

\begin{lemma}\label{l2-3}
Suppose for some constant $\theta\in (0,1)$, \eqref{a2-2-1a} and \eqref{a2-2-2} in Assumption
{\bf (Exi.($\theta$))}{\rm(i)} hold. Then there exists a constant
$C_1>0$, independent of $R_0$,
such that for every $R_0<R<r_G/(2c_*)$
and ${R^{\theta}/2}\le r \le 2R$,
\begin{equation}\label{a2-2-3a}
\inf_{x \in B(0,6R)}\sum_{y \in V: \rho(x,y)>3r}\frac{w_{x,y}}{\rho(x,y)^{d+\alpha}}\ge
C_1r^{-\alpha}.
\end{equation}
Here $c_*$ is the constant in Assumption {\bf (Exi.($\theta$))}{\rm(i)}.
\end{lemma}
\begin{proof}
Noting that {$c_*>4$},
for every $x\in V$ and
$1\le r<r_G/c_*$, we have
\begin{align*}
\sum_{y\in V: 3r<\rho(x,y)\le c_*r, w_{x,y}>0}\mu_y&
\ge
\mu(B^w(x,c_*r))-\mu(B(x,4r))\ge c_0c_G^{-1}(c_*r)^{d}-c_G(4r)^{d}
\ge c_1r^{d},
\end{align*}
where we have used \eqref{al2-1} and \eqref{a2-2-1a}.

On the other hand, for every $R_0<R<r_G/(2c_*)$, $x \in B(0,6R)$ and
${R^{\theta}/2}\le r \le 2R$,
\begin{align*}
\sum_{y\in V: 3r<\rho(x,y)\le c_*r,w_{x,y}>0}\mu_y&\le
\Big(\sum_{y\in B^w(x,c^*r)}w_{x,y}^{-1}\mu_y\Big)^{1/2}
\Big(\sum_{y\in V: 3r<\rho(x,y)\le c_*r}w_{x,y}\mu_y\Big)^{1/2}\\
&\le c_2r^{d/2}\Big(\sum_{y\in V: 3r<\rho(x,y)\le c_*r}w_{x,y}\Big)^{1/2},
\end{align*} where in the first inequality we have applied the Cauchy-Schwarz inequality, and we used \eqref{a2-2-2} in the last inequality.

Combining the previous two estimates together yields the statement that
for every $R_0<R<r_G/(2c_*)$, $x \in B(0,6R)$ and
${R^{\theta}/2}\le r \le 2R$,
$\sum_{y\in V: 3r<\rho(x,y)\le c_*r}w_{x,y}\ge c_3 r^{d},$ and so
$$
\sum_{y\in V: \rho(x,y)>3r}\frac{w_{x,y}}{\rho(x,y)^{d+\alpha}}\!\!\!
\ge \sum_{y\in V: 3r<\rho(x,y)\le c_*r}\frac{w_{x,y}}{\rho(x,y)^{d+\alpha}}\ge (c_*r)^{-d-\alpha}\!\!\!\sum_{y\in V: 3r<\rho(x,y)\le c_*r}w_{x,y}\ge c_4r^{-\alpha}.
$$
Thus, \eqref{a2-2-3a} is proved.
\end{proof}

\begin{proof}[Proof of Theorem $\ref{exit}$]
Suppose that Assumption {\bf(Exi.($\theta$))} holds with some $\theta\in
(0,1)$ and $R_0\ge 1$. Then, for any $\theta<\theta_1<\theta'<1$,
$R_0<R<r_G$ and $R^{\delta}\le s \le R$ with
$\delta=\theta/\theta_1$, we know that \eqref{np1-1}, \eqref{np1-1a}
and \eqref{np2-1a} hold uniformly (that is, they hold with uniform
constants) for every ${s^{\theta_1}}\le r \le s$ and $x_0\in
B(0,2R)$. Hence, according to \eqref{np2-2a} and \eqref{np2-2}, we
obtain that for every $\theta'\in (\theta,1)$, there exists a
constant $R_1\ge R_0$ such that for each $R_1<R<r_G$ and
${R^{\delta}}\le r \le R$, \eqref{l2-2-0a} and
\begin{equation}\label{l2-2-2}
\sup_{x \in B(0,2R)}\Pp_x\big(\tau_{B(x,r)}\le t \big)\le
\frac{1}{8}+\frac{c_1t}{r^{\alpha}}, \quad \ \forall\ t>0
\end{equation} hold true.
In particular, taking $t=(8c_1)^{-1}r^{\alpha}$ in \eqref{l2-2-2},
we get \eqref{l2-2-1a} immediately.

Let $C_0$ be the constant in \eqref{l2-2-1a}. For any $R>R_1$, $x\in
B(0, 2R)$ and ${R^{\delta}}\le r \le R$, we have
\begin{align*} \Ee_x[\tau_{B(x,r)}]
&= \int_0^{\infty} \Pp_x(\tau_{B(x,r)}> s)\,ds\ge  \int_0^{C_0r^{\alpha}} \Pp_x(\tau_{B(x,r)}> s)\,ds\\
&\ge C_0r^{\alpha} \Pp_x(\tau_{B(x,r)}>C_0r^{\alpha})\ge
\frac{3C_0r^{\alpha}}{4}.
\end{align*} This gives us the first inequality in \eqref{l2-2-1}.
On the other hand, let $c_*$ be the constant in Assumption {\bf
(Exi.($\theta$))}{\rm(i)}. By the L\'evy system (see \cite[Appendix
A]{CK08}), for any $R_1<R<r_G/(2c_*)$, $x\in B(0,2R)$ and
${R^{\delta}}\le r \le R$,
\begin{align*}
1&\ge \Pp_x\big(X_{\tau_{B(x,r)}}\notin B(x,2r)\big)=
\Ee_x\left[\int_0^{\tau_{B(x,r)}}\sum_{y\in V:\rho(x,y)>2r}\frac{w_{X_s,y}}{\rho(X_s,y)^{d+\alpha}}\mu_y\,ds\right]\\
&\ge c_M^{-1}\Ee_x\left[\int_0^{\tau_{B(x,r)}}\sum_{y\in V: \rho(y,X_s)> 3r}\frac{w_{X_s,y}}{\rho(X_s,y)^{d+\alpha}}\,ds\right]\\
&\ge c_M^{-1}\left(\inf_{v\in B(0,2R+r)}
\sum_{y\in V: \rho(y,v)> 3r}\frac{w_{v,y}}{\rho(v,y)^{d+\alpha}}\right) \Ee_x[\tau_{B(x,r)}]\ge c_2r^{-\alpha} \Ee_x[\tau_{B(x,r)}], \end{align*} where in the
last inequality we have used \eqref{a2-2-3a}, also thanks to the fact that $\delta=\theta/\theta_1>\theta$.
Thus, we also prove the third inequality in \eqref{l2-2-1}.
\end{proof}

Finally, we remark that, when $\alpha\in(0,1)$
one can obtain the probability estimate \eqref{np2-2} for the exit
time in a more direct way under the following assumption.

 \paragraph{{\bf Assumption (Exi.($\theta$)')}}
 {\it Suppose that for some fixed $\theta\in (0,1)$ and $0\in V,$ there exist constants $R_0\ge 1$ and $C_1>0$
 such that
\begin{itemize}
 \item[(i)]  for
every $R_0<R<r_G$ and ${R^{\theta}/2}\le r\le 2R$, \begin{equation}\label{l2-1-0}
\sup_{x\in B(0,6R)}\sum_{y\in V:\rho(x,y)\le  r}
\frac{w_{x,y}}{\rho(x,y)^{d+\alpha-1}}\le C_1 r^{1-\alpha}
\end{equation} and \eqref{a2-2-2} hold.
\item[(ii)] ${\rm(ii)}$ in Assumption {\bf(Exi.($\theta$))} is satisfied.
\end{itemize}
}

\begin{proposition}\label{L:tight}
Under \eqref{l2-1-0} and ${\rm(ii)}$ in Assumption {\bf(Exi.($\theta$))},
there exist constants $R_1>R_0$ and
$C_1>0$,
which are independent of $R_0$,
such that for all $R_1<R<r_G$, $x\in B(0,2R)$, ${R^{\theta}}\le r\le R$ and $t>0$,
\begin{equation}\label{l2-1-1}
\Pp_x(\tau_{B(x,r)}\le t)\le \frac{
C_1t}{r^{\alpha}}.
\end{equation}
\end{proposition}
\begin{proof} Fix $x\in B(0,2R)$.
Given $f\in C_b^1([0,\infty))$ with $f(0)=0$ and $f(u)=1$ for all
$u\ge1$, we set
$f_{x,r}(z)=f\left(\frac{\rho(z,x)}{r}\right)$ for any $ z\in V$ and $ r>0.$
For any $r>0$,
$$\left\{f_{x,r}(X_t)-f_{x,r}(X_0)-\int_0^t Lf_{x,r}(X_s)\,ds, t\ge0\right\} $$ is a local martingale.
Then, for any $t>0$ and $x \in V$,
\begin{align*}\Pp_x(\tau_{B(x,r)}\le t)\le
&\Ee_x f_{x,r}(X_{t\wedge \tau_{B(x,r)}})\!=\!\Ee_x\left[\int_0^{t\wedge \tau_{B(x,r)}}Lf_{x,r}(X_s)\,ds\right]\le t\sup_{z\in B(x,r)}Lf_{x,r}(z),\end{align*} where we used the
fact that $f_{x,r}(x)=0$ in the equality above.

Furthermore, for any $x \in V$ and $z\in B(x,r)$,
\begin{align*} L f_{x,r}(z)=&\sum_{y\in V} \big(f_{x,r}(y)-f_{x,r}(z)\big)\frac{w_{y,z}}{\rho(z,y)^{d+\alpha}}\mu_y\\
=&\sum_{y\in V: \rho(y,z)\le r} \left(f_{x,r}(y)-f_{x,r}(z)\right) \frac{w_{y,z}}{\rho(y,z)^{d+\alpha}}\mu_y\\
&+ \sum_{y\in V: \rho(y,z)> r} \left(f_{x,r}(y)-f_{x,r}(z)\right)\frac{w_{y,z}}{\rho(y,z)^{d+\alpha}}\mu_y\\
\le&c_1\left(r^{-1}\sum_{y\in V: \rho(z,y)\le
r}\frac{w_{y,z}}{\rho(y,z)^{d+\alpha-1}}
+\sum_{y \in V: \rho(z,y)>r}\frac{w_{y,z}}{\rho(y,z)^{d+\alpha}}\right)=:c_1(I_1(z,r)+I_2(z,r)),
\end{align*}
where in the first inequality above we have used
$|f_{x,r}(y)-f_{x,r}(z)|\le c_1r^{-1}\rho(y,z).$
According to \eqref{l2-1-0} and
\eqref{a2-2-3}, we can find a constant $R_1\ge1$ such that for all
$R_1<R<r_G$, $x \in B(0,2R)$ and $R^{\theta}\le r \le R$,
$
\sup_{z \in B(x,r)}\big(I_1(z,r)+I_2(z,r)\big)\le c_2r^{-\alpha}.
$

Combining with all estimates above, we prove the desired assertion.
\end{proof}

\subsection{H\"older regularity}\label{subsection3.3}
Let $\R_+:=[0,\infty)$
 and  $Z:=(Z_t)_{t\ge0}=(U_t,X_t)_{t\ge0}$ be
the time-space process such that $U_t=U_0+t$ for any $t\ge0$. Denote
by $\Pp_{(s,x)}$ the probability of the process $Z$ starting from
$(s,x)\in \R_+\times V$. For any subset $A\subseteq \R_+\times V$,
define $\tau_A=\inf\{s>0:Z_s\in A\}$ and $\sigma_A=\inf\{s>0:Z_s\in
A\}$ (for simplicity of notation we still use $\tau_A$ to denote the exit time
from $A\subset \R_+\times V$ for process $(Z_t)_{t\ge 0}$).
We say that a measurable function $q(t,x)$ on $\R_+\times V$
is caloric in an open set $A\subseteq \R_+\times V$, if for
every relatively compact open subset $A_1$ of $A$,
$q(t,x)=\Ee_{(t,x)}q(Z_{\tau_{A_1}})$ for every $(t,x)\in A_1$.

For any $t\ge0$, $x\in V$ and $R\ge1$, let $Q(t,x,R)=\left(t, t+
C_0R^{\alpha}\right)\times B(x,R)$ and $ d\nu=ds\times d\mu$, where
$C_0$ is the constant in \eqref{l2-2-1a}. In the following, let
$c_*$ and $\theta$ be the constants in Assumption {\bf
(Exi.($\theta$))}{\rm(i)}.
The main result in this part is the following theorem.
 \begin{theorem}\label{T:holder}
Suppose that Assumption {\bf (Exi.($\theta$))} holds {with some $\theta \in
(0,1)$}.
Then, there exist constants $R_1\ge1$, {$\delta
\in (\theta,1)$}, $\beta\in (0,1)$ and $C_1>0$,
independent of  $R_0$, $R_1$ and $x_0$ such that for every $R_1<R<r_G/(2c_*)$, $x_0 \in B(0,R)$, ${R^{\delta}}\le r \le R$, $t_0\ge0$ and caloric
  function $q$
 on $Q(t_0,x_0, 2r)$,
 \begin{equation}\label{t3-2-1}
 |q(s,x)-q(t,y)|\le C_1\|q\|_{\infty,r}\left(\frac{|t-s|^{1/\alpha}+\rho(x,y)}{r} \right)^\beta,
 \end{equation}
 holds for all $(s,x),(t,y)\in Q(t_0,x_0, r)$ such that
 $(C_0^{-1}|s-t|)^{1/\alpha}+\rho(x,y)\ge {2r^{\delta}}$, where
 $\|q\|_{\infty,r}=\sup_{(s,x)\in [t_0, t_0+C_0(2r)^\alpha]\times V}q(s,x)$.
\end{theorem}
\begin{remark}\label{remhold}
Note that unlike the case of random walks
on the supercritical
percolation cluster (\cite[Proposition 3.2]{BaHa}), in which the
H\"older regularity holds for all points in the parabolic cylinder
when $r$ is large enough, in the preset setting we can only obtain
the H\"older regularity in the region
$(C_0^{-1}|s-t|)^{1/\alpha}+\rho(x,y)\ge 2r^{\delta}$ inside the
cylinder.
\end{remark}

To prove Theorem \ref{T:holder}, we need the following Krylov-type
estimates.
\begin{proposition}\label{Kr}
If Assumption {\bf (Exi.($\theta$))}
holds with some $\theta\in (0,1)$, then
there exist constants  $R_1\ge1$, {$\delta \in (\theta,1)$} and $C_1>0$,
independent of $R_0$ and $R_1$
such that for any $R_1<R<r_G/(2c_*)$, $2R^{\delta}\le r \le R$, $x\in
B(0,2R)$, $t\ge0$ and $A\subseteq Q(t,x,r/2)$ with
$\frac{\nu(A)}{\nu(Q(t,x,r/2))}\ge {1}/{2}$, it holds that
\begin{equation}\label{p2-2-1a}
\Pp_{(t,x)}(\sigma_A<\tau_{Q(t,x,r)})\ge C_1.
\end{equation}
\end{proposition}
\begin{proof}
We write $Q_r=Q(t,x,r)$ for simplicity.
Without loss of generality, we may and can assume that $t=0$ and  $\Pp_{(0,x)}(\sigma_A<\tau_{Q_r})\le 1/4$;
otherwise the conclusion holds
trivially. Let $T=\sigma_A\wedge\tau_{Q_r}$ and $A_s=\{y\in V:
(s,y)\in A\}$ for all $s>0$. According to the L\'evy system
(see \cite[Appendix A]{CK08}),
\begin{align*}
\Pp_{(0,x)}(\sigma_A<\tau_{Q_r})&\ge\Ee_{(0,x)}\left(\sum_{s\le T} \I_{\{X_s\neq X_{s-},X_s\in A_s\}}\right)=\Ee_{(0,x)} \left[\int_0^T\sum_{u\in A_s} \frac{w_{X_s,u}}{\rho(X_s,u)^{d+\alpha}}\mu_u\,ds\right]\\
&\ge c_M^{-1}\Ee_{(0,x)} \left[\int_0^{C_0(r/2)^{\alpha}}\sum_{u\in A_s} \frac{w_{X_s,u}}{\rho(X_s,u)^{d+\alpha}}\,ds; T\ge C_0(r/2)^{\alpha}\right]\\
&\ge c_1r^{-d-\alpha}\left( \inf_{z\in B(x,r)}\int_0^{C_0(r/2)^{\alpha}}\sum_{u\in A_s} w_{z,u}\,ds \right)
\Pp_{(0,x)}(T\ge C_0(r/2)^{\alpha}),
\end{align*}
where in the last inequality we have used fact that $\rho(u,z)\le 2r$
for every $u,z\in B(x,r)$.

Furthermore,
according to Theorem \ref{exit}(1),
there exist constants $R_1\ge 1$ and {$\delta \in (\theta,1)$} such that for any $R_1<R<r_G/(2c_*)$,
$R^{\delta} \le r/2 \le R$ and $x \in B(0,2R)$,
\begin{align*}\Pp_{(0,x)}\big(T\ge C_0(r/2)^{\alpha}\big)&=\Pp_{(0,x)}\big(\sigma_A\wedge\tau_{Q_r}\ge C_0(r/2)^{\alpha}\big)\\
 & \ge1- \Pp_{(0,x)}\big(\sigma_A< \tau_{Q_r}\big)-\Pp_{x}\big(\tau_{B(x,r)}\le C_0(r/2)^{\alpha}\big)\ge  1-\frac{1}{4}-\frac{1}{4}\ge \frac{1}{2},\end{align*}
where in the first inequality we have used the fact that
$$
\Pp_{(0,x)}\big(\tau_{Q_r}\le C_0(r/2)^{\alpha}\big)=
\Pp_x\big(\tau_{B(x,r)}\wedge (C_0r^{\alpha})\le  C_0(r/2)^{\alpha}\big)
=\Pp_x\big(\tau_{B(x,r)}\le C_0(r/2)^{\alpha}\big),
$$
and the second inequality follows from \eqref{l2-2-1a}.

On the other hand, let $Q^w_z(t,x,r):=(t, t+C_0r^\alpha)\times B^w_z(x,r)$. Then, for every $R_1<R<r_G$, ${2R^{\delta}} \le r \le R$,
$x\in B(0,2R)$ and $z\in B(x,r)$,
\begin{align*}
\nu(A\cap  Q^w_z(0,x,r/2))&=\int_0^{C_0(r/2)^{\alpha}}\sum_{u \in A_s\cap B^w_z(x,r/2)}\mu_u\, ds\\
&\le \Big(\int_0^{C_0(r/2)^{\alpha}}\sum_{u\in A_s \cap B^w_z(x,r/2)} w_{z,u}^{-1}\mu_u\,ds\Big)^{1/2}
\Big(\int_0^{C_0(r/2)^{\alpha}}\sum_{u\in A_s} w_{z,u}\mu_u\,ds\Big)^{1/2}\\
&\le c_3r^{\alpha/2}\Big(\sum_{u\in B_z^w(x,r)} w_{z,u}^{-1}\Big)^{1/2}
\Big(\int_0^{C_0(r/2)^{\alpha}}\sum_{u\in A_s} w_{z,u}\,ds\Big)^{1/2}\\
&\le c_3r^{\alpha/2}\Big(\sup_{z\in B(0,3R)}\sum_{u\in  B^w(z,2r)}w_{z,u}^{-1}\Big)^{1/2}
\Big(\int_0^{C_0(r/2)^{\alpha}}\sum_{u\in A_s} w_{z,u}\,ds\Big)^{1/2}\\
&\le c_4r^{(d+\alpha)/2}\Big(\int_0^{C_0(r/2)^{\alpha}}\sum_{u\in A_s} w_{z,u}\,ds\Big)^{1/2},
\end{align*}
where in the first inequality we have used the Cauchy-Schwarz
inequality, the third inequality is due to the fact that $B^w_z(x,r)\subset B^w(z,2r)$ for all
$z\in B(x,r)$, and the last inequality follows from
\eqref{a2-2-2}. Note that, by \eqref{a2-2-1a} and the assumption that $\frac{\nu(A)}{\nu(Q(0,x,r/2))}\ge {1}/{2}$, we have
$\nu(A\cap Q_z^w(0,x,r/2))\ge \big(1/2+c_0-1 \big)\cdot \nu(Q(0,x,r/2))\ge c_5r^{d+\alpha}.$
Combining all estimates above yields that for all $R_1<R<r_G$, $2R^{\delta}\le r \le R$, $x\in
B(0,2R)$ and $z \in B(x,r)$,
$
\int_0^{C_0(r/2)^{\alpha}}\sum_{u\in A_s} w_{z,u}\,ds\ge c_6r^{d+\alpha}.
$
According to all the estimates above, we prove the required assertion.
 \end{proof}

We also need the following hitting probability estimate.

\begin{lemma}\label{l3-3}
Suppose that Assumption {\bf (Exi.($\theta$))}\ holds {with some $\theta \in (0,1)$}. Then there are constants $R_1\ge1$, {$\delta\in (\theta,1)$} and $C_1>0$,
independent of $R_0$ and $R_1$
such that for every $R_1<R<r_G/(2c_*)$, {$R^{\delta}\le r
\le R$}, $x\in B(0,2R)$, $K>4r$, $t\ge0$
and $z \in B(x,{r}/{2})$, it holds that
\begin{equation}\label{e:ex}
\begin{split}
\Pp_x(X_{\tau_{Q(t,x,r)}}\notin B(z,K))\le C_1\left(\frac{r}{K}\right)^\alpha.
\end{split}
\end{equation}
\end{lemma}
\begin{proof}
According to the L\'evy system, we know that for every $z \in B(x,r/2)$,
\begin{align*}
\Pp_x(X_{\tau_{Q(t,x,r)}}\notin B(z,K))
&=\Ee_x\left[\int_0^{\tau_{B(x,r)}}\sum_{y\notin B(z,K)} \frac{w_{X_s,y}}{\rho(X_s,y)^{d+\alpha}}\mu_y\,ds \right]\\
&\le c_1\sup_{u\in B(x,r)}\left(\sum_{y\in V: \rho(u,y)>K-2r}\frac{w_{u,y}}{\rho(u,y)^{
d+\alpha}}
\right)\Ee_{x}[\tau_{B(x,r)}]\\
&\le c_1\sup_{u \in B(0,2R)}
\left(\sum_{y\in V: \rho(u,y)>K/2}\frac{w_{u,y}}{\rho(u,y)^{d+\alpha}}
\right)\Ee_{x}[\tau_{B(x,r)}].
\end{align*}
Note that $K/2>2r\ge R^{\delta}$
and $R^{\delta}\le r \le R$.
Then, by \eqref{a2-2-3} and \eqref{l2-2-1}, we can find a constant
$R_1\ge1$ such that for all $R_1<R<r_G/(2c_*)$ and $x\in B(0,2R)$,
$$
\sup_{u \in B(0,2R)}
\left(\sum_{y\in V: \rho(u,y)>K/2}\frac{w_{u,y}}{\rho(u,y)^{d+\alpha}}
\right)\le c_2K^{-\alpha}
$$ and $\Ee_{x}[\tau_{B(x,r)}]\le c_3r^{\alpha}.$
Combining with all the estimates above immediately yields \eqref{e:ex}.
\end{proof}

 \begin{proof}[Proof of Theorem $\ref{T:holder}$]
 We mainly follow the argument of \cite[Theorem 4.14]{CK}
 with some modification.
 For simplicity, we assume that  $\|q\|_{\infty, r}=1$ and $q\ge 0$.
Now, we first show that there are constants $\eta\in (0,1)$, $\delta \in (\sqrt{\delta_0},1)$ with
$\delta_0\in (0,1)$
being the constant $\delta$ so that all Theorem \ref{exit}, Proposition \ref{Kr} and
Lemma \ref{l3-3} are available,
$R_1>R_0$ and $\xi\in (0,(1/4)\wedge \eta^{1/\alpha})$ (which are
determined later)  such that for any $R_1<R<r_G/(2c_*)$, ${R^{\delta}}\le r \le
R$, $k\ge 1$ with $\xi^kr\ge {2r^{\delta}}$, and any $(\tilde
t,\tilde x)\in Q(t_0,x_0,r)$ with $x_0\in B(0,R)$ and $t_0\ge0$,
\begin{equation}\label{e:ph1}\sup_{Q(\tilde t,
\tilde x, \xi^k r)} q-\inf_{Q(\tilde t,\tilde x, \xi^k r)} q \le \eta^k.\end{equation}

Let $Q_i=Q(\tilde t,\tilde x, \xi^i r)$ and $B_i=B(\tilde x, \xi^i
r)$. Define $a_i=\inf_{Q_i}q$ and $b_i=\sup_{Q_i}q$. Clearly,
$b_i-a_i\le \eta^i$ for all $i\le 0$. Suppose that $b_i-a_i\le
\eta^i$ for all $i\le k$ with some $k\ge 0$. Choose $z_1, z_2\in
Q_{k+1}$ such that $q(z_1)=b_{k+1}$ and $q(z_2)= a_{k+1}.$ Letting
$z_1=(t_1,x_1)$, we define $ \tilde Q_k=Q(t_1,x_1,\xi^k r),$
$\tilde Q_{k+1}=Q(t_1,x_1,\xi^{k+1} r)$ and
$$A_{k}=\left\{z\in \tilde Q_{k+1}:q(z)\le \frac{a_k+b_k}{2}\right\}.$$
Without of loss of generality, we may and do assume that
$\nu(A_{k})/\nu(\tilde Q_{k+1})\ge 1/2;$ otherwise, we will choose
$1-q$ instead of $q$. We have
\begin{align*}b_{k+1}-a_{k+1} =&q(z_1)-q(z_2)= \Ee_{z_1}[q(Z_{\sigma_{A_k}\wedge\tau_{\tilde Q_k}})]-q(z_2)\\
=&  \Ee_{z_1}\left[q(Z_{\sigma_{A_{k}}\wedge\tau_{\tilde Q_k}})-q(z_2): \sigma_{A_{k}}\le \tau_{\tilde Q_k}\right]\\
&+\Ee_{z_1}\left[q(Z_{\sigma_{A_k}\wedge\tau_{\tilde Q_k}})-q(z_2): \sigma_{A_k}> \tau_{\tilde Q_k}, X_{\tau_{\tilde Q_k}}\in B_{k-1}\right]\\
&+\sum_{i=1}^{\infty}\Ee_{z_1}\left[q(Z_{\sigma_{A_k}\wedge\tau_{\tilde
Q_k}})-q(z_2): \sigma_{A_k}> \tau_{\tilde Q_k},
X_{\tau_{\tilde Q_k}}\in B_{k-i-1}\setminus B_{k-i}\right]\\
=&:I_1+I_2+I_3.   \end{align*}
It is easy to see that
$$I_1\le \left(\frac{a_k+b_k}{2}-a_k\right)\Pp_{z_1}(\sigma_{A_k}\le \tau_{\tilde Q_k})\le \frac{b_k-a_k}{2} p_k\le \frac{\eta^k}{2}p_k
=\eta^{k+1} \eta^{-1} \frac{p_k}{2}$$ and
$I_2\le (b_{k-1}-a_{k-1}) (1-p_k)\le \eta^{k-1}(1-p_k)=
\eta^{k+1}\eta^{-2}(1-p_k),$ where
$p_k:= \Pp_{z_1}(\sigma_{A_k}\le \tau_{\tilde Q_k})=\Pp_{(t_1,x_1)}(\sigma_{A_k}\le \tau_{Q(t_1,x_1,\xi^k
r)}).$ On the other hand, since {$\xi^k r\ge 2r^{\delta}\ge
2R^{\delta_0}$}, $\tilde x \in B(x_1,{\xi^{k+1} r})\subset B(x_1,{\xi^k r}/{2})$ and
$\xi^{k-i}r>4\xi^{k}r$ for $i\ge 1$, we can apply \eqref{e:ex} and
 obtain that
$$
 \Pp_{x_1}(X_{\tau_{\tilde Q_k}}\in B_{k-i-1}\setminus B_{k-i})\le
 \Pp_{x_1}\big(X_{\tau_{Q(t_1,x_1,\xi^k r)}}\in B_{k-i}^c\big)\le
 c_2\left(\frac{\xi^k r}{\xi^{k-i}r}\right)^{\alpha}.
$$
Thus,
\begin{align*}I_3&\le \sum_{i=1}^{\infty}(b_{k-i-1}-a_{k-i-1}) \Pp_{x_1}(X_{\tau_{\tilde Q_k}}\in B_{k-i-1}\setminus B_{k-i})\\
&\le c_2\sum_{i=1}^\infty \eta^{(k-i-1)}
\left(\frac{\xi^{k}r}{\xi^{k-i}r}\right)^\alpha \le
\frac{c_2\eta^{k+1} \eta^{-2}\xi^\alpha}{\eta-\xi^\alpha}.
\end{align*}
Note that, since $x_1\in B(0, 2R)$ and {$\xi^k r\ge 2r^{\delta}\ge 2R^{\delta_0}$}, by \eqref{p2-2-1a} we
have $p_k\ge c_3>0$. Combining with all the conclusions above, we
arrive at the statement that
\begin{align*}b_{k+1}-a_{k+1}&\le \eta^{k+1}\left(  \frac{\eta^{-1} p_k}{2}+\eta^{-2}(1-p_k)+
\frac{c_2\eta^{-2}\xi^\alpha}{\eta-\xi^\alpha} \right)\\
&=\eta^{k+1}\left[\eta^{-2}-\Big(\eta^{-2}-\frac{\eta^{-1}}{2}\Big)p_k+
\frac{c_2\eta^{-2}\xi^\alpha}{\eta-\xi^\alpha} \right]\\
&\le\eta^{k+1}\left(\eta^{-2}(1-c_3)+\frac{\eta^{-1}c_3}{2}+
\frac{c_2\eta^{-2}\xi^\alpha}{\eta-\xi^\alpha}
\right).\end{align*} Choosing $\eta$ close to $1$ and then $\xi\in
(0,(1/4)\wedge \eta^{1/\alpha})$ close to $0$ such that $$
\eta^{-2}(1-c_3)+\frac{\eta^{-1}c_3}{2} +
\frac{c_2\eta^{-2}\xi^\alpha}{\eta-\xi^\alpha}\le 1,$$ we get
$b_{k+1}-a_{k+1}\le \eta_{k+1}.$ This proves \eqref{e:ph1}.

For any $(s,x),(t,y)\in Q(t_0,x_0, r)$ with $s\le t$ and $(C_0^{-1}|t-s|)^{1/\alpha}+\rho(x,y)\ge 2r^{\delta}$,
let $k$ be the smallest integer such that $(C_0^{-1}|s-t|)^{1/\alpha}+\rho(x,y)\ge \xi^{k+1} r$. Then,
$(C_0^{-1}|s-t|)^{1/\alpha}+\rho(x,y)\le \xi^{k}r$, and so $\xi^k r\ge 2r^{\delta}$ and $(t,y)\in Q(s,x,\xi^kr)$. According to \eqref{e:ph1}, we know that
$$|q(s,x)-q(t,y)|\le \eta^k\le \eta^{-1}\left(\frac{(C_0^{-1}|s-t|)^{1/\alpha}+\rho(x,y)}{r}\right)^{\log_\xi\eta}.$$ The proof is finished.
\end{proof}

\begin{remark}\label{R:H} (1) By carefully tracking the proofs of Theorems \ref{exit} and $\ref{T:holder}$, we know that
the constant $C_0$ in \eqref{l2-2-1a} and
the constants $C_1$ in \eqref{l2-2-0a}
and \eqref{t3-2-1} only depend on
$c_M$ and $c_G$ given by \eqref{al2-0} and \eqref{al2-1} respectively, as well as $C_1$ and $c_0$ in Assumption {\bf (Exi.($\theta$))}.

(2) According to Proposition \ref{L:tight}, the proof of Theorem \ref{exit} and
the arguments in this subsection, we can obtain that, when
$\alpha\in(0,1)$, Theorems \ref{exit} and \ref{T:holder} still hold
under assumption {\bf(Exi.($\theta$)')}. \end{remark}

\section{Convergence of stable-like processes
on metric measure spaces}\label{section3} In this section, we give
convergence criteria for stable-like processes on metric measure
spaces.
The section is split into four parts. In Subsection
\ref{subsection4.1} we first adopt the framework from \cite{CKK},
which essentially assumes that the metric measure space
$(F,\rho,\mu)$, as the state space of the stable-like process, is
endowed with the graph approximations $(V_n,E_{V_n})_{n\ge1}$. Then,
in Subsection \ref{S:app} we define a class of Dirichlet forms
$(D_{V_n},\F_{V_n})_{n\ge1}$ on the approximating graphs
$(V_n,E_{V_n})_{n\ge1}$, which associate a class of symmetric
$\alpha$-stable-like processes $(X^{(n)})_{n\ge1}$. We also consider
the corresponding scaling limit $(D_0,\F_0)$ of Dirichlet forms,
which generates a $\alpha$-stable-like process $X$ living on
$(F,\rho,\mu)$. With those constructions at hand, we will establish
sufficient conditions for the weak convergence of $(X^n)_{n\ge1}$ to
the process $X$. The proof contains two key steps. The first one is
to apply the results from \cite{CKK} and get the generalized
Mosco
convergence of Dirichlet forms $(D_{V_n},\F_{V_n})_{n\ge1}$ into
$(D_0,\F_0)$, which indicates that the associated semigroups
converge in the $L^2$-sense. The second one is to strengthen this
convergence into the required weak convergence of $(X^n)_{n\ge1}$.
For this, we will make full use of exit time estimates for the
processes $(X^n)_{n\ge1}$ and the H\"{o}lder regularity of
associated caloric
functions studied in Subsections
\ref{subsection3.2} and \ref{subsection3.3}, respectively.

\subsection{Basic setting}\label{subsection4.1}
Let $(F,\rho,m)$ be a metric measure space, where $(F,\rho)$ is a locally compact separable and connected
metric space, and $m$ is a Radon measure on $F$. For every $x\in F$ and $r>0$,
let $B_F(x,r)=\{z\in F: \rho(z,x)<r\}$.
We always assume the following assumptions on $(F,\rho,m)$.
 \paragraph{{\bf Assumption (MMS)}}\label{a3-1}{\it
\begin{itemize}
\item [(i)] For any $x\in F$ and $r>0$, the closure of $B_F(x,r)$ is
compact,
and  $m(\partial (B_F(x,r)))=0$, where $\partial
(B_F(x,r))=\overline{B_F(x,r)}\backslash B_F(x,r).$
\item [(ii)] $\rho: F\times F\rightarrow \R_+$ is geodesic, i.e., for any $x,y\in F$, there exists a
continuous map $\gamma: [0,\rho(x,y)]\rightarrow F$ such that $\gamma(0)=x$, $\gamma(\rho(x,y))=y$
and $\rho(\gamma(s),\gamma(t))=t-s$ for all $0\le s \le t \le \rho(x,y)$.
\item[(iii)] There exist constants $c_F\ge1$ and $d>0$ such that
\begin{equation}\label{a3-1-1}
c_F^{-1}r^{d}\le m(B_F(x,r))\le c_Fr^{d},\quad \forall\ x\in F,\ 0<r<r_F:=\sup_{y,z\in F}\rho(y,z).
\end{equation}
\end{itemize}}

The metric measure space $(F,\rho,m)$ will serve as the state space
of the stable-like process $Y$ which will be defined later.

According to \cite[Theorem 2.1]{CKK}, such a metric measure space is
endowed with the following graph approximations.

\begin{lemma}\label{L:ac} Under assumption {\bf(MMS)}, $F$ admits a sequence of approximating graphs $\{G_n:=(V_n,E_{V_n}),n\ge1\}$ such that the following properties hold.
\begin{itemize}
\item [(1)] For every $n \ge1$, $V_n\subseteq F$, and $(V_n,E_{V_n})$ is connected and has
uniformly bounded degree. Moreover,
$\cup_{n=1}^{\infty}V_n$ is dense in $F$.

\item[(2)] There exist positive constants
$C_1$ and $C_2$ such that for every $n \ge1$ and $ x,y\in V_n$,
\begin{equation}\label{a3-3-1}
\frac{C_1}{n}\rho_{n}(x,y)\le \rho(x,y) \le \frac{C_2}{n}\rho_{n}(x,y),
\end{equation}
where $\rho_n$ is the graph distance of $(V_n,E_{V_n})$.

\medskip

\item[(3)] For each $n\ge1$, there exist a class of subsets $\{U_n(x): x\in V_n\}$ of $F$ such that
$\bigcup_{x \in V_n}U_n(x)\subset F$, $m\big(U_n(x)\cap U_n(y)\big)=0$ for $x \neq y$,
\begin{equation}\label{a3-3-2} V_n\cap \text{Int}\, U_n(x)=\{x\},\,\,
\sup\{\rho(y,z):y,z\in U_n(x)\}\le \frac{C_3}{n},\quad \forall\ x\in V_n,
\end{equation}
and
\begin{equation}\label{a3-3-3}
\frac{C_4}{n^{d}}\le m\big(U_n(x)\big)\le \frac{C_5}{n^{d}},\quad \forall\ n\ge1,\ x\in V_n,
\end{equation} where $\text{Int}\, U_n(x)$ denotes the set of the interior points of $U_n(x)$.

Moreover, for all $r>0$ and $y\in F$,
\begin{equation}\label{a3-3-2a}
\lim_{n \rightarrow \infty}m\Big(
B_F(y,r)\bigcap\big(F\setminus\bigcup_{x \in V_n}U_n(x)\big)\Big)=0.
\end{equation}
 For each $n\ge1$ and $y \in F\setminus\bigcup_{x \in V_n}U_n(x)$, there exists $z\in V_n$  such that
$\rho(y,z)\le {C_6}{n}^{-1}.$ Here $C_i$ $(i=3,\cdots,6)$ are positive constants independent of $n$.
\end{itemize}
\end{lemma}

We will consider stable-like processes on the
graphs $\{G_n\}_{n\ge 1}$.

\subsection{Stable-like processes on graphs and the metric measure spaces}\label{S:app}
We first introduce a class of Dirichlet forms $(D_{V_n}, \F_{V_n})$ on
the graph $(V_n,E_{V_n})$.
For any $n\ge1$, define
\begin{align*}
D_{V_n}(f,f)&=\frac{1}{2}\sum_{x,y\in V_n}(f(x)-f(y))^2\frac{w_{x,y}^{(n)}}{\rho(x,y)^{d+\alpha}}
m_n(x)m_n(y),\quad f\in \F_{V_n},\\
\F_{V_n}&=\{f\in L^2(V_n;m_n): D_{V_n}(f,f)<\infty\},
\end{align*}
where $\alpha\in (0,2)$, $\rho(x,y)$ is the distance function on $F$, $m_n$ is the measure on $V_n$ defined by
$$
m_n(A):=\sum_{x \in A}m\big(U_n(x)\big),\ \ \forall\ A\subset V_n,
$$(for simplicity, we write $m_n(x)=m_n(\{x\})$ for all $x\in V_n$),
and $\{w_{x,y}^{(n)}:x,y\in V_n\}$ is a sequence
satisfying that $w_{x,y}^{(n)}\ge0$ and $w_{x,y}^{(n)}=w_{y,x}^{(n)}$ for all $x\neq y$, and
$$
\sum_{y\in V_n}\frac{w_{x,y}^{(n)}}{\rho(x,y)^{d+\alpha}}m_n(y)<\infty,\quad x\in V_n.
$$  We note that, in the definition of
the Dirichlet form
$(D_{V_n}, \F_{V_n})$ we use the metric $\rho(x,y)$ instead of the graph metric $\rho_n(x,y)$ on $V_n$.
According to \cite[Theorem 2.1]{CKK},  for any $n\ge1$, $(D_{V_n}, \F_{V_n})$ is a regular Dirichlet form on $L^2(V_n;m_n)$. Let $X^{(n)}:=\{(X_t^{(n)})_{t\ge0}, (\Pp_x)_{x\in V_n}\}$ be the associated symmetric Markov process.

To obtain the weak
convergence for  $X^{(n)}$, we also introduce a kind of scaling processes
associated with $\{X^{(n)}\}_{n\ge1}$.
For any $n\ge1$, let $\mathbf{P}_n$  be the projection map from $(V_n, \rho)$ to $(V_n,\rho_{n})$ such that $\mathbf{P}_n(x):=x$ for
$x \in V_n$. Define a measure $\tilde m_n$ on
$(V_n,\rho_{n})$ as follows
$$
\tilde m_n(A)=n^{d}m_n\big(\mathbf{P}_n^{-1}(A)\big)=n^{d}\sum_{x \in \mathbf{P}_n^{-1}(A)}m_n(x),\quad A\subset V_n.
$$ For simplicity, $\tilde m_n(x)=\tilde m_n(\{x\})$ for any $x\in V_n$.
For any $n\ge1$, we consider the following Dirichlet form
$\big(\tilde D_{V_n},\tilde \F_{V_n}\big)$ on $L^2(V_n;\tilde m_n)$:
\begin{align*}
\tilde D_{V_n}(f,f)&=\frac{1}{2}\sum_{x,y\in V_n}(f(x)-f(y))^2 \frac{\tilde w_{x,y}^{(n)}}{\rho_{n}(x,y)^{d+\alpha}}\tilde m_n(x)
\tilde m_n(y),\quad
f\in \tilde \F_{V_n},\\
\tilde \F_{V_n}&=\{f\in L^2(V_n;\tilde m_n): \tilde D_{V_n}(f,f)<\infty\},
\end{align*}
where $$\tilde w^{(n)}_{x,y}:=w^{(n)}_{x,y}
\left(\frac{\rho_{n}(x,y)}{n\rho(x,y)}\right)^{d+\alpha},\quad x,y\in V_n.$$
Note that $\tilde D_{V_n}(f,f)=n^{d-\alpha}D_{V_n}(f,f)$ and $\tilde \F_{V_n}=\F_{V_n}$.
Let $\tilde X^{(n)}$ be the symmetric Markov process associated with $\big(\tilde D_{V_n},\tilde \F_{V_n}\big)$.
According to the expressions of $\big(D_{V_n},\F_{V_n}\big)$ and $\big(\tilde D_{V_n},\tilde \F_{V_n}\big)$, we know that
$\big(\mathbf{P}_n(X^{(n)}_t)\big)_{t\ge 0}$ has the same distribution as
$\big(\tilde X^{(n)}_{n^{\alpha}t}\big)_{t\ge 0}$.

As a candidate of the scaling limit of the
Dirichlet forms $(D_{V_n}, \F_{V_n})$
on the graphs $(V_n, E_{V_n})$, we now define a symmetric Dirichlet
form $(D_0, \F_0)$ on $L^2(F;m)$ as follows
\begin{equation}\label{e4-1}
\begin{split}
D_0(f,f)&=\frac{1}{2}\int_{\{F\times F\setminus {\rm diag}\}}\big(f(x)-f(y)\big)^2\frac{c(x,y)}{\rho(x,y)^{d+\alpha}}\,m(dx)\,m(dy),\quad f \in \F_0,\\
\F_0&=\{f\in L^2(F;m): D_0(f,f)<\infty\},
\end{split}
\end{equation}
where $\alpha\in(0,2)$,  ${\rm diag}:=\{(x,y)\in F\times F:x=y\}$ and  $c:F\times
F\rightarrow (0,\infty)$ is a symmetric
continuous
function such that $0<c_1\le c(x,y)\le c_2<\infty$ for all
$(x,y) \in F\times F\setminus{\rm diag}$ and some constants
$c_1,c_2$. According to \eqref{a3-1-1} and the fact that $\alpha\in(0,2)$, we have
\begin{align*}
&\sup_{x\in F}\int_{F\setminus\{y\in F: y\neq x\}}\big(1\wedge \rho^2(x,y)\big)\frac{c(x,y)}{\rho(x,y)^{d+\alpha}}\,m(dy)\\
&\le \sup_{x\in F}\sum_{k=0}^\infty\int_{\{y\in F: 2^{-(1+k)}<\rho(y,x)\le 2^{-k}\}}\frac{c(x,y)}{\rho(x,y)^{d+\alpha-2}}\,m(dy)\\
&\quad+\sup_{x\in F}\sum_{k=0}^\infty\int_{\{y\in F: 2^{k}<\rho(y,x)\le 2^{1+k}\}}\frac{c(x,y)}{\rho(x,y)^{d+\alpha}}\,m(dy)\\
&\le c_2\sup_{x\in F}
\left(\sum_{k=0}^{\infty}m(B_F(x,2^{-k}))2^{(d+\alpha-2)(1+k)}+\sum_{k=0}^{\infty}m(B_F(x,2^{1+k}))2^{-(d+\alpha)k}\right)\\
&\le c_3\left(
\sum_{k=0}^{\infty}2^{-(2-\alpha) k}+\sum_{k=0}^{\infty}2^{-\alpha k}\right)<\infty.
\end{align*}
This implies ${\rm Lip}_c(F)\subseteq \F_0$, where ${\rm Lip}_c(F)$ denotes the space of Lipschitz continuous functions
on $F$ with compact support.  We also need the following assumption on $(D_0,\F_0)$.

 \paragraph{{\bf Assumption (Dir.)}}\label{a3-2}
{\it
${\rm Lip}_c(F)$ is dense in $\F_0$ under the norm $\|\cdot\|_{D_0,1}:=\!\!\big(D_0(\cdot,\cdot)+\|\cdot\|^2_{L^2(F;m)}\big)^{1/2}$.}

\noindent
Therefore, $(D_0,\F_0)$ is a regular Dirichlet form on $L^2(F;m)$, and there exists a strong Markov process
$Y:=(Y_{t})_{t\ge 0}$ associated with $(D_0,\F_0)$. Moreover, by \cite[Theorem 1.1]{CK} or \cite[Theorem 1.2]{CK08},
the process $Y$ has a heat kernel $p^Y:(0,\infty)\times F \times F \rightarrow (0,\infty)$, which is jointly continuous.
In particular, the process
$Y:=\big((Y_{t})_{t\ge 0}, (\Pp_{x}^Y)_{x \in
F}\big)$ can start from all $x\in F$.
The process $Y$ is called a $\alpha$-stable-like process in the literature, see \cite{CK, CK08}.
Two-sided estimates for heat kernel $p^Y(t,x,y)$ of the process $Y$ have been obtained in \cite{CK}.

\subsection{Generalized Mosco convergence}
To study the convergence property of process $X^{(n)}$, we will use
some results from \cite{CKK}, which are concerned with the
generalized Mosco convergence of $X^{(n)}$.

For any $n\ge1$, we define an
extension operator $E_n:L^2\big(V_n;m_n\big) \rightarrow
L^2(F;m)$ as follows
\begin{equation}\label{e4-2a}
E_n(g)(z)=
\begin{cases}
 g(x), & z\in \text{Int}U_n(x)\ \text{for  some}\ x\in V_n,\\
0,   & z\in F\setminus\bigcup_{x \in V_n}U_n(x),\
\end{cases}\quad g\in L^2\big(V_n;m_n\big).\end{equation}
Note that because $m(\partial U_n(x))=0$ for any $x\in V_n$ by Assumption {\bf(MMS)}(i), there
is no need to worry about $E_n(g)$ on $\bigcup_{x\in V_n}\partial
U_n(x)$, and the function $E_n(g)$ is a.s.\ well defined on $F$.
Note also that the definition of the extension operator $E_n$ above is a little different from that in \cite{CKK}, see \cite[(2.14)]{CKK}.  Furthermore, we  define a projection (restriction) operator
$\pi_n: L^2(F;m) \rightarrow L^2\big(V_n;m_n\big)$ as
follows
$$
\pi_n(f)(x)=m_n(x)^{-1}\int_{U_n(x)}f(z)\,m(dz),\quad x \in V_n,\
f\in L^2\big(F;m\big).
$$

\begin{remark} By Lemma \ref{L:ac}, under assumption {\bf (MMS)}, the space $F$ admits a sequence of approximating graphs $\{(V_n,E_{V_n}):n\ge1\}$ enjoying all the properties mentioned in Lemma \ref{L:ac}. Though  these properties are weaker than {\bf(AG.1)}--{\bf(AG.3)} in \cite[Theorem 2.1]{CKK},
one can verify that \cite[Lemma 4.1]{CKK} and so \cite[Theorem 4.7]{CKK} still hold with notations above.
These weaker properties hold for the case that $F$ is a bounded Lipschitz domain in $\R^d$.
\end{remark}

For simplicity, we assume that
there exists a point $0 \in \bigcap_{n=1}^{\infty}V_n$; otherwise, we can take a sequence $\{o_n\}_{n\ge1}$ such that
$o_n\in V_n$ for all $n\ge1$ and $\lim_{n \rightarrow \infty}o_n$ exists, and then the arguments below still hold true with this limit point $0:=\lim_{n \rightarrow \infty}o_n$.

Fix $0\in \cap_{n=1}^{\infty}V_n$. We assume that the following conditions hold for $\{w_{x,y}^{(n)}: x,y \in V_n\}$.

 \paragraph{{\bf Assumption (Mos.)}} \label{a4-2}{\it

\begin{itemize}
\item[(i)]  For every $R>0$,
\begin{equation}\label{p3-1-1}
\lim_{\varepsilon \to 0}\limsup_{n \rightarrow \infty}
 \bigg[n^{-2d}\sum_{x,y\in B_F(0,R)\cap V_n: 0<\rho(x,y)\le \varepsilon}\frac{w_{x,y}^{(n)}}{\rho(x,y)^{d+\alpha-2}}\bigg]=0
\end{equation} and
\begin{equation}\label{p3-1-1a}
\lim_{l \rightarrow \infty}\limsup_{n \rightarrow \infty}
\bigg[n^{-2d}\sum_{x,y\in B_F(0,R)\cap V_n: \rho(x,y)\ge
l}\frac{w_{x,y}^{(n)}}{\rho(x,y)^{d+\alpha}}\bigg]=0.
\end{equation}

\item[(ii)] For any sufficiently small $\varepsilon>0$, large $R>0$ and any
$f \in {\rm Lip}_c(F)$,
\begin{equation}\label{p3-1-2}
\begin{split}
\lim_{n \rightarrow \infty}\left[n^{-d}\!\!
 \sum_{x\in B_F(0,R)\cap V_n}\!\!\bigg(\sum_{y \in B_F(0,R)\cap V_n: \rho(x,y)>\varepsilon}
\!\!\!\!\big(f(x)\!\!-\!\!f(y)\big)\frac{(w_{x,y}^{(n)}\!\!-\!\!c(x,y))}{\rho(x,y)^{d+\alpha}}m_n(y)\bigg)^2\right]=0.
\end{split}
\end{equation}

\item[(iii)] For any sufficiently small $\varepsilon>0$, large $R>0$ and any
$f\in C_b(B_F(0,R))$,
\begin{equation}\label{p3-1-3}
\begin{split}
\lim_{n \rightarrow \infty}\sum_{x,y \in B_F(0,R)\cap V_n: \rho(x,y)>\varepsilon}
\big(f(x)-f(y)\big)^2\frac{\big(w_{x,y}^{(n)}-c(x,y)\big)}{\rho(x,y)^{d+\alpha}}m_n(x)m_n(y)=0.
\end{split}
\end{equation}
\end{itemize}}

Denote  by $(P^Y_{t})_{t \ge 0}$ the Markov semigroup of the process $Y$, and
denote by $(P_t^{(n)})_{t \ge 0}$ the Markov semigroup of the process
$X^{(n)}$. We set
$\hat P_t^{(n)} f(x)= E_n(P_t^{(n)}(\pi_n(f)))(x)$ for any $f\in L^2(F;m). $

\begin{proposition}\label{p3-1}
Suppose that Assumptions {\bf{(MMS)}, \bf{(Dir.)}} and {\bf{(Mos.)}} hold. Then
$$
\lim_{n \rightarrow \infty}\|\hat P_t^{(n)}f-P_{t}^Y
f\|_{L^2(F;m)}=0,\quad f\in L^2(F;m),\ t>0.
$$
\end{proposition}
\begin{proof}
It is easy to see that the  Dirichlet form $(D_0,\F_0)$ satisfies
{\bf $(A2)$} in \cite[Section 2]{CKK}. By assumption {\bf (Dir.)}
and the continuity of $c(x,y)$,
we know that {\bf $(A3)^*$} in \cite[Section 2]{CKK} holds true.

Clearly, condition {\bf $(A4)^*$} (i) in \cite[Section 2]{CKK} is a
direct consequence of \eqref{p3-1-1} and \eqref{p3-1-1a}. For any
$R,\varepsilon>0$ and $f \in {\rm Lip}_c(F)$, define
\begin{align*}
 L_{R,\varepsilon}f(x)&=\int_{\{z\in B_F(0,R): \rho(z,x)>\varepsilon\}}
(f(z)-f(x))
\frac{c(x,z)}{\rho(x,z)^{d+\alpha}}\,m(dz),\quad x\in F,\\
\overline{L_{R,\varepsilon}^n}f(x)&=\sum_{z \in B_F(0,R)\cap V_n: \rho(x,z)>\varepsilon}(f(z)-f(x))
\frac{w^{(n)}_{x,z}}{\rho(x,z)^{d+\alpha}}m_n(z),\quad x\in V_n,\\
L_{R,\varepsilon}^n f(x)&=E_n(\overline{L_{R,\varepsilon}^n}f)(x),\quad x\in F.
\end{align*}
Then,
$$\int_{B_F(0,R)}|L_{R,\varepsilon}^n f(x)-L_{R,\varepsilon}f(x)|^2\, m(dx)
\le \sum_{i=1}^4I_{i,n},
$$
where
\begin{align*}
I_{1,n}&=2\sum_{x\in B_F(0,R)\cap V_n}\left(\sum_{{y \in B_F(0,R)\cap V_n:}\atop{\rho(x,y)>\varepsilon}}
\big(f(x)-f(y)\big)\frac{(w_{x,y}^{(n)}-c(x,y))}{\rho(x,y)^{d+\alpha}}m_n(y)\right)^2m_n(x),\\
I_{2,n}&=8\text{osc}_n(f)^2 \sum_{x\in B_F(0,R)\cap V_n}\left(\sum_{y \in B_F(0,R)\cap V_n: \rho(x,y)>\varepsilon}
\frac{c(x,y)}{\rho(x,y)^{d+\alpha}}m_n(y)\right)^2m_n(x),\\
I_{3,n}&=8\|f\|_{\infty}^2\text{osc}_n(c)^2
\int_{B_F(0,R)}\left(\int_{B_F(0,R)\cap\{y\in F: \rho(x,y)>\varepsilon\}}
\frac{1}{\rho(x,y)^{d+\alpha}}\,m(dy)\right)^2\,m(dx),\\
I_{4,n}&=4\|f\|_{\infty}^2\|c\|_{\infty}^2\int_{B_F(0,R)\cap(F\setminus\cup_{z\in V_n}U_n(z))}
\!\!\left(\int_{B_F(0,R)\cap({F\setminus\cup_{z\in V_n}U_n(z))}\atop{\cap\{y\in F: \rho(x,y)>\varepsilon\}}}\!\!
 \frac{1}{\rho(x,y)^{d+\alpha}}\,m(dy)\right)^2\!\!m(dx),\\
\text{osc}_n(f)&=\sup_{x\in B_F(0,R)\cap V_n,x_1,x_2\in U_n(x)}
|f(x_1)-f(x_2)|,\\
 \text{osc}_n(c)&=\sup_{x,y\in B_F(0,R)\cap V_n,x_1,x_2\in U_n(x),y_1,y_2\in U_n(y)}|c(x_1,y_1)-c(x_2,y_2)|.
\end{align*}

It follows from \eqref{a3-3-3} and \eqref{p3-1-2} that $\lim_{n \rightarrow \infty}I_{1,n}=0$. Since
$f\in {\rm Lip}_c(F)$,
$\text{osc}_n(f) \rightarrow 0$ as $n \rightarrow \infty$. Then,
we arrive at the statement that
\begin{align*}
\limsup_{n\rightarrow \infty}I_{2,n} &\le c_1\varepsilon^{-2(d+\alpha)}
\big[\limsup_{n\rightarrow \infty} \text{osc}_n(f)^2\big]\\
 &\quad\times\sup_{n\ge1}\Bigg\{n^{-3d}\sum_{x\in B_F(0,R)\cap V_n}\Big(\sum_{y\in B_F(0,R)\cap V_n:\rho(x,y)>\varepsilon}
 {c(x,y)}\Big)^2
\Bigg\}\\
&\le c_2(\varepsilon)
\big[\limsup_{n\rightarrow \infty} \text{osc}_n(f)^2\big]=0.
\end{align*}
By the continuity of $c(x,y)$, it is also easy to see  that $\lim_{n \rightarrow \infty}I_{3,n}=0$. Obviously,
\eqref{a3-3-2a} implies that  $\lim_{n \rightarrow \infty}I_{4,n}=0$.
Therefore, we have
$$
\lim_{n \rightarrow \infty}\int_{B_F(0,R)}|L_{R,\varepsilon}^n f(x)-L_{R,\varepsilon}f(x)|^2\, m(dx)=0,
$$
which implies that condition {$(A4)^*$ (ii)} in \cite[Section
2]{CKK} is satisfied.

Similarly, with aid of \eqref{p3-1-3}, we can claim that condition
{$(A4)^*$} (iii) in \cite[Section 2]{CKK} is also fulfilled.
Therefore, we can verify that all the conditions of {\bf $(A4)^*$}
in \cite[Section 2]{CKK} hold under assumptions {\bf{(MMS)},
\bf{(Dir.)}} and {\bf{(Mos.)}}. Hence, the required assertion
follows from \cite[Theorem 4.7 and Theorem 8.3]{CKK}.
\end{proof}

\subsection{Weak convergence}
The main purpose of this subsection is to establish the weak
convergence theorem of the law for $X^{(n)}$ (with a fixed starting point).
For any $T\in
(0,\infty]$, denote by $\D([0,T];F)$ the collection of
c\`{a}dl\`{a}g $F$-valued functions on $[0,T]$ equipped with the
Skorohod topology. Let $\Pp_{x}^{(n)}$ be the law of $X^{(n)}$ with
starting point $x\in V_n$. Note that $\Pp_{x}^{(n)}$ can be seen as
a distribution on $\D([0,T]; F)$.
Our approach for the weak convergence of $X^{(n)}$ contains the
following two key ingredients. One is to use exit time estimates
from Subsection \ref{subsection3.2} to show the tightness of
$\Pp_{\cdot}^{(n)}$; the other one is to apply the H\"{o}lder
regularity of caloric
functions from Subsection
\ref{subsection3.3}, along with Proposition \ref{p3-1}, to derive
the convergence of $\Pp_{\cdot}^{(n)}$ in the sense of finite
dimensional distributions.

We will make use of scaling processes $\{\tilde X^{(n)}\}_{n\ge1}$
constructed in Subsection \ref{S:app}. First, we consider some
properties of the space $(V_n,\rho_n,\tilde m_n)$. For any $x\in
V_n$ and $r>0$, let $B_{V_n}(x,r)=\{z \in V_n: \rho_{n}(z,x)\le
r\}$.
\begin{lemma}\label{l4-1}
Under assumption {\bf(MMS)}, there are constants $C_0>0$ and $c_V\ge1$ such that for
all $n\ge1$,
\begin{equation}\label{l4-1-1}
c_V^{-1}\le \tilde m_n(x)\le c_V,\quad x\in V_n
\end{equation} and
\begin{equation}\label{l4-1-2}
c_V^{-1} r^{d}\le \tilde m_n(B_{V_n}(x,r))\le c_Vr^{d},\quad x\in V_n, 1\le r <C_0nr_F,
\end{equation}
where $r_F$ is the constant in \eqref{a3-1-1}.
\end{lemma}
\begin{proof}
By the definition of $\tilde m_n$ and \eqref{a3-3-3}, \eqref{l4-1-1} holds trivially.

Note that, for any $x\in V_n$, $y \in B_F(x,r)\cap V_n$ and $z \in
U_n(y)$, by \eqref{a3-3-2}, we have $\rho(z,x)\le
\rho(z,y)+\rho(y,x) \le C_3n^{-1}+r,$ and so $\bigcup_{y \in
B_F(x,r)\cap V_n}U_n(y)\subseteq B_F(x,r+C_3n^{-1}).$ Hence, for
any $x\in V_n$ and $1\le r<{(nr_F-C_3)}/{C_2}$ (where $C_2$  and
$C_3$ are constants in \eqref{a3-3-1} and \eqref{a3-3-2}),
\begin{align*}
\tilde m_n\big(B_{V_n}(x,r)\big)&=n^{d}m_n\big(B_{V_n}(x,r)\cap V_n\big)\le
n^{d}m_n\big(B_{F}(x,C_2n^{-1}r)\cap V_n\big)\\
&=n^{d}\sum_{y \in B_F(x,C_2n^{-1}r)\cap V_n}m\big(U_n(y)\big)\le n^{d}m\big(B_F(x,C_2n^{-1}r+C_3n^{-1})\big)\le c_0r^{d},
\end{align*}
where in the first inequality we used
\eqref{a3-3-1}, the second inequality is due to the facts that $m(U_n(x)\cap U_n(y))=0$ for all $x\neq y$ and $\bigcup_{y \in B_F(x,C_2n^{-1}r)\cap V_n}U_n(y)
\subseteq B_F(x,C_2n^{-1}r+C_3n^{-1})$ as explained above, and the last inequality follows from \eqref{a3-1-1}.

On the other hand, for any
$z \in B_F(x,r)$, by $(3)$  in Lemma \ref{L:ac}, there exists $y\in V_n$ such that $\rho(y,z)\le c_0n^{-1}$ for some constant $c_0>0$, and so
$\rho(y,x)\le \rho(z,x)+\rho(z,y)\le r+c_0n^{-1}.$ This implies that $B_F(x,r)\subset \bigcup_{y \in
B_F(x,r+c_0n^{-1})\cap V_n}B_F(y,c_0n^{-1}).$  Hence, for
$(2(C_1^{-1}c_0))\vee1< r<({nr_F+c_0})/{C_1}$ (where $C_1$  is the
constant in \eqref{a3-3-1}) and $x \in V_n$,
\begin{align*}
\tilde m_n\big(B_{V_n}(x,r)\big)&=n^{d}m_n\big(B_{V_n}(x,r)\big)\ge
n^{d}m_n\big(B_{F}(x,C_1n^{-1}r)\cap V_n\big)\\
&=n^{d}\sum_{y \in B_F(x,C_1n^{-1}r)\cap V_n}m\big(U_n(y)\big)\ge c_1n^{d}\sum_{y \in B_F(x,C_1n^{-1}r)\cap V_n}m\big(B_F(y,c_0n^{-1})\big)\\
&\ge c_1 n^{d}m\big(B_F(x,C_1n^{-1}r-c_0n^{-1})\big)\ge c_2r^{d},
\end{align*}
where in the first inequality we used
\eqref{a3-3-1} again, the second inequality follows from \eqref{a3-1-1} and \eqref{a3-3-3},
the third inequality is due to  $\bigcup_{y \in B_F(x,C_1n^{-1}r)\cap V_n}B_F(y,c_0n^{-1})
\supseteq B_F(x,C_1n^{-1}r-c_0n^{-1})$ as claimed before, and in the last one we have used \eqref{a3-1-1}.

Therefore, combining both estimates above and changing the corresponding constants properly, we prove \eqref{l4-1-2}.
\end{proof}

By \eqref{a3-3-1}, for all $n\ge1$,
$
 \sup_{x,y \in V_n}\rho_{n}(x,y)\le C_1^{-1}nr_F,
 $ where $r_F$ is the constant in \eqref{a3-1-1}.
Below, we let $C_0'=C_1^{-1}$.
For any $x,z\in
V_n$ and $r>0$, let $B^{w^{(n)}}_{V_n,z}(x,r)=\{y\in B_{V_n}(x,r): w^{(n)}_{y,z}>0\}$, and $B^{w^{(n)}}_{V_n}(x,r)=
B^{w^{(n)}}_{V_n,x}(x,r)$.
We need the following further assumptions
on $\{w_{x,y}^{(n)}:x,y\in V_n\}$.

 \paragraph{{\bf Assumption (Wea.($\theta$))}}
 {\it
Suppose that for some fixed $\theta\in (0,1)$,
there exist constants $R_0\ge1$, $c_0\in (1/2,1)$ and $C_3>0$
such that the following conditions hold.
\begin{itemize}
\item [(i)] For any $n \ge1$,
$R_0<R<C_0'nr_F$ and ${R^{\theta}/2}\le r \le 2R$,
\begin{equation}\label{a4-3-1}
\sup_{x \in B_{V_n}(0,6R)}\sum_{y\in V_n: \rho_n(y,x)\le r}
\frac{w_{x,y}^{(n)}}{\rho_n(x,y)^{d+\alpha-2}}\le C_3r^{2-\alpha},
\end{equation}
\begin{equation}\label{a4-3-1a}
m_n(B_{V_n,z}^{w^{(n)}}(x,r))\ge c_0 m_n(B_{V_n}(x,r)),\quad \forall\ x,z\in B_{V_n}(0,6R),
\end{equation}
and
\begin{equation}\label{a4-3-2}
\sup_{x\in B_{V_n}(0,6R)\cap V_n}\sum_{y \in V_n: \rho_n(y,x)\le
c_*r,w^{(n)}_{x,y}>0} (w_{x,y}^{(n)})^{-1}\le C_3r^d,
\end{equation}
where $c_*=8c_V^{2/d}$ and $c_V$ is the constant in \eqref{l4-1-2}.

When $\alpha\in (0,1)$, \eqref{a4-3-1} can be replaced by
\begin{equation}\label{e5-1}
\sup_{x \in B_{V_n}(0,6R)}\sum_{y \in V_n: \rho_n(y,x)\le
r} \frac{w_{x,y}^{(n)}}{\rho_n(x,y)^{d+\alpha-1}}\le
C_3r^{1-\alpha}.
\end{equation}

\item[(ii)] For every $n\ge1$,
$R_0<R<C_0'nr_F$ and $r\ge {R^{\theta}/2}$,
\begin{equation}\label{a4-3-3}
\sup_{x \in B_{V_n}(0,6R)}\sum_{y \in V_n:
\rho_n(x,y)>r}\frac{w_{x,y}^{(n)}}{\rho_n(x,y)^{d+\alpha}}\le
C_3r^{-\alpha}.
\end{equation}
\end{itemize}
}

\medskip

The main result of this section is as follows. It is in some sense a
generalization of \cite[Proposition 2.8]{CCK}. Indeed, in our case
we have the H\"older regularity of caloric
functions only in the region
$(C_0^{-1}|s-t|)^{1/\alpha}+\rho(x,y)\ge 2r^{\delta}$ (see Theorem
\ref{T:holder}), hence more careful arguments are required.

\begin{theorem}\label{t3-1}
Suppose that Assumptions {\bf (MMS)}, {\bf(Dir.)}, {\bf(Mos.)} and
{\bf(Wea.($\theta$))} hold for some $\theta\in (0,1)$.
Then, for any $\{x_n\in V_n: n\ge 1\}$
such that $\lim_{n \rightarrow \infty}x_n=x$ for some $x \in F$,
it holds that for every $T>0$, $\Pp^{(n)}_{x_n}$ converges weakly to
$\Pp_{x}^Y$ on the space of all probability measures on
$\D([0,T];F)$, where $\Pp^{(n)}_{x_n}$ and $\Pp_x^Y$ denote the laws of
$X_{\cdot}^{(n)}$ and $Y_{\cdot}$ on $\D([0,T]; F)$, respectively.
\end{theorem}
\begin{proof}
Throughout the proof, we denote the law of $(X_t^{(n)})_{t\ge 0}$ on $\D([0,\infty);F)$ and  that of $(\tilde X_t^{(n)})_{t\ge 0}$
on $\D([0,\infty);V_n)$ by
$\Pp_{\cdot}^{(n)}$ and $\tilde\Pp_{\cdot}^{(n)}$, respectively. Let  $X_{\cdot}^{(n)}$ and $\tilde X_{\cdot}^{(n)}$ be
their associated canonical paths.

Suppose that $\{x_n\in V_n: n\ge 1\}$ is a sequence with $\lim_{n
\rightarrow \infty}x_n=x$ for some $x \in F$.

{\bf Step (1):}  We show
that for each fixed $T>0$, $\{\Pp_{x_n}^{(n)}\}_{n\ge1}$ is tight on
$\D([0,T];F)$. To prove the tightness of
$\{\Pp_{x_n}^{(n)}\}_{n\ge1}$, it suffices to verify that
\begin{equation}\label{t3-1-0a}
\lim_{R \rightarrow \infty}\limsup_{n \rightarrow \infty}\Pp_{x_n}^{(n)}
\big(\sup_{s\in [0,T]}\rho(0,X_{s}^{(n)})>R\big)=0,
\end{equation}
and for any sequence of stopping time $\{\tau_n\}_{n\ge1}$
such that $\tau_n\le T$ and any sequence
$\{\varepsilon_n\}_{n\ge1}$ with $\lim_{n \rightarrow
\infty}\varepsilon_n=0$,
\begin{equation}\label{t3-1-1}
\limsup_{n \rightarrow \infty}\Pp_{x_n}^{(n)}\Big(
\rho\big(X_{\tau_n+\varepsilon_n}^{(n)}, X_{\tau_n}^{(n)}\big)>\eta\Big)=0,\quad \eta>0.
\end{equation} See, e.g., \cite[Theorem 1]{Ad}.

When $r_F<\infty$, \eqref{t3-1-0a} holds trivially. Now, we are going to prove \eqref{t3-1-0a} for the case that
$r_F=\infty$. As we mentioned above, $\big(\mathbf{P}_n(X^{(n)}_t)\big)_{t\ge 0}$ has the same distribution as
$\big(\tilde X^{(n)}_{n^{\alpha}t}\big)_{t\ge 0}$, where $(\tilde X_t^{(n)})_{t\ge 0}$ is a strong Markov process generated by
the Dirichlet form $(\tilde D_{V_n},\tilde \F_{V_n})$. Therefore,
\begin{equation}\label{t3-1-4}
\begin{split}
\Pp_{x_n}^{(n)} \big(\sup_{s\in [0,T]}\rho(X_{s}^{(n)},0)>R\big)&=\Pp_{x_n}^{(n)} \Big(\sup_{s\in [0,T]}\rho\big(\mathbf{P_n}(X_{s}^{(n)}),0\big)>R\Big)\\
&=\tilde \Pp_{x_n}^{(n)} \Big(\sup_{s \in
[0,n^{\alpha}T]}\rho(\tilde X_s^{(n)},0)>R\Big)\\
&\le \tilde \Pp_{x_n}^{(n)} \Big(\sup_{s \in
[0,n^{\alpha}T]}\rho_{n}(\tilde X_s^{(n)},0)>c_1^*nR\Big),
\end{split}
\end{equation}
where the last inequality follows the fact that $\rho_{n}(x,y)\ge c_1^*n\rho(x,y)$ for all $x,y \in V_n$, thanks to \eqref{a3-3-1}.

On the other hand, under assumption {\bf(Wea.)}, it is easy to
verify that assumption {\bf(Exi.)} (or assumption {\bf(Exi')} when
$\alpha\in (0,1)$) holds for conductances $\tilde w_{x,y}^{(n)}$ on the space $(V_n,\rho_n,\tilde m_n)$
{with associated constants independent of $n$.} {Combining this fact
with \eqref{l4-1-1} and \eqref{l4-1-2}, we can apply} Theorem
\ref{exit} and Remark \ref{R:H}(1) (or Remark \ref{R:H}(2)) to derive that for any fixed
$\theta'\in (\theta,1)$, there exist constants {$\delta\in
(\theta,1)$} and $R_1\ge1$, such that for all $n\ge1$,
$R_1<R<C_0'r_Fn$ and ${R^{\delta}}\le r \le R$,
\begin{equation}\label{t3-1-3a}
\sup_{x \in B_{V_n}(0,2R)\cap V_n}\tilde
\Pp_{x}^{(n)}\big(\tau_{B_{V_n}(0,r)} (\tilde X^{(n)}_{\cdot})\le
t\big)\le c_1\Big(\frac{t}{r^{\alpha}}\Big)^{1/3}, \quad \forall\
t\ge r^{\theta'\alpha},
\end{equation} where $B_{V_n}(x,r)=\{z \in V_n: \rho_{n}(z,x)\le
r\},$ $\tau_{B_{V_n}(0,r)}(\tilde X^{(n)}_{\cdot})$ is the first exit time from $B_{V_n}(0,r)$ of the process $\tilde X^{(n)}_\cdot$, and $c_1>0$ is independent of $R_1$, $n$, $r$, $R$ and $r_F$.

Suppose that $\rho(x_n,0)\le K$ for all $n\ge 1$ and some constant
$K>0$. Note that, also thanks to \eqref{a3-3-1}, $\rho_{n}(x_n,0)\le
c_2^*n\rho(x_n,0)\le c_2^*nK$. For every fixed $R>2c_2^*K/c^*_1$ and
$T>0$, we have
$R_1<c^*_1nR< C_0'nr_F$ (since $r_F=\infty)$ and
$n^{\alpha}T>\big({c^*_1nR}/{2}\big)^{\theta'\alpha}$ for $n$ large
enough. Thus, by \eqref{t3-1-4} and \eqref{t3-1-3a},
\begin{align*}
\Pp_{x_n}^{(n)} \big(\sup_{s\in [0,T]}\rho(X_{s}^{(n)},0)>R\big)
&\le \tilde \Pp_{x_n}^{(n)} \big(\sup_{s\in [0,n^{\alpha}T]} \rho_{n}(\tilde X_{s}^{(n)},0)>c^*_1nR\big)\\
&\le \sup_{z \in B_{V_n}(0,c_2^*nK)\cap V_n}\tilde \Pp_{z}^{(n)}\big(\tau_{B_{V_n}(0,c^*_1nR)}(\tilde X_{\cdot}^{(n)})\le n^{\alpha}T\big)\\
&\le  \sup_{z \in
B_{V_n}(0,{c^*_1nR}/{2})\cap V_n}\tilde \Pp_{z}^{(n)}\big(\tau_{B_{V_n}(z,{c_1^*nR}/{2})}(\tilde X_{\cdot}^{(n)})\le
n^{\alpha}T\big) \\
&\le
c_1\left(\frac{n^{\alpha}T}{(c^*_1nR/2)^{\alpha}}\right)^{1/3}=
c_2\left(\frac{T}{R^{\alpha}}\right)^{1/3},
\end{align*}
which implies
\begin{align*}
&\lim_{R \rightarrow \infty}\limsup_{n \rightarrow
\infty}\Pp_{x_n}^{(n)} (\sup_{s\in [0,T]}\rho(X_{s}^{(n)},0)>R) \le
\lim_{R \rightarrow
\infty}c_2\left(\frac{T}{R^{\alpha}}\right)^{1/3}=0.
\end{align*} This proves \eqref{t3-1-0a}.

Next, let $\{\tau_n\}_{n\ge1}$ be a sequence of stopping time
such that $\tau_n\le T$, and $\{\varepsilon_n\}_{n\ge1}$ be a sequence such that $\lim_{n\to\infty}\varepsilon_n=0$. By the strong Markov property, for every $\eta>0$ small enough and
$R\ge1$ large enough,
\begin{align*}
& \Pp_{x_n}^{(n)}\big(\rho(X_{\tau_n+\varepsilon_n}^{(n)},X_{\tau_n}^{(n)})>\eta\big)= \Ee_{x_n}^{(n)}\big[\Pp^{(n)}_{X_{\tau_n}^{(n)}}\rho(X_{\varepsilon_n}^{(n)},X_{0}^{(n)})>\eta)\big]\\
&\le \sup_{z \in B_F(0,R)\cap V_n}\Pp_{z}^{(n)}\big(\rho(X_{\varepsilon_n}^{(n)},X_{0}^{(n)})>\eta\big)+\Pp_{x_n}^{(n)}
\big(\sup_{s\in [0,T]}\rho(X_{s}^{(n)},0)>R\big)\\
&\le \sup_{z \in B_{V_n}(0,(c_2^*nR)\wedge (C_0'nr_F))\cap V_n}\tilde \Pp_{z}^{(n)}\big(
\rho_{n}(\tilde X_{n^{\alpha}\varepsilon_n}^{(n)},\tilde X_{0}^{(n)})>c_1^*n\eta\big)+\Pp_{x_n}^{(n)}
\big(\sup_{s\in [0,T]}\rho(X_{s}^{(n)},0)>R\big)\\
&\le \sup_{z \in B_{V_n}(0,(c_2^*nR)\wedge (C_0'nr_F))\cap V_n}\!\!\tilde \Pp_{z}^{(n)}\big(\tau_{B_{V_n}(z,c_1^*n\eta)}(\tilde X_{\cdot}^{(n)}\big)\le
n^{\alpha}\varepsilon_n)\!+\!\Pp_{x_n}^{(n)} \big(\sup_{s\in
[0,T]}\rho(X_{s}^{(n)},0)>R\big),
\end{align*}
where in the second inequality we have used the fact that $c_1^*n\rho(x,y)\le \rho_{n}(x,y)\le c_2^*n\rho(x,y)$ for $x,y\in V_n$, due to
\eqref{a3-3-1}. Taking $n$ large enough and $\eta$ small enough such that $c_2^*nR>R_1$ and ${(c_2^*nR)^{\delta}}\le c_1^*n\eta\le
(c_2^*nR)\wedge (C_0'nr_F)$. Then, it
follows from \eqref{t3-1-3a} that
\begin{align*}
&\sup_{z \in B_{V_n}(0,(c_2^*nR)\wedge (C_0'nr_F))\cap V_n}\tilde \Pp_{z}^{(n)}\big(\tau_{B_{V_n}(z,c^*_1n\eta )}(\tilde X_{\cdot}^{(n)})\le
n^{\alpha}\varepsilon_n)\\
&\le \sup_{z \in B_{V_n}(0,(c_2^*nR)\wedge (C_0'nr_F))\cap V_n}\tilde \Pp_{z}^{(n)}\big(\tau_{B_{V_n}(z,c_1^*n\eta )}(\tilde X_{\cdot}^{(n)})\le (n^{\alpha}\varepsilon_n)\vee (2c_1^*n\eta)^{\theta'\alpha}\big)\\
&\le c_1\left(\frac{(n^{\alpha}\varepsilon_n)\vee
(2c_1^*n\eta)^{\theta'\alpha}}{(c_1^*n\eta)^{\alpha}}\right)^{1/3} \le
c_3\left((\varepsilon_n\eta^{-\alpha})\vee
(n\eta)^{-(1-\theta')\alpha}\right)^{1/3}.
\end{align*}
Combining both
estimates above with \eqref{t3-1-0a},
we obtain \eqref{t3-1-1}.

{\bf Step (2):} Now it suffices to show that any finite dimensional distribution of
$\Pp_{x_{n}}^{(n)}$ converges to that of $\Pp^Y_{x}$. We first claim
that for any fixed $t>0$, $f\in C_{\infty}(F)\cap
L^2(F;m)$ and a sequence $\{z_n: z_n \in
V_n\}_{n=1}^{\infty}$ with $\lim_{n \rightarrow \infty}z_n=z\in F$,

\begin{equation}\label{t3-1-1a}
\lim_{n \rightarrow \infty}E_n\big(P_t^{(n)}f\big)(z_n)=P^Y_{t}f(z),
\end{equation}
where $C_\infty(F)$ denotes the set of continuous functions on $F$
vanishing at infinity.

Indeed,
according to assumption {\bf(Mos.)},
Proposition \ref{p3-1} and \eqref{a3-3-2a},
there are a subsequence of $\{\hat P_t^{(n)}f: n\ge 1\}$ (we still
denote it by $\{\hat P_t^{(n)}f:n\ge 1\}$ for simplicity) and a
sequence
{$\{y_k\in \cap_{n\ge N_0}\cup_{x\in V_n}\text{Int}(U_n(x)):k\ge 1\}$} with some $N_0\ge1$
such that
(i) $y_k\neq z$ and $\lim_{k \rightarrow \infty}y_k=z$; (ii) for every $k\ge1$,
\begin{equation}\label{t3-1-2}
\lim_{n \rightarrow \infty}\hat P_t^{(n)} f(y_k)=P_{t}^Y f(y_k).
\end{equation}
For every $k\ge 1$ and $t>0$, we have
\begin{equation}\label{t3-1-2a}
\begin{split}
&|E_n(P_t^{(n)} f)(z_n)-P_{t}^Y f(z)|\\
&\le |\hat P_t^{(n)} f(y_k)-P_{t}^Y f(y_k)|+
|\hat P_t^{(n)} f(y_k)- E_n(P_t^{(n)} f)(y_k)|\\
&\quad+|E_n(P_t^{(n)} f)(y_k)-E_n(P_t^{(n)} f)(z_n)|
+|P_{t}^Y f(z)-P_{t}^Y f(y_k) |\\
&=: |\hat P_t^{(n)} f(y_k)-P_{t}^Y f(y_k)|+\sum_{i=1}^3 J_{i,n,k}.
\end{split}
\end{equation}

Recall that $\hat P_t^{(n)} f(x)=E_n(P_t^{(n)}(\pi_n(f)))(x)$ for all $x\in F$. By the definition of $\pi_n$,
$$\lim_{n \rightarrow \infty}\sup_{z\in V_n}|\pi_n (f)(z)-f(z)|=0$$ for any $f\in C_\infty(F).$
Hence,
\begin{align*}
\lim_{n \rightarrow \infty}\sup_{k \ge1}J_{1,n,k}&= \lim_{n
\rightarrow \infty}\sup_{k \ge1}
|E_n (P_t^{(n)} (\pi_n (f)))(y_k)- E_n(P_t^{(n)} f)(y_k)|\\
&\le \lim_{n \rightarrow \infty}\sup_{z\in V_n}|\pi_n f(z)-f(z)|=0,
\end{align*} where in the last inequality we used the contractivity of $(P_t^{(n)})_{t\ge1}$ in $L^\infty(V_n; m_n)$.

In the following, for any $x\in \cap_{n\ge N_0}\cup_{x\in V_n}\text{Int}(U_n(x))$ and $n\ge N_0$,
let $[x]_n\in V_n$ be
such that $x\in U_n([x]_n)$ and $\rho(x,[x]_n)\le c_3^*n^{-1}$, due
to (3) in Lemma \ref{L:ac}. For any $n\ge N_0$ and $z\in \cap_{n\ge N_0}\cup_{x\in V_n}\text{Int}(U_n(x))$,
noticing that $\big(\tilde X^{(n)}_{n^{\alpha}t}\big)_{t\ge 0}$ has
the same distribution as $\big(\mathbf{P_n}(X^{(n)}_t)\big)_{t\ge
0}$,
\begin{align*}
E_n(P_t^{(n)}f)(z)&=P_t^{(n)}f([z]_n)=\Ee_{[z]_n}^{(n)}[f(X_t^{(n)})]=\tilde \Ee_{[z]_n}^{(n)}[f(\tilde X^{(n)}_{n^{\alpha}t})]=\tilde P_{n^{\alpha}t}^{(n)}f([z]_n),
\end{align*}
where $\tilde P_t^{(n)}f(\cdot):=\tilde \Ee_{\cdot}^{(n)}[f(\tilde X^{(n)}_t)]$ is the Markov semigroup of
$\tilde X^{(n)}:=(\tilde X_t^{(n)})_{t\ge 0}$.
{As mentioned above, due to assumption {\bf(Wea.($\theta$))}
and Lemma \ref{l4-1},}
we can apply Theorem \ref{T:holder} and Remark \ref{R:H}(1) (also thanks to Remark \ref{R:H}(2))
to obtain that there are constants
{$\delta \in (\theta,1)$ and $R_1\ge1$} such that for all $R_1<R<C'_0nr_F$, \eqref{t3-2-1} holds for every $(\tilde X^{(n)}_t)_{t\ge 0}$ and with associated constants independent of $n$.
Let $C_0>0$ be the constant in \eqref{l2-2-1a}.
For fixed $T>0$, we define $H_{T,n,f}(s,x)=\tilde P_{1+n^{\alpha}T-s}^{(n)}f(x)$, which
is a caloric
function on
$Q_{V_n}\big(0,y,(2^{-1}C_0^{-1}n^{\alpha}T)^{1/\alpha}\big)$ for each $y\in V_n$. Take $K$ large enough such that $K>(2^{-1}C_0^{-1}t)^{1/\alpha}$, $R_1<nK<C_0'nr_F$ and
$z_n\in B_{V_n}(0,nK)$ for all $n\ge 1$.
According the facts that $y_k\rightarrow z$ as $k\to\infty$ and $y_k\neq z$ for all $k\ge1$,
for any {fixed} $t>0$, every
$k \ge 1$ and $n\ge N_0$ large enough, it holds that
$0<\varepsilon_k<\rho(y_k,z_n)\le \tilde r_0\le
(4c_2^*)^{-1}((2^{-1}C_0^{-1}t)^{1/\alpha}\wedge (2^{-1}C_0'r_F))$, where $\tilde r_0$ and $\varepsilon_k$ are positive constants with
$\lim_{k \rightarrow \infty}\varepsilon_k=0$,
and $c_2^*>0$ is the constant such that
$\rho_{n}(x,y)\le c_2^*n\rho(x,y)$ for any $x,y\in V_n$.
Furthermore, for these fixed {$k$ and $t$}, we take $n$ large enough such that ${(nK)^{\delta}\le r_n\le nK}$
and $n\varepsilon_k\ge 4(c_1^*)^{-1}r_n^{\delta}$, where $r_n:=(4c_2^*)^{-1}n\tilde r_0$.
Hence,
\begin{align*}
\rho_{n}\big(z_n,[y_k]_n\big)&\ge c^*_1n\rho\big(z_n,[y_k]_n\big)\ge c_1^*n
\big(\rho\big(z_n,y_k\big)-\rho\big(y_k,[y_k]_n\big)\big)\\
&\ge  c_1^*n\varepsilon_k-2c^*_1c_3^*\ge r_n^{\delta}
\end{align*} and
\begin{align*}
\rho_{n}\big(z_n,[y_k]_n\big)&\le c_2^*n \rho\big(z_n,[y_k]_n\big)\le
c_2^*n\big(\rho\big(z_n,y_k\big)+\rho\big(y_k,[y_k]_n\big)\big)\\
&{ \le 4^{-1}r_n+2c_2^*c_3^*\le 2^{-1}r_n},
\end{align*}
where we used the fact that $\rho(y_k,[y_k]_n)\le c_3^*n^{-1}$ for all $k$.

Then as a summary, $(nK)^{\delta}\le
r_n\le nK$, $z_n\in B_{V_n}(0,nK)$, and
$z_n$, $[y_k]_n \in Q_{V_n}\big(0,z_n, r_n\big)$
with $\rho_{n}\big(z_n, [y_k]_n\big)\ge r_n^{\delta}$.
Now, applying
\eqref{t3-2-1} to the caloric function $H_{t,n,f}$ on
$Q_{V_n}(0,z_n,r_n)$, we can obtain that
\begin{align*}
&|\tilde P_{n^{\alpha}t}^{(n)}f([y_k]_n)-\tilde P_{n^{\alpha}t}^{(n)}f(z_n)|\\
&=|H_{t,n,f}(1,[y_k]_n)-H_{t,n,f}(1,z_n)|\le c_4\|\tilde
P_{n^{\alpha}t}^{(n)}f\|_{\infty}\bigg|\frac{\rho_{n}\big([y_k]_n,z_n\big)}{r_n}\bigg|^{\beta}\\
&\le c_{5}(t)\|f\|_{\infty}\rho\big([y_k]_n,z_n\big)^{\beta}\le c_{6}(t)\|f\|_{\infty}\big(\rho(y_k,z_n)^{\beta}+n^{-\beta}\big).
\end{align*}
In particular, constants $c_4$, $c_5(t)$ and $c_6(t)$ are independent of $n$, since
\eqref{t3-2-1} holds for $(\tilde X^{(n)}_t)_{t\ge 0}$ with associated constants independent of $n$.

This yields immediately that
\begin{align*}
\limsup_{n \rightarrow \infty}J_{2,n,k}&=
\limsup_{n \rightarrow \infty}|\tilde P_{n^{\alpha}t}^{(n)}f([y_k]_n)-\tilde P_{n^{\alpha}t}^{(n)}f(z_n)|\\
&\le c_{6}(t)\limsup_{n \rightarrow \infty}\|f\|_{\infty}\big(\rho(y_k,z_n)^{\beta}+n^{-\beta}\big)=c_{7}(t)\|f\|_{\infty}\rho(y_k,z)^{\beta}.
\end{align*}

According to \cite[Theorem 4.14]{CK},
$
J_{3,n,k}\le c_{8}(t)\|f\|_{\infty}\rho(y_k,z)^{\beta}.
$

Combining all estimates with \eqref{t3-1-2a} and \eqref{t3-1-2}, we
arrive at the statement that
$$
\limsup_{n \rightarrow \infty}\big|E_n(P_t^{(n)} f)(z_n)-P_{t}^Y
f(z)\big|\le c_{9}(t)\|f\|_{\infty}\rho(y_k,z)^{\beta},
$$
where $c_9(t)>0$ is independent of $k$. Note that $k$ is arbitrary, letting $k\rightarrow \infty$ in the
last inequality, then we  prove \eqref{t3-1-1a}. In particular, according to \cite[Lemma 2.7]{CCK}, \eqref{t3-1-1a} implies
that for every compact set
$K\subseteq F$,
\begin{equation}\label{t3-1-3}
\limsup_{n \rightarrow \infty}\sup_{x \in K}|E_n(P_t^{(n)}
f)(x)-P_{t}^Y f(x)|=0.
\end{equation}

Next, for any $f_1,f_2\in C_\infty(F)$, $0<s<t\le T$ and any
sequence $x_n \in V_n$ with $\lim_{n \rightarrow \infty}x_n=x\in F$,
\begin{align*}
&\Ee_{x_n}^{(n)}\big[f_1(X_{s}^{(n)})f_2(X_{t}^{(n)})\big]=
\Ee_{x_n}^{(n)}\big[f_1(X_{s}^{(n)})P_{t-s}^{(n)}f_2(X_{s}^{(n)})\big]\\
&=\Ee_{x_n}^{(n)}\big[f_1(X_{s}^{(n)})P^Y_{t-s}f_2(X_{s}^{(n)})\big]+
\Ee_{x_n}^{(n)}\big[f_1(X_{s}^{(n)})\big(P_{t-s}^{(n)}f_2(X_{s}^{(n)})-P^Y_{t-s}f_2(X_{s}^{(n)})\big)\big]\\
&=:J_{1,n}+J_{2,n}.
\end{align*}
Set $g(z)=f_1(z)P^Y_{t-s}f_2(z)$. Then $g\in C_\infty(F)$, due to
the $C_\infty$-Feller property of the process $Y$, see \cite[Theorem
1.1]{CK}. Then, according to \eqref{t3-1-1a}, we have
$$
\lim_{n \rightarrow \infty}J_{1,n}=\lim_{n \rightarrow
\infty}P_t^{(n)}g(x_n)=P_{s}^Yg(x)=
\Ee_{x}^Y\big[f_1(Y_{s})f_2(Y_{t})\big].
$$
On the other hand, for any $t>0$, $R>2K$ and $n$ large enough,
\begin{align*}
J_{2,n}&\le \|f_1\|_{\infty}\sup_{z \in B_F(0,R)}
\big|E_n(P_{t-s}^{(n)}f_2)(z)-P^Y_{t-s}f_2(z)\big|+\|f_1\|_{\infty}\|f_2\|_{\infty}\Pp_{x_n}^{(n)}\big(\sup_{s
\in [0,t]}\rho(X_s^{(n)},0)>R\big),
\end{align*}
By \eqref{t3-1-0a} and \eqref{t3-1-3}, we let $n \rightarrow \infty$
and then $R\rightarrow \infty$ in the last inequality, yielding that
$ \lim_{n \rightarrow \infty}J_{2,n}=0.
$
Combining all above estimates, we prove that
$$
\lim_{n \rightarrow \infty}
\Ee_{x_n}^{(n)}\big[f_1(X_{s}^{(n)})f_2(X_{t}^{(n)})\big]=\Ee^Y_{x}\big[f_1(Y_{s})f_2(Y_{t})\big].
$$
Following the same arguments as above and using the induction procedure,
we can obtain from \cite[Chapter 3; Proposition 4.4 and Theorem 7.8(b)]{EK} that any finite
dimensional distribution of $\Pp_{x_n}^{(n)}$ converges to $\Pp^Y_{x}$. The proof is finished.
\end{proof}

\begin{remark}\label{r4-6} As shown in the proof of Theorem \ref{t3-1} above, the role of adopting the
generalized Mosco convergence is to identify the limit process in
the $L^2$ sense. Actually, according to \cite[Theorem 5.1]{CKK},
under Assumption {\bf(Mos.)}\ only,  any finite dimensional
distribution of $X^{(n)}$ converges to that of $Y$, when the initial
distribution is absolutely continuous  with respect to the reference
measure $m$. Thus, Theorem \ref{t3-1} improves this weak convergence
for any initial distribution. We emphasize that such improvement is
highly non-trivial, see \cite{HK} for discussions on the uniformly
elliptic case by using heat kernel estimates. Here, we will make use
of the H\"{o}lder regularity of caloric
functions on large scale
(Theorem \ref{T:holder}). This is much weaker than the approach used
in \cite[Proposition 2.8]{CCK}, where the H\"older regularity of
caloric
functions is assumed to be satisfied on the whole space.
\end{remark}

\section{Random conductance model: quenched invariance principle}\label{section5}

In this section, we
will apply results from Section \ref{section3} to study the quenched invariance principle for
a few random conductance models with stable-like jumps.

\subsection{Quenched invariance principle for stable-like processes on $d$-sets.}
\label{subsu5-1}
Let $(F,\rho,m)$ be a metric measure space satisfying assumption {\bf(MMS)}.
By Lemma \ref{L:ac},
we have a sequence of graphs with measure $\{(V_n,\rho_n, m_n): n\ge 1\}$ that approximate
$(F,\rho,m)$. In this
part, we further assume the following:
\begin{itemize}
\item[(i)]
$\rho(\cdot,\cdot)$ is a metric with dilation; namely,
there exists another distance $\bar \rho$ on $F$ such that
\begin{itemize}
\item [(i$'$)] for all $x,y\in F$,
$
C_1\bar \rho(x,y)\le \rho(x,y) \le C_2\bar\rho(x,y)$ holds for some constants $0<C_1<C_2<\infty$.

\item[(i$''$)] for each $x,y\in V_1$ and $n\in \N$, there exist
$x^{(n^{-1})}$ and $y^{(n^{-1})}\in V_n$ (we will write
$x^{(n^{-1})}:=n^{-1}x$, $y^{(n^{-1})}:=n^{-1}y$ for notational simplicity) such that $\bar \rho(n^{-1}x,n^{-1}y)=n^{-1} \bar \rho(x,y)$.

\end{itemize}
\item[(ii)] There exists $0\in { V_1}\subset F$ such that $n^{-1}0=0$ for all $n\in \N$.
\item[(iii)]
$V_n=n^{-1}{V_1}:=\{n^{-1}z: z\in V_1\}$, and
$F$ is a closure of $\cup_{n\ge 1}V_n$.
Moreover, $m_n=n^{-d}\mu_n$, $\mu_n(A)=\mu_1(nA)$ for all $A\subset V_n$
and $n\in \N$,
where $\mu_n$ denotes the counting measure on $V_n$.
\end{itemize}

\begin{remark}\label{r5-1}
Obviously conditions (i$'$) and (i$''$) in assumption (i) above hold true for a bounded Lipschitz domain $F \subset \R^d$.
For simplicity, in the arguments below we assume that
$\rho(n^{-1}x,n^{-1}y)$ $=n^{-1} \rho(x,y)$ for all $n\in \N$ and $x,y\in V_1$; otherwise,
we can express Dirichlet forms
$(D_{V_n}^{\w},\F_{n}^{\w})$ and $(D_{0},\F_{0})$ below with $\rho$,
$w_{x,y}^{(n)}(\w)$ and $c(x,y)$ replaced by
$\bar \rho$,
$\bar w_{x,y}^{(n)}(\w):=\frac{\bar\rho(x,y)^{d+\alpha}}{\rho(x,y)^{d+\alpha}}
w_{x,y}^{(n)}(\w)$
and
$\bar c(x,y):=\frac{\bar\rho(x,y)^{d+\alpha}}{\rho(x,y)^{d+\alpha}}c(x,y)$, respectively.
 Hence, by applying the arguments below for $\bar \rho$, $\bar w_{x,y}^{(n)}(\w)$ and $\bar c(x,y)$,
we can still
obtain the quenched invariance principle for $(X_t^{\w})_{t\ge 0}$.
\end{remark}

Let
$\{w_{x,y}(\w): x,y \in V_1\}$ be a sequence of random variables
defined on
a probability space $\big(\Omega,\F,\Pp\big)$ such that
$w_{x,y}(\w)=w_{y,x}(\w)$ and
$w_{x,y}(\w)\ge 0$
for all $x\neq y\in V_1$.
For any $x\in V_n$, $m_n(x):=m_n(\{x\})=n^{-d}$.
Define
\begin{equation}\label{kkk2}
w_{x,y}^{(n)}(\w)=w_{nx,ny}(\w).
\end{equation}
We consider the following class of Dirichlet forms
\begin{align*}
D_{V_n}^{\w}(f,f)&=\frac{1}{2
}\sum_{x,y\in V_n}(f(x)-f(y))^2\frac{w_{x,y}^{(n)}(\w)}{\rho(x,y)^{d+\alpha}}
m_n(x)m_n(y),\quad f\in \F_{n}^{\w},\\
\F_{n}^{\w}&=\{f\in L^2(V_n;m_n): D_{V_n}^{\w}(f,f)<\infty\}.
\end{align*}
Let $X^{V_1,\w}$ be the strong Markov process on $V_1$ associated with $(D_{V_1}^{\w},\F_{1}^{\w})$.
Then, it is easy to show that for a.s.\ $\w\in \Omega$, $(D_{V_n}^{\w},\F_{n}^{\w})$ generates a Markov process $X^{(n),\w}=(X_t^{(n),\w})_{t\ge0}$ such that $X_t^{(n),\w}={n}^{-1}X_{n^{\alpha}t}^{V_1,\w}$ for all $t\ge0$.
Here and what follows,
\lq\lq =\rq\rq\, means two processes enjoy the
same distribution.

Now, consider the Dirichlet
form $(D_0,\F_0)$ given by \eqref{e4-1}, i.e.,
\begin{equation*}
\begin{split}
D_0(f,f)&=\frac{1}{2}\int_{\{F\times F\setminus {\rm diag}\}}\big(f(x)-f(y)\big)^2\frac{c(x,y)}{\rho(x,y)^{d+\alpha}}\,m(dx)\,m(dy),\quad f \in \F_0,\\
\F_0&=\{f\in L^2(F;m): D_0(f,f)<\infty\},
\end{split}
\end{equation*}
where $\alpha\in(0,2)$,  ${\rm diag}:=\{(x,y)\in F\times F:x=y\}$, and $c:F\times
F\rightarrow (0,\infty)$ is a
symmetric
continuous function
such that $0<c_1\le c(x,y)\le c_2<\infty$ for all
$(x,y) \in F\times F\setminus{\rm diag}$ and some constants
$c_1,c_2$. We suppose that assumption {\bf(Dir.)} holds. Let $Y:=((Y_t)_{t\ge0}, (\Pp_{x}^Y)_{x\in F})$ be a
$\alpha$-stable-like process on $F$.

We next apply Theorem \ref{t3-1} to prove the quenched invariance principle for $(X^{\w}_t)_{t\ge 0}$ under some assumptions on $w_{x,y}$.
We first assume that the following holds.
 \paragraph{{\bf Assumption (Den.)}}  \begin{itemize}\it
 \item[(i)] $\Ee[w_{x,y}]=J_1(x,y)$ and
  $\Ee[w_{x,y}^{-1}\I_{\{w_{x,y}>0\}}]=J_2(x,y)$ for any $x,y\in V_1$, where $0<C_1<J_i(x,y)<C_2<\infty$ for all $i=1,2$ and $x,y\in V_1$.
 \item[(ii)] For every compact set
$S\subseteq F$,
\begin{equation}\label{p3-2-0}
\lim_{n \rightarrow \infty}\Big[\sup_{x,y\in S\cap V_n}\Big|J_1(nx,ny)-c(x,y)\Big|\Big]=0.
\end{equation}
\item [(iii)] There exists a countable subset $\Xi\subset{\rm Lip}_c(F)$ such that $\Xi$ is dense
in ${\rm Lip}_c(F)$ under the uniform topology, i.e., for every $f\in {\rm Lip}_c(F)$,
\begin{equation}\label{p3-2-2a}
\inf_{g\in \Xi}\sup_{x\in F}|g(x)-f(x)|=0.
\end{equation}
\end{itemize}
\begin{remark}
Obviously when $F=\R^d$ and $m$ is the Lebesgue measure,
it follows from \eqref{p3-2-0} that for any $x\neq y\in \R^d$ and $s\neq0$,
$c(x,y)=c(sx,sy),$
which implies that the limit process $(Y_t)_{t\ge 0}$ satisfies the scaling invariant property as follows
$$ \Pp^Y_{\varepsilon^{-1}x}\left(\left(\varepsilon Y_{t\varepsilon^{-\alpha}}\right)_{t\ge 0}\in A\right)=
\Pp^Y_x\left(\left(Y_{t}\right)_{t\ge 0}\in A\right)$$ for any $x\in \R^d,$ $\varepsilon>0$ and $ A\subset \mathscr{D}([0,\infty);\R^d)$.
\end{remark}

For $\varepsilon>0$, $x\in V_1$, $R,r>0$,
$c_0>1/2$,
$c_0^*\ge2$ and
a bounded function $h$
on $V_1\times V_1$, define
\begin{align*}
p_1(r,R,\varepsilon)&=\Pp\Big(\Big|\sum_{x,y \in V_1: \rho(0,x)\le R, \rho(x,y)\le r}(w_{x,y}-J_1(x,y))\Big|>\varepsilon r^{d}R^d\Big),\\
p_2(x,r,\varepsilon)&=\Pp\Big(\Big|\sum_{y \in V_1:\rho(x,y)\le r}\big(w_{x,y}-J_1(x,y)\big)\Big|>\varepsilon r^{d}\Big),\\
p_3(x,r,\varepsilon)&=\Pp\Big(\Big|
\sum_{y\in V_1: \rho(x,y)\le r}\frac{(w_{x,y}-J_1(x,y))}{\rho(x,y)^{d+\alpha-2}}\Big|>\varepsilon r^{2-\alpha}\Big),\\
p_3^*(x,r,\varepsilon)&=\Pp\Big(\Big|
\sum_{y\in V_1: \rho(x,y)\le r}\frac{(w_{x,y}-J_1(x,y))}{\rho(x,y)^{d+\alpha-1}}\Big|>\varepsilon r^{1-\alpha}\Big),\quad \alpha\in(0,1),\\
p_4(x,r,c_0^*,\varepsilon)&=\Pp\Big(\Big|\sum_{y\in V_1: \rho(x,y)\le
c_0^*r}\big(w_{x,y}^{-1}-J_2(x,y)\big)\Big|
>\varepsilon r^d\Big),\\
p_5^{(n)}(x,R,r,h,\varepsilon)&=\Pp\Big(\Big|\sum_{{y \in B_F(0,nR)\cap V_1:}
\atop{\rho(x,y)\ge nr}}h(x,y)\frac{(w_{x,y}-J_1(x,y))}{\rho(x,y)^{d+\alpha}}\Big|^2>\varepsilon (nr)^{-2\alpha}\Big),\\
p_6(x,z,r,c_0)&=\Pp\left(\frac{\mu_1\{y\in V_1: \rho(y,x)\le r, w_{y,z}>0\}}{\mu_1\{y\in V_1: \rho(y,x)\le r\}}\le c_0\right).
\end{align*}

\begin{theorem}\label{t5-1}
Suppose that assumption {\bf(Den.)}\! holds, and that there exists a constant $\theta \in (0,1)$
such that
\begin{itemize}
\item[(i)]
for any $\varepsilon_0$ and $\varepsilon$ small enough, any $N$ large enough, and any sequence of
bounded function $\{h_n\}_{n\ge 1}$  on $V_1\times V_1$ with $\sup_{n\ge 1}\|h_n\|_{\infty}<\infty$,
\begin{equation}\label{p3-2-1}
\sum_{R=1}^{\infty}\sum_{r=1}^{R}p_1(r,R,\varepsilon_0)<\infty,
\end{equation}
\begin{equation}\label{p3-2-1a}
\sum_{R=1}^{\infty}\sum_{x \in
B_F(0,6R)\cap V_1}\sum_{r=R^{\theta}/2}^{\infty}p_2(x,r,\varepsilon_0)<\infty,
\end{equation} and
\begin{equation}\label{p3-2-2}
\sum_{n=1}^{\infty}\sum_{x\in B_F(0,nN)\cap V_1}p_5^{(n)}(x,N,\varepsilon,h_n,\varepsilon_0)<\infty.
\end{equation}

\item[(ii)] any $\varepsilon_0$ small enough, \begin{equation}\label{l3-1-1}
\sum_{R=1}^{\infty}\sum_{x \in
B_F(0,6R)\cap V_1}p_3(x,{R^{\theta}},\varepsilon_0)<\infty
\end{equation} and
\begin{equation}\label{l3-1-2}
\sum_{R=1}^{\infty}\sum_{x \in
B_F(0,6R)\cap V_1}\sum_{r=R^{\theta}/2}^{2R}
p_4(x,r,c_0^*,\varepsilon_0)<\infty,
\end{equation} for any fixed $c_0^*\ge0$, as well as
\begin{equation}\label{l3-1-2-a}
\sum_{R=1}^{\infty}\sum_{x,z\in
B_F(0,6R)\cap V_1}\sum_{r=R^{\theta}/2}^{2R}
p_6(x,z,r,c_0)<\infty
\end{equation} for some fixed $c_0>1/2$.

When $\alpha\in (0,1)$, \eqref{l3-1-1} can be replaced by
\begin{equation}\label{l3-1-1-1}
\sum_{R=1}^{\infty}\sum_{x \in
B_F(0,6R)\cap V_1}p^*_3(x,R^{\theta},\varepsilon_0)<\infty.
\end{equation} \end{itemize}
 Then for
$\Pp$-a.s.\ $\w\in \Omega$ and
any $\{x_n\in V_n: n\ge 1\}$
such that $\lim_{n \rightarrow \infty}x_n=x$ with some $x \in F$,
it holds that for every $T>0$, $\Pp^{(n),\w}_{x_n}$ converges weakly to
$\Pp_{x}^Y$ on the space of all probability measures on
$\D([0,T];F)$, where $\Pp^{(n),\w}_{x_n}$ denotes the distribution of
process $X_t^{(n),\w}=n^{-1}X^{V_1,\w}_{n^{\alpha}t}$.
\end{theorem}

Theorem \ref{t5-1} immediately holds by applying Theorem \ref{t3-1}, Lemmas \ref{p5-2} and \ref{l5-1-0}  below
to process $X_t^{(n),\w}$.

\begin{lemma}\label{p5-2} Under assumption ${\rm(i)}$ in Theorem $\ref{t5-1}$, for $\Pp$-a.s.\ $\w\in \Omega$,
Assumption {\bf(Mos.)} holds for the conductance $\{w_{x,y}^{(n)}(\w)\}$.
\end{lemma}
\begin{proof}
Under \eqref{p3-2-1}, for any $\varepsilon_0>0$,
\begin{align*}
&\sum_{R=1}^{\infty} \Pp\Big(\bigcup_{r=1}^{R}
\Big\{\Big|\sum_{x,y\in V_1: \rho(0,x)\le R,\rho(x,y)\le r}\big(w_{x,y}-J_1(x,y)\big)\Big|>\varepsilon_0 r^{d}R^d\Big\}\Big)\\
&\le \sum_{R=1}^{\infty}\sum_{r=1}^{R}\Pp\Big(
\Big|\sum_{x,y\in V_1: \rho(0,x)\le R,\rho(x,y)\le r}\big(w_{x,y}-J_1(x,y)\big)\Big|>\varepsilon_0 r^{d}R^d\Big)=\sum_{R=1}^{\infty}\sum_{r=1}^{R}p_1(r,R,\varepsilon_0)<\infty.
\end{align*}
Since $C_1\le J_1(x,y)\le C_2$ for all $x,y \in V_1$ and some
positive constants $C_1$ and $C_2$, by the Borel-Cantelli lemma, we
know that, for $\Pp$-a.s.\ $\w \in \Omega$, there exists a constant
$R_0(\w)\ge1$ such that for every $R>R_0(\w)$,
$$
c_1r^{d}R^d\le \sum_{x,y\in V_1:\rho(0,x)\le R, \rho(x,y)\le r}w_{x,y}(\w)\le
c_2r^{d}R^d,\quad \forall\
 1\le r\le R,
$$
where $c_1,c_2$ are positive constants independent of $\w$. Then,
for any $0<2\eta<N$ and $nN>R_0(\w)$, we have
\begin{align*}
&n^{-2d}\sum_{x,y\in B_F(0,N)\cap V_n: 0<\rho(x,y)\le \eta}\frac{w_{nx,ny}(\w)}{\rho(x,y)^{d+\alpha-2}}\\
&\le n^{-d+\alpha-2}\sum_{k=0}^{[\log(n\eta)/\log 2]+1}
\sum_{x,y\in V_1: \rho(0,x)\le nN \text{ and }2^k\le \rho(x,y)<2^{k+1}}\frac{w_{x,y}(\w)}{\rho(x,y)^{d+\alpha-2}}\\
&\le n^{-d+\alpha-2}\sum_{k=0}^{[\log(n\eta)/\log 2]+1}2^{-k(d+\alpha-2)}\sum_{x,y\in V_1: \rho(0,x)\le nN \text{ and }2^k\le \rho(x,y)<2^{k+1}}
w_{x,y}(\w)\\
&\le c_3n^{-d+\alpha-2}\sum_{k=0}^{[\log(n\eta)/\log 2]+1}2^{-k(d+\alpha-2)}2^{(k+1)d}(nN)^d\le
c_4N^d\eta^{2-\alpha}.
\end{align*}
This yields that \eqref{p3-1-1} holds for $\Pp$-a.s.\ $\w\in \Omega$.

According to \eqref{p3-2-1a}, for every $\varepsilon_0>0$
small enough,
\begin{align*}
&\sum_{R=1}^{\infty}\Pp\Big(\bigcup_{x\in B_F(0,6R)\cap V_1}\bigcup_{r=R^{\theta}/2}^{\infty}\Big\{
\Big|\sum_{y\in V_1: \rho(x,y)\le r}\big(w_{x,y}-J_1(x,y)\big)\Big|>\varepsilon_0 r^d\Big\}\Big)\\
&\le \sum_{R=1}^{\infty}\sum_{x \in B_F(0,6R)\cap V_1}\sum_{r=R^{\theta}/2}^{\infty}
\Pp\Big(\Big\{
\Big|\sum_{y\in V_1: \rho(x,y)\le r}\big(w_{x,y}-J_1(x,y)\big)\Big|>\varepsilon_0 r^d\Big\}\Big)\\
&\le  \sum_{R=1}^{\infty}\sum_{x \in B_F(0,6R)}\sum_{r=R^{\theta}/2}^{\infty}p_2(x,r,\varepsilon_0)<\infty.
\end{align*}
Hence, by the Borel-Cantelli lemma, we can find
a constant $R_{1}(\w)>0$ such that for every $R>R_{1}(\w)$,
$x\in B_F(0,6R)\cap V_1$ and $r\ge R^{\theta}/2$,
$
\Big|\sum_{y\in V_1:\rho(x,y)\le r}(w_{x,y}-J_1(x,y))\Big| \le
\varepsilon_0 r^{d}.
$
Due to the fact that $0<C_1\le J_1(x,y)\le
C_2<\infty$ for any $x,y\in V_1$ again, we arrive at the statement that for all $R>R_1(\w)$,
\begin{equation}\label{p3-2-3}
c_5r^{d}\le \sum_{y\in V_1:\rho(x,y)\le r} w_{x,y} \le c_6 r^{d},\quad
\forall\
 x\in B_F(0,6R),\ r\ge R^{\theta}/2.
\end{equation}
Therefore, by \eqref{p3-2-3}, for every $n,j\ge1$ large enough such that $2nN>R_1(\w)$ and $j>N$,
\begin{align*}
&n^{-2d}\sum_{x,y\in B_F(0,N)\cap V_n: \rho(x,y)\ge j}\frac{w_{nx,ny}(\w)}{\rho(x,y)^{d+\alpha}}\\
&\le n^{-d+\alpha}\sum_{x \in V_1:\rho(0,x)\le nN}
\sum_{y \in V_1: \rho(x,y)\ge nj}\frac{w_{x,y}(\w)}{\rho(x,y)^{d+\alpha}}\\
&\le n^{-d+\alpha}\sum_{x \in V_1: \rho(0,x)\le nN}\sum_{k=\big[\frac{\log (nj)}{\log 2}\big]}^{\infty}
2^{-k(d+\alpha)}\sum_{y\in V_1: \rho(x,y)\le 2^{k+1}}w_{x,y}(\w)\\
&\le c_7n^{-d+\alpha}\sum_{x \in V_1: \rho(0,x)\le nN}
\sum_{k=\big[\frac{\log (nj)}{\log 2}\big]}^{\infty}2^{-k(d+\alpha)}2^{(k+1)d}
\le c_8N^dj^{-\alpha}.
\end{align*}
Hence, letting $n \rightarrow \infty$ first and then $j \rightarrow
\infty$, we prove that \eqref{p3-1-1a} holds for $\Pp$-a.s. $\w\in
\Omega$.

Given $f \in {\rm Lip}_c(F)$, let
\begin{equation*}
h_n(x,y):=
\begin{cases}
 f(n^{-1}{y})-f(n^{-1}{x}),&n^{-1}{x},n^{-1}{y}\in V_n,\\
0, &\text{otherwise}.
\end{cases}
\end{equation*}
Let $\mathbb{Q}_+$ be the set of all positive rational numbers. Applying \eqref{p3-2-2} to
$h_n(x,y)$ and using the Borel-Cantelli lemma, we
can find a null set $\Phi(f)$ so that
for every $\omega \notin \Phi(f)$,
$\varepsilon,\varepsilon_0\in \mathbb{Q}_+$ small enough and $R\in \mathbb{Q}_+$ large enough, there
exists a constant $n_{0}(\w)>0$ (which may depend on
$\varepsilon_0$, $\varepsilon$, $N$ and $f$) such that for every $n>n_{0}(\w)$ and $x\in {B_F(0,nR)}\cap V_1$,
 $$
\Big|\sum_{y\in B_F(0,nR)\cap V_1: \rho(x,y)\ge n\varepsilon}\big(
f(n^{-1}{y}\big)-f(n^{-1}{x})\big)\frac{(w_{x,y}(\w)-J_1(x,y))}{\rho(x,y)^{d+\alpha}}\Big|^2 \le
\varepsilon_0(n\varepsilon)^{-2\alpha}.
$$
Then, for $n$ large enough such that $n\varepsilon>(nR)^{\theta}$, we have
\begin{align*}
&n^{-d}
\sum_{x\in B_F(0,R)\cap V_n}\left(\sum_{{y \in B_F(0,R)\cap V_n:}\atop {\rho(x,y)>\varepsilon}}
\big(f(x)-f(y)\big)\frac{\big(w_{nx,ny}(\w)-J_1(nx,ny)\big)}{\rho(x,y)^{d+\alpha}}m_n(y)\right)^2\\
&=n^{-d+2\alpha}
\sum_{x\in B_F(0,nR)\cap V_1}\left(\sum_{y \in B_F(0,nR)\cap V_1 : \rho(x,y)>n\varepsilon}
h_n(x,y)\frac{\big(w_{x,y}(\w)-J_1(x,y)\big)}{\rho(x,y)^{d+\alpha}}\right)^2\\
&\le n^{-d+2\alpha}
\sum_{x \in B_F(0,nR)\cap V_1}\varepsilon_0(n\varepsilon)^{-2\alpha}\le
c_9R^d\varepsilon^{-2\alpha} \varepsilon_0.
\end{align*}
On the other hand, due to \eqref{p3-2-0}, we can verify that every
fixed $R>0$ and $\varepsilon>0$,
\begin{align*}
&\lim_{n \rightarrow \infty}
n^{-d}\sum_{x\in B_F(0,R)\cap V_n}\bigg(\sum_{{y \in B_F(0,R)\cap V_n:}\atop {\rho(x,y)>\varepsilon}}
\big(f(x)-f(y)\big)\frac{\big(J_1(nx,ny)-c(x,y)\big)}{\rho(x,y)^{d+\alpha}}m_n(y)\bigg)^2\\
&\le 4\|f\|_{\infty}^2\varepsilon^{-2(d+\alpha)}\lim_{n \rightarrow \infty}
n^{-3d}
\sum_{x\in B_F(0,R)\cap V_n}\Big(\sum_{y \in B_F(0,R)\cap V_n: \rho(x,y)>\varepsilon}\big|J_1(nx,ny)-c(x,y)\big|\Big)^2\\
&\le c_{10}\|f\|_{\infty}^2\varepsilon^{-2(d+\alpha)}R^{d}\lim_{n \rightarrow \infty}\bigg\{
n^{-2d}\sum_{x,y\in B_F(0,R)\cap V_n}\big(J_1(nx,ny)-c(x,y)\big)^2\bigg\}=0.
\end{align*}
Combining two estimates above, we can obtain that
for every $f\in {\rm Lip}_c(F)$, \eqref{p3-1-2}
holds for all $\w\notin \Phi(f)$ and $\varepsilon,R\in \mathbb{Q}_+$ with $\varepsilon$ small enough and $R$ large enough, by first letting $n \rightarrow\infty$ and then taking $\varepsilon_0\rightarrow 0$.

For the countable subset $\Xi\subset {\rm Lip}_c(F)$ in Assumption {\bf (Den.)}(iii), we define
$\Phi_1:=\cup_{f\in \Xi}\Phi(f)$. Obviously $\Phi_1$ is a null set. It is also easy to see that
for every $\w\notin \Phi_1$, $f\in \Xi$ and $\varepsilon,R\in \mathbb{Q}_+$ with $\varepsilon$ small enough and $R$ large enough,
\begin{equation}\label{p3-2-4}
\lim_{n \rightarrow \infty}T_n(f,\varepsilon, R, \w)=0,
\end{equation}
where
\begin{align*}
T_n(f,\varepsilon, R, \w):=n^{-d}\!\!\sum_{x\in B_F(0,R)\cap V_n}\!\!\!\!\bigg(\sum_{y\in B_F(0,R)\cap V_n: \rho(x,y)> \varepsilon}\!\!\big(
f(y)-f(x)\big)\frac{w_{nx,ny}(\w)-c(x,y)}{\rho(x,y)^{d+\alpha}}m_n(y)\!\!\bigg)^2.
\end{align*}

Furthermore, note that for every $f,g\in {\rm Lip}_c(F)$ and $\varepsilon, R>0$,
\begin{align*}
&T_n(f-g,\varepsilon,R,\w)\\
&\le 4n^{-d}\sup_{x\in F}|f(x)-g(x)|^2\cdot \Big(\sum_{x\in B_F(0,R)\cap V_n}
\Big(\sum_{y\in B_F(0,R)\cap V_n: \rho(x,y)\ge \varepsilon}\frac{w_{nx,ny}(\w)+c(x,y)}{\rho(x,y)^{d+\alpha}}m_n(y)\Big)^2\Big)\\
&\le c_{11}n^{-3d}\varepsilon^{-2(d+\alpha)}\sup_{x\in F}|f(x)-g(x)|^2\cdot
\Bigg[\sum_{x\in B_F(0,R)\cap V_n}\Big(\sum_{y\in B_F(0,R)\cap V_n}w_{nx,ny}+c(x,y)\Big)^2\Bigg].
\end{align*}
Then, by \eqref{p3-2-3}, we can find a null set $\Phi_2\in \Omega$ such that
for each $\w\notin \Phi_2$,
there exists
$n_0(\w)\ge1$ so that for all $n>n_0(\w)$, $f,g\in {\rm Lip}_c(F)$ and $\varepsilon, R>0$,
$$
T_n(f-g,\varepsilon,R,\w)\le c_{12}\varepsilon^{-2(d+\alpha)}R^{3d}\sup_{x\in F}|f(x)-g(x)|^2.
$$
Combining the estimate above with \eqref{p3-2-2a} yields that for each $f\in {\rm Lip}_c(F)$, $\varepsilon,R>0$ and $\w \notin \Phi_2$,
\begin{align}\label{p3-2-5a}
\inf_{g\in \Xi}\limsup_{n \rightarrow \infty}T_n(f-g,\varepsilon,R,\w)=0.
\end{align}

Let $\Phi:=\Phi_1\cup\Phi_2$. Then, $\Phi$ is a null set. According to \eqref{p3-2-4}, we know that
for every
$f\in {\rm Lip}_c(F)$, $g\in \Xi$, $\w\notin \Phi$ and $\varepsilon, R\in \mathbb{Q}_+$ with $\varepsilon$ small enough and $R$ large enough,
\begin{align*}
\limsup_{n \rightarrow \infty}T_n(f,\varepsilon,R,\w)&\le 2\limsup_{n \rightarrow \infty}T_n(g,\varepsilon,R,\w)
+2\limsup_{n \rightarrow \infty}T_n(f-g,\varepsilon,R,\w)\\
&=2\limsup_{n \rightarrow \infty}T_n(f-g,\varepsilon,R,\w).
\end{align*}
Therefore, taking infimum over all $g\in \Xi$ in the inequality above and applying \eqref{p3-2-5a}, we can prove that \eqref{p3-1-2} holds for every $f\in {\rm Lip}_c(F)$,
$\w\notin \Phi$ and $\varepsilon, R\in \mathbb{Q}_+$ with $\varepsilon$ small enough and $R$ large enough.  By the same (and even simpler) approximation arguments as above,
we can further verify that \eqref{p3-1-2} holds for every $f\in {\rm Lip}_c(F)$,
$\w\notin \Phi$, $\varepsilon>0$ small enough and $R>0$ large enough.

Since \eqref{p3-1-3} can be proved in the similar way, we omit it here.
\end{proof}
We note that the last part of the proof above that handles the null set carefully is motivated by the proof of \cite[Theorem 1.1]{CCKW}.

\begin{lemma}\label{l5-1-0} Suppose that there exists a constant $\theta\in (0,1)$ such that condition \eqref{p3-2-1a} and assumption ${\rm (ii)}$ in Theorem $\ref{t5-1}$ hold.
Then for $\Pp$-a.s. $\w\in \Omega$, Assumption {\bf(Wea.($\theta$))} holds for the conductance $\{w_{x,y}^{(n)}(\w)\}$.
\end{lemma}
\begin{proof}
First, according to \eqref{l3-1-2-a}, the property $\mu_n(A)=\mu_1(nA)$ and the definitions of $m_n$ and $w^{(n)}_{x,y}$, we can easily deduce from the Borel-Cantelli lemma
that there is a constant $R_0(\w)>0$ such that for
any $R>R_0(\w)$ and $R^\theta/2\le r \le R$, \eqref{a4-3-1a} holds.

By \eqref{l3-1-1},
\begin{align*}
&\sum_{R=1}^{\infty}\Pp\Bigg(\bigcup_{x \in B_F(0,6R)\cap V_1}\Big\{
\Big|\sum_{y\in V_1: \rho(x,y)\le
R^{\theta}}\frac{\big(w_{x,y}-J_1(x,y)\big)}{\rho(x,y)^{d+\alpha-2}}\Big|>\varepsilon_0
R^{\theta(2-\alpha)}\Big\}\Bigg)\\
&\le \sum_{R=1}^{\infty}\sum_{x \in B_F(0,6R)\cap V_1}
\Pp\Big(\Big|
\sum_{y\in V_1: \rho(x,y)\le R^{\theta}}
\frac{\big(w_{x,y}-J_1(x,y)\big)}{\rho(x,y)^{d+\alpha-2}}\Big|>\varepsilon_0 R^{\theta(2-\alpha)}\Big)\\
&=\sum_{R=1}^{\infty}\sum_{x \in
B_F(0,6R)\cap V_1} p_3(x,R^{\theta},\varepsilon_0)<\infty.
\end{align*}
Hence,  by the Borel-Cantelli lemma, there exists a constant $R_0(\w)>0$ such that for
any $R>R_0(\w)$,
\begin{equation}\label{l3-1-4}
\sum_{y\in V_1:\rho(x,y)\le R^{\theta}}\frac{w_{x,y}}{\rho(x,y)^{d+\alpha-2}}\le
c_1R^{\theta(2-\alpha)},\quad \forall\ x\in B_F(0,6R)\cap V_1.
\end{equation}
Furthermore, using \eqref{p3-2-3} and choosing $\varepsilon_0$ small
enough and $R_0(\w)$ large enough, we find that for every $R>R_0(\w)$,
\begin{equation}\label{l3-1-4a}
c_2^{-1}r^{d}\le \sum_{y \in V_1: \rho(x,y)\le r}w_{x,y}\le
c_2r^{d},\quad \forall\ r>R^{\theta}/2,\ x\in B_F(0,6R)\cap V_1.
\end{equation}
Combining this with \eqref{l3-1-4}, we see that for every
$R>R_0(w)$, $x\in B_F(0,6R/n)\cap V_n$,
and $R^{\theta}/2\le r \le 2R$,
\begin{align*}
&n^{-(d+\alpha-2)}\sum_{y\in V_n:\rho(x,y)\le r/n}
\frac{w^{(n)}_{x,y}}{\rho(x,y)^{d+\alpha-2}}\\
&\le \sum_{y\in V_1: \rho(x,y)< R^{\theta}/2}\frac{w_{x,y}}{\rho(x,y)^{d+\alpha-2}}+
\sum_{k=[\log (R^{\theta}/2)/\log 2]}^{[\log r
/\log 2]+1} 2^{-k(d+\alpha-2)}\Big(\sum_{y\in V_1: 2^k<\rho(x,y)\le 2^{k+1}}
w_{x,y}\Big)\\
&\le c_4 \Big(R^{\theta(2-\alpha)}+
\sum_{k=[\log (R^{\theta}/2)/\log 2]}^{[\log r
/\log 2]+1}2^{-k(\alpha-2)}\Big)\le c_5 r^{2-\alpha}.
\end{align*}
Therefore, \eqref{a4-3-1} holds for $\Pp$-a.s.\ $\w\in \Omega$.

Due to \eqref{l3-1-4a} again, we know that for every $R>R_0(\w)$, $x\in B_F(0,6R/n)\cap V_n$
and $r>R^{\theta}/2$,
\begin{align*}
n^{-(d+\alpha)}\sum_{y\in V_n: \rho(x,y)>r/n}
\frac{w^{(n)}_{x,y}}{\rho_n(x,y)^{d+\alpha}}&\le
\sum_{k=[\log r
/\log 2]}^{\infty}2^{-k(d+\alpha)}
\Big(\sum_{y \in V_1: 2^k<\rho(x,y)\le 2^{k+1}}w_{x,y}\Big)\\
&\le c_6\sum_{k=[\log r
/\log 2]}^{\infty}2^{-k(d+\alpha)}2^{d(k+1)}\le c_7 r^{-\alpha},
\end{align*}
which implies that \eqref{a4-3-3} is satisfied for $\Pp$-a.s.\ $\w\in \Omega$.

Following the arguments above, and using \eqref{l3-1-2} and the
Borel-Cantelli lemma, we can obtain that \eqref{a4-3-2} holds for
$\Pp$-a.s.\ $\w\in \Omega$. On the other hand, when $\alpha\in (0,1)$, we can use \eqref{l3-1-1-1} to prove that \eqref{e5-1} holds for
$\Pp$-a.s.\ $\w\in \Omega$.
The proof is complete.
\end{proof}

\subsection{Examples}
As an application of Theorem \ref{t5-1}, we consider  four examples.
One is a lattice on a half/quarter space, and other three are time-change of stable-like processes,
a bounded Lipschitz domain and a fractal graph respectively.
We also show the the quenched invariance principle for a class of constant speed $\alpha$-stable-like random walks on $\Z^d$, by using Theorem \ref{th1} and the time change argument.
\subsubsection{Lattice on a half/quarter space}\label{laths}
Let $F:=\mathbb{R}^{d_1}_+\times\mathbb{R}^{d_2}$ with $d_1,d_2\in \mathbb{N}\cup\{0\}$,
 and $\rho$ and $m$ be the Euclidean distance and the Lebesgue measure respectively, which clearly satisfy assumption {\bf(MMS)}.
Therefore the process $Y$ associated with Dirichlet form  $(D_0,\mathscr{F}_0)$ is a reflected stable-like process on $F$, see e.g. \cite{CK}.
Obviously $(D_0,\mathscr{F}_0)$ satisfies assumption {\bf(Dir.)}.
Here we will take
$
V_1=\mathbb{L}:=\mathbb{Z}^{d_1}_+\times\mathbb{Z}^{d_2},
$ and $K_n\equiv 1$ for all $n\in \N$.
Note that the scaling limit of $n^{-1}\mathbb{L}$ is $F$.

Let $E:=\{(x,y): x,y\in \mathbb{L}\}$ be the collection of unordered pairs on $\mathbb{L}$,
$\{w_{x,y}: (x,y) \in E\}$ be a sequence of non-negative independent random
variables, and $(X_t^{\w})_{t\ge 0}$ be the Markov process with infinitesimal generator
$L_{\mathbb{L}}^{\w}$ defined by \eqref{eq:geneoe}.
Obviously $(X_t^{\w})_{t\ge 0}$  is the symmetric Hunt process associated with the Dirichlet form
$(D_{V_1}^{\w},\F^{\w}_1)$ with $V_1=\mathbb{L}$ and $w^{(1)}_{x,y}(\w)=w_{x,y}(\w)$.

\begin{proposition}\label{ex3-1} Let $d:=d_1+d_2>4-2\alpha$.
Suppose that $\{w_{x,y}: (x,y)\in E\}$ is a sequence of non-negative independent
random variables satisfying that
$$\sup_{x,y \in \mathbb{L},x\neq y}\Pp\big(w_{x,y}=0\big)<1/2$$
and
\begin{equation}\label{ex3-1-1}
\sup_{x,y\in \mathbb{L}}\Ee[w_{x,y}^{2p}]<\infty~~~\mbox{and}~~~\sup_{x,y\in
\mathbb{L}}\Ee[w_{x,y}^{-2q}\I_{\{w_{x,y}>0\}}]<\infty
\end{equation}
with
$p>\max\big\{{{(d+2)}/{d}, (d+1)}/(2(2-\alpha))\big\}$ and $q>{(d+2)}/{d}.$ If moreover
\eqref{p3-2-0} holds true, then the quenched invariance principle
holds for
 $X^{\w}_{\cdot}$ with the limit process $Y$. Moreover, when $\alpha\in (0,1)$, the conclusion still holds true for $d>2-2\alpha$, if $p>\max\big\{{(d+1)}/(2(1-\alpha)), {(d+2)}/{d}\big\}$ and $q>{(d+2)}/{d}.$
\end{proposition}
\begin{proof}
According to Theorem \ref{t5-1}, it suffices to verify
\eqref{p3-2-1} --- \eqref{l3-1-1-1}.
We first verify \eqref{l3-1-2-a}.
Suppose that $p_0:=\sup_{x,y \in \mathbb{L},x\neq y}\Pp\big(w_{x,y}=0\big)<1/2$. Denote by
$L(x,r):=|\{y\in \mathbb{L}: |y-x|\le r\}|$.
Let $\{\eta_y\}_{\{y\in \mathbb{L}:|y-x|\le r\}}$ be a sequence of i.i.d.\ Bernoulli random variables (possibly on the extended probability space) such that $\Pp(\eta_y=1)=1-p_0$, $\Pp(\eta_y=0)=p_0$ and $\eta_y\le \I_{\{w_{z,y}\neq 0\}}$ a.s..
Take $1/2<c_0<1-p_0$, then for every $r>0$ and $x,z\in \mathbb{L}$,
\begin{align*}
\Pp\Big(\sum_{y\in \mathbb{L}:|y-x|\le r}
\I_{\{w_{z,y}\neq 0\}}\le c_0L(x,r)\Big)
&\le \Pp\Big(\sum_{y\in \mathbb{L}:|y-x|\le r}
\eta_y\le c_0L(x,r)\Big)\\
&=\Pp\left(\frac{\sum_{y\in \mathbb{L}:|y-x|\le r} (\eta_y-\Ee[\eta_y])}{L(x,r)}\le c_0-(1-p_0)\right)\\
&\le c'e^{-c''r^d},
\end{align*}
where in the first inequality we used the fact that $\sum_{y\in \mathbb{L}:|y-x|\le r}\eta_y\le \sum_{y\in \mathbb{L}:|y-x|\le r}\I_{\{w_{z,y}\neq 0\}}$ a.s., and the last inequality follows from the Cramer theorem for i.i.d.\ Bernoulli random variables and the fact $L(x,r)\asymp r^d$.
The estimate above yields that
$$
\sum_{R=1}^{\infty}\sum_{x,z\in
B_F(0,6R)\cap V_1}\sum_{r=R^{\theta}/2}^{2R}
p_6(x,z,r,c_0)<\infty.$$
This is, \eqref{l3-1-2-a} holds with $c_0$ chosen above.

Recall that, for a sequence of mutually independent random variables $\{\eta_n\}_{n\ge 1}$ satisfying
$\Ee\left[\eta_n\right]=0$ for every $n\in \mathbb{N}_+$, $M_n:=\sum_{k=1}^n \eta_k$ is a martingale with respect to the natural filtration. Then, by the
Burkholder-Gundy-Davis inequality, for any $l\ge 2$
and $n\in \mathbb{N}_+$,
\begin{equation}\label{extra1}
\begin{split}
& \Ee\left[\left|\sum_{k=1}^n \eta_k\right|^l\right]=\Ee\left[|M_n|^l\right]\le
c_1\Ee\left[\langle M\rangle_n^{l/2}\right]\le c_2n^{l/2-1}
\sum_{k=1}^n\Ee\left[|\eta_k|^l\right],
\end{split}
\end{equation}
where $\langle M \rangle_n=\sum_{k=1}^n \eta_k^2$ is the variational
process associated with $M_n$, and $c_1,c_2>0$ are independent of $n$.

Hence, by \eqref{extra1},  we know that for every $\varepsilon_0>0$, $R$, $r>0$, $c_0^*\ge2$,
$n\ge 1$ and a subsequence of bounded measurable functions $\{h_n\}_{n\ge1}$ on $\mathbb{L}\times \mathbb{L}$ such that $\sup_{n\ge1}\|h_n\|_\infty<\infty$,
\begin{align*}
p_1(r,R,\varepsilon_0)&\le\! \varepsilon_0^{-2p}R^{-2pd}r^{-2pd}
\Ee\Big[\Big|\sum_{x,y\in \mathbb{L}: |x|\le R, |y-x|\le r}\!\!(w_{x,y}\!-\!\Ee[w_{x,y}])\Big|^{2p}\Big]\!\le\! c_3(\varepsilon_0)r^{-pd}R^{-pd}, \\
p_2(x,r,\varepsilon_0)&\le \varepsilon_0^{-2p}r^{-2pd}\Ee
\Big[\Big|\sum_{y\in \mathbb{L}:|y-x|\le r}
\big(w_{x,y}-\Ee[w_{x,y}]\big)\Big|^{2p}\Big]\le c_4(\varepsilon_0)r^{-pd},\\
p_4(x,r,c_0^*,\varepsilon_0)&\le \varepsilon_0^{-2q}r^{-2qd}
\Ee\Big[\Big|\sum_{y\in \mathbb{L}: |y-x|\le c_0^*r}\big(w_{x,y}^{-1}-\Ee[w_{x,y}^{-1}]\big)\Big|^{2q}\Big]\le c_5(\varepsilon_0,c_0^*)r^{-qd},\\
p_5^{(n)}(x,N,\varepsilon,h_n,\varepsilon_0)&\le
c_6(\varepsilon_0)n^{2\alpha p} \cdot(n\varepsilon)^{-2p(d+\alpha)}\\
 &\quad \times
\Ee\left[\left|\sum_{y\in \mathbb{L}: |y-x|\ge n\varepsilon, |y|\le nN}h_n(x,y) (n\varepsilon)^{d+\alpha}\frac{(w_{x,y}-\Ee[w_{x,y}])}{|x-y|^{d+\alpha}}\right|^{2p}\right]\\
&\le c_7(\varepsilon_0, N,\varepsilon,\sup_{n\ge1}\|h_n\|_\infty)n^{2\alpha p}n^{pd}n^{-2p(d+\alpha)}=c_7(\varepsilon_0,N,\varepsilon,\sup_{n\ge1}\|h_n\|_\infty)n^{-pd}.
\end{align*}

In the following, we fix $x\in \mathbb{L}$. Set $\xi(y):=
\frac{(w_{x,y}-\Ee[w_{x,y}])}{|x-y|^{d+\alpha-2}}$ for every $y\in \mathbb{L}$ with
$y\neq x$.  Clearly,
$\{\xi(y)\}_{y\in \mathbb{L}: y\neq x}$ are mutually independent. By \eqref{ex3-1-1},
$\Ee[\xi(y)]=0$ and $\Ee[|\xi(y)|^{2p}]\le c_8|x-y|^{-2p(d+\alpha-2)}$.
Choosing $0<\delta<\frac{d+2\alpha-4}{d}$ (thanks to $d+2\alpha-4>0$) and applying the first inequality in \eqref{extra1}, we arrive at that for every $r\ge 1$
\begin{align*}
&\Ee\left[\left|\sum_{y\in \mathbb{L}: |y-x|\le r}\frac{\left(w_{x,y}-\Ee[w_{x,y}]\right)}{|x-y|^{d+\alpha-2}}\right|^{2p}\right]\\
&=\Ee\left[\left|\sum_{y\in \mathbb{L}: |y-x|\le r} \xi(y)\right|^{2p}\right]\le
c_{9}\Ee\left[\left|\sum_{y\in \mathbb{L}: |y-x|\le r}|\xi(y)|^{2}\right|^{p}\right]\\
&\le c_9\Ee\left[\left(\sum_{y\in \mathbb{L}: |y-x|\le r}|\xi(y)|^{2p}|x-y|^{d(p+\delta-1)}\right)\cdot
\left(\sum_{y\in \mathbb{L}: |y-x|\le r}|x-y|^{-\frac{d(p+\delta-1)}{p-1}}\right)\right]\le c_{10}.
\end{align*}
Here the second inequality above follows from the H\"older inequality,  in the last inequality we used
$\Ee[|\xi(y)|^{2p}]\le c_8|x-y|^{-2p(d+\alpha-2)}$, and $c_9, c_{10}$ are independent of $r$.
Then, by the Markov inequality, we know that for every $x \in \mathbb{L}$,
$p_3(x,R,\varepsilon_0)\le c_{11}(\varepsilon_0)R^{-2(2-\alpha)p}$.

Under assumptions of the proposition, we can choose $\theta\in
(0,1)$ (close to $1$) such that
$$p>\max\left\{\frac{d+1+\theta}{d\theta },\frac{d+1}{2\theta(2-\alpha)}\right\}\,\, \text{ and    }\,\, q>\frac{d+1+\theta}{d\theta},$$ also thanks to the condition that $d>4-2\alpha$ again. Then, according to all the estimates above, we know
immediately that
 \eqref{p3-2-1} --- \eqref{l3-1-2} hold for this $\theta\in (0,1)$ and
 every sufficiently small $\varepsilon_0>0$.

Suppose that $\alpha\in (0,1)$. If $d>2-2\alpha$, $p>\max\big\{{(d+1)}/(2(1-\alpha)), {(d+2)}/{d}\big\}$ and $q>{(d+2)}/{d}$, then we can choose $\theta\in
(0,1)$ (close to $1$) such that
$$p>\max\left\{\frac{d+1+\theta}{d\theta },\frac{d+1}{2\theta(1-\alpha)}\right\} \,\, \text{ and    }\,\, q>\frac{d+1+\theta}{d\theta}.$$ Following the argument above, we can prove that \eqref{p3-2-1} --- \eqref{p3-2-2}, \eqref{l3-1-2} and \eqref{l3-1-1-1} are satisfied.  Then, the desired assertion follows from Theorem \ref{t5-1} again. The proof is complete.
\end{proof}

Theorem \ref{th1} is a direct consequence of Proposition
\ref{ex3-1}, since \eqref{p3-2-0} holds trivially in this setting.
We note that the idea of the proof above using the
Burkholder-Gundy-Davis inequality comes from the proof of \cite[Theorem 1.1]{CCKW}.

\subsubsection{Time-change of $\alpha$-stable-like process on $\R^d$} We
can
consider the case that the approximating measure $m_n$ is not $n^{-d}\mu_n$.
Let us first fix the triple $(F,\rho,m)$ with $F=\R^d$,
$\rho$ being the Euclidean distance and $m(dx)=K(x)\,dx$, where $dx$ denotes the Lebesgue measure on $\R^d$ and
$K$ is a continuous function on $\R^d$ satisfying that $0<C_1\le K(x)\le C_2<\infty$ for some constants
$C_1\le C_2$.
Then, the process $Y$ associated with the Dirichlet form
$(D_0,\F_0)$ given at the beginning of
Subsection \ref{subsu5-1} is a time-change of symmetric $\alpha$-stable process on $\R^d$ with $c(x,y)=K(x)^{-1}K(y)^{-1}$ for $x,y\in \R^d$. It is obvious that $(D_0,\F_0)$
satisfies assumption {\bf(Dir.)}.

Similar to the previous part, we can take $V_1=\Z^d$, and $m_n=n^{-d}K_n\mu_n$
with $\mu_n$ being the counting measure on $n^{-1}\Z^d$ and
\begin{equation*}
K_n(x)=n^{-d}\int_{U_n(x)}K(x)\,dx,\quad x\in n^{-1}\Z^d,
\end{equation*}
where $U_n(x)=\Pi_{i=1}^d[x_i,x_i+n^{-1})$ for any $x=(x_1,\cdots,x_d)\in n^{-1}\Z^d$.

Suppose that $E:=\{(x,y): x,y\in \mathbb{L}\}$ is the collection of unordered pairs on $\mathbb{L}$,
and $\{w_{x,y}:
(x,y) \in E\}$ is a sequence of non-negative independent random
variables. Let $(X_t^{\w,n})_{t\ge 0}$ and $(Y_t^{\w,n})_{t\ge 0}$ be symmetric Hunt processes associated with the Dirichlet form
$(D_{V_n}^{\w},\F^{\w}_n)$ on $L^2(V_n;m_n)$ and $L^2(V_n;n^{-d}\mu_n)$ respectively, where
\begin{align*}
D_{V_n}^{\w}(f,f)&=\sum_{x,y\in V_n} \big(f(x)-f(y)\big)^2 \frac{w_{nx,ny}}{|x-y|^{d+\alpha}},\\
\F^{\w}_n&=\{f\in L^2(V_n;\mu_n): D_{V_n}^{\w}(f,f)<\infty\}.
\end{align*}
Let $(X_{t}^{\w})_{t\ge 0}:= (X_{t}^{\w,1})_{t\ge 0}$. It is easy to verify that
$\frac{1}{n}X_{n^\alpha t}^{\w}=Y_{A_{n,t}^{-1}}^{\w,n}$ with
$A_{n,t}:=\int_0^t K_n\big(Y_s^{\w,n}\big)\,ds$ for any $t>0$.

Note that for any compact set $S \subset \R^d$,
$
\lim_{n \rightarrow \infty}\sup_{x\in S}\big|K_n([x]_n)-K(x)\big|=0.
$
Following the same arguments in the proof of
Proposition \ref{ex3-1},  we can obtain that under assumption \eqref{ex3-1-1} the quenched invariance
principle holds for $(X_t^{\w})_{t\ge 0}$ with limiting process $Y$ being a time-change of symmetric $\alpha$-stable process on $\R^d$.

\begin{remark}\label{r5-2}
From the example above, we know that to identity the limit process consists of two ingredients. One
is to verify locally weak convergence of $m_n$ to $m$, and the other is to justify convergence of the jumping kernel
for the associated Dirichlet form. In fact, by carefully tracking the proof above, we can see that if the measure $m_n$ is replaced
by a more general (random) measure which converges locally weakly to $m$,
then the quenched invariance principle still holds
with the same limiting process.
\end{remark}

\subsubsection{Bounded Lipschitz domain}\label{bddLipd}
In fact, Proposition \ref{ex3-1} holds not only for a half/quarter space, but
also for the closure of a bounded Lipschitz domain in
$\mathbb{R}^d$,
whose intrinsic distance is equivalent to the Euclidean distance and whose
volume growth is with order $d$.
In details, let $F\subset \mathbb{R}^d$ be
a closed set such that for any $x,y\in F$ and $r>0$,
$
c_1r^{d}\le m(B_F(x,r))\le c_2r^{d}$ and
$
c_1|x-y|\le \rho_F(x,y)\le c_2|x-y|,
$
where $$\rho_F(x,y):=\inf\left\{\int_{0}^1 |\dot{\gamma}(s)|\,ds: \gamma\in C^1([0,1];F), \gamma(0)=x,\gamma(1)=y\right\}$$ is the
intrinsic distance on $F$, $m$ is the Lebesgue measure,
 and $B_F(x,r)$ is the ball with respect to $\rho_F$. For example, these properties are satisfied when $F$ is an inner uniform domain; see \cite[Chapter 2.3]{GS}.
For $x=(x_1,\cdots,x_d)\in n^{-1}\Z^d$,
set $U_n(x)=\Pi_{i=1}^d[x_i,x_i+n^{-1}).$
Note that when $F$ is the closure of a {bounded} Lipschitz domain, $V_n:=\{n^{-1}\mathbb{Z}^d\cap F: U_n(x)\subset F\}$
satisfies the properties given in Lemma \ref{L:ac}.
Suppose that $\{w_{x,,y}:(x,y)\in E\}$ is a sequence of independent
random variables satisfying the conditions in Proposition \ref{ex3-1}.
Then the conclusion of Proposition \ref{ex3-1} holds on $F$. Indeed, in this case, by taking
$V_n$ as above, the proofs of Theorem \ref{t5-1} and Proposition \ref{ex3-1} go through without any change
({with $\rho$ replaced by $\rho_F$ as explained in Remark \ref{r5-1}}). Note that neither $V_n=n^{-1}V_1$ nor
$X_t^{(n),\w}={n}^{-1}X_{n^{\alpha}t}^{V_1,\w}$
holds in general in this setting.
{(However, we can verify that $X_t^{(n),\w}={n}^{-1}X_{n^{\alpha}t}^{\tilde V_{n},\w}$, where
$\tilde V_n:=nV_n \subset nF.$)}
Note that the proofs do not require these properties, and the integrability condition
given for all $x,y\in \mathbb{Z}^d$ is (more than) enough for the estimates in the
proofs to hold.

\subsubsection{Fractal graph} The arguments in Example \ref{laths}
work for more general graphs
that satisfy (i)--(iv), and that its scaling limit $(F,\rho,m)$ and Dirichlet form which satisfy {\bf(MMS)} and {\bf(Dir.)} respectively as discussed at the beginning of subsection \ref{subsu5-1}. In particular, we can
prove quenched invariance principle for stable-like processes on various fractal graphs.

Here we introduce the most typical fractal graph; namely the Sierpinski gasket graph.
Let $e_0=(0,0,\cdots, 0)\in \R^{N}$, and for
$1\le i\le N$, $e_i$ be the unit vector in $\R^{N}$ whose $i$-th element is $1$. Set
$F_i(x)=(x-e_i)/2+e_i$ for $0\le i\le N$.
Then, there exists the unique non-void compact set such that
$K=\cup_{i=0}^N F_i(K)$; $K$ is called the $N$-dimensional
Sierpinski gasket. Set $F:=\cup_{n=0}^{\infty}2^nK$, which is the unbounded Sierpinski gasket. Let
\[V_1=\bigcup_{m=0}^\infty 2^m\Big(\bigcup_{i_1,\cdots,i_m=0}^NF_{i_1}\circ\cdots
\circ F_{i_m}(\{e_0,\cdots, e_N\})\Big),~~V_n=2^{-n+1}V_1.\]
(Hence, $n^{-1}$ in the definition of $V_n$ in the previous subsection is now
$2^{-n+1}$.) The closure of $\cup_{m\ge 1}V_m$ is $F$.
$F$ satisfies assumption {\bf(MMS)} with $d=\log (N+1)/\log 2$. We can naturally construct
a regular stable-like Dirichlet form satisfying
assumption {\bf(Dir.)}.
Let $\{w_{x,y}: x,y\in V_1\}$ be a sequence of independent random
variables. Then we have Proposition \ref{ex3-1}
with the same proof in this case as well.

\subsubsection{Constant speed $\alpha$-stable-like random walk on $\Z^d$}\label{constSRW}
We have considered in Theorem \ref{th1} the quenched invariance principle for variable speed $\alpha$-stable-like random walks. In this part, we will show
the quenched invariance principle for a class of constant speed $\alpha$-stable-like random walks on $\Z^d$, by using Theorem \ref{th1} and the time change argument.

Let $\mathbb{L}=\Z^d$ with $d>4-2\alpha$, and $E$ be the set of all unordered pairs on $\Z^d$.
Suppose that $\{w_{x,y}(\w): (x,y)\in E\}$ is the set of mutually independent non-negative random variables
such that $\Ee[w_{x,y}]=1$, \eqref{eq: prob} and \eqref{eq:fhibw} hold.
Suppose for simplicity that they have
the same distribution $\lambda$ on $[0,\infty)$.
In the following, without loss of generality we take $\Omega=[0,\infty)^{E}$, $\Pp=\prod_{(x,y)\in E}\lambda$ and
$w_{x,y}(\w)=\w(x,y)$ for any $\w\in \Omega$. For every $z\in \Z^d$, define $\tau_z:\Omega \rightarrow \Omega$ by
$\tau_z\w(x,y)=\w(x+z,y+z)$. It is standard to verify that $\Pp$ is stationary and ergodic with respect to the translations $\{\tau_z\}_{z\in \Z^d}$ of $\Z^d$, see e.g.\ the proof of \cite[Theorem 3.2]{BeB}.

Let $(X_t^{\w})_{t\ge 0}$ be the $\Z^d$-valued
process as in Theorem \ref{th1}, which is a variable speed $\alpha$-stable-like random walk. Now we define a
constant speed $\alpha$-stable-like random walk $(Z_t^\w)_{t\ge 0}$ via the time change as follows
\begin{equation*}
Z_t^\w:=X_{a_t(\w)}^\w,
\end{equation*}
where $a_t(\w):=\inf\{s\ge 0: A_s(\w)\ge t\}$, $A_t(\w):=\int_0^t \nu_{X_s^\w}(\w)\, ds$ and
$\nu_x(\w):=\sum_{y\in \Z^d: y\neq x}\frac{\w(x,y)}{|x-y|^{d+\alpha}}$ for every $x\in \Z^d$.
Letting $Z_t^{(n),\w}:=n^{-1}Z_{n^{\alpha}t}^\w$ and $X_t^{(n),\w}:=n^{-1}X_{n^{\alpha}t}^\w$,
we have $Z_t^{(n),\w}=X^{(n),\w}_{a_t^{(n)}(\w)}$, where $a_t^{(n)}(\w):=n^{-\alpha}a_{n^\alpha t}(\w)$.

According to the argument of \cite[Lemma 2.4]{ADS1}, we can prove that $\Pp$ is stationary,
reversible and ergodic with respect to the environment process $(\tau_{X_t^\w}\w)_{t\ge 0}$. Therefore,
by the ergodic theorem, we have
\begin{align*}
\lim_{t \rightarrow \infty} \frac{A_t(\w)}{t}=\lim_{t \rightarrow \infty} \frac{\int_0^t \nu_0(\tau_{X_s^\w}\w)\,ds}{t}=
\Ee[\nu_0(\w)]=\sum_{y\in \Z^d:y\neq 0}\frac{1}{|y|^{d+\alpha}}=:C_{d,\alpha},\,\,\,\, a.s..
\end{align*}
Combining this with the strictly increasing property of $t\mapsto a_t(\w)$ (which is due to the fact that  $\nu_x(\w)>0$), we prove that for every $T>0$,
\begin{equation}\label{p5-2-1}
\begin{split}
&\lim_{n \rightarrow \infty}\sup_{t\in [0,T]}|a_t^{(n)}(\w)-C_{d,\alpha}^{-1}t|\\
&=\lim_{n \rightarrow \infty}
\sup_{t\in [0,T]}|n^{-\alpha}a_{n^\alpha t}(\w)-C_{d,\alpha}^{-1}t|=0\\
&\le \lim_{n \rightarrow \infty}\sup_{t\in [n^{-\delta},T]}|n^{-\alpha}a_{n^\alpha t}(\w)-C_{d,\alpha}^{-1}t|
+\lim_{n \rightarrow \infty}\sup_{t\in [0,n^{-\delta}]}(n^{-\alpha}a_{n^\alpha t}(\w)+C_{d,\alpha}^{-1}t)\\
&\le \lim_{n \rightarrow \infty}\sup_{t\in [n^{-\delta},T]}|n^{-\alpha}a_{n^\alpha t}(\w)-C_{d,\alpha}^{-1}t|
+\lim_{n \rightarrow \infty}(n^{-\alpha}a_{n^{\alpha-\delta}}(\w)+C_{d,\alpha}^{-1}n^{-\delta})=0,\ a.s.,
\end{split}
\end{equation}
where $\delta$ is any fixed positive constant such that $\delta<\alpha$.

According to Theorem \ref{th1}, for any $T>0$ and a.s. $\w\in \Omega$, $X_{\cdot}^{(n),\w}$ converges under the Skorohod topology on
$\mathscr{D}([0,T];\R^d)$
to a symmetric $\alpha$-stable L\'evy process $Y_{\cdot}$ on $\R^d$ with jumping measure $|z|^{-d-\alpha}\,dz$.
Then, by the fact that $Z_t^{(n),\w}=X_{a^{(n)}_t(\w)}^{(n),\w}$, \eqref{p5-2-1} and the definition of the Skorohod topology, we can obtain that, for any $T>0$ and a.s. $\w\in \Omega$,
$Z_{\cdot}^{(n),\w}$ converges under the Skorohod topology on
$\mathscr{D}([0,T];\R^d)$
to the process $Y_{C_{d,\alpha}^{-1}\cdot}$.

\ \

\noindent {\bf Acknowledgements.}  The authors are very grateful to referees for their helpful suggestions of corrections. The research of Xin Chen is
supported by National Natural Science Foundation of China (No.\ 11501361 and No.\ 11871338).\ The research of Takashi Kumagai is supported by the
Grant-in-Aid for Scientific Research (A) 25247007 and 17H01093,
Japan.\ The research is
supported by the National Natural Science Foundation of China (No. 11831014), the Program for Probability and Statistics: Theory and Application (No.\ IRTL1704) and the Program for Innovative Research Team in Science and Technology in Fujian Province University (IRTSTFJ).

\end{document}